\documentclass[10pt]{article}

\usepackage{amssymb, amsthm, mathtools}
\usepackage[hidelinks, citecolor=blue]{hyperref}
\hypersetup{colorlinks=true, urlcolor = blue, linkcolor=black}
\usepackage{stmaryrd}
\usepackage{verbatim}
\usepackage{bm}
\usepackage{color}

\newcommand{\nc}{\normalcolor}

\numberwithin{equation}{section}

\newtheorem{theorem}{Theorem}
\newtheorem{corollary}[theorem]{Corollary}
\newtheorem{lemma}[theorem]{Lemma}
\newtheorem{prop}[theorem]{Proposition}
\newtheorem{defn}[theorem]{Definition}

\newcommand{\zz}{{\mathbb{Z}}}

\newcommand{\pp}{{\mathbb{P}}}
\newcommand{\oh}{\mathcal{O}}

\newcommand{\tr}{{\operatorname{Tr}\,}}

\newcommand{\beq}[1]{\begin{equation} \label{#1}}
\newcommand{\eeq}{\end{equation}}

\newcommand{\expb}[1]{\operatorname{exp}\left\{#1\right\}}
\newcommand{\ntr}[1]{\left\langle #1 \right\rangle}
\usepackage[top=1.25in, bottom=1.25in, left=1.25in, right=1.25in]{geometry}
\newcommand*\samethanks[1][\value{footnote}]{\footnotemark[#1]}
\renewcommand{\nc}{\newcommand}
\nc{\rnc}{\renewcommand}
\nc{\R}{\mathbb{R}}
\nc{\C}{\mathbb{C}}
\rnc{\d}{\mathrm{d}}
\nc{\E}{\mathbb{E}}

\begin{document}
\title{Gaussian statistics for left and right eigenvectors of complex non-Hermitian matrices}

\author{Sofiia Dubova\thanks{Department of Mathematics, Harvard University} \and Kevin Yang\samethanks \and Horng-Tzer Yau\samethanks \and Jun Yin\thanks{Department of Mathematics, University of California, Los Angeles}}
\maketitle

\begin{abstract}
We consider a constant-size subset of left and right eigenvectors of an $N\times N$ i.i.d. complex non-Hermitian matrix associated with the eigenvalues with pairwise distances at least $N^{-\frac12+\epsilon}$. We show that arbitrary constant rank projections of these eigenvectors are Gaussian and jointly independent.
\end{abstract}

\tableofcontents

\section{Introduction}
Eigenvector statistics of random matrix ensembles have been extensively studied in random matrix theory. In the case of Hermitian random matrices, which can be viewed as Hamiltonians of disordered quantum systems, the Quantum Unique Ergodicity (QUE) conjecture \cite{RS94} asserts that their eigenvectors tend to be distributed uniformly on a sphere. For Gaussian orthogonal and Gaussian unitary ensembles this statement is trivially correct due to their invariance under multiplication by orthogonal and unitary matrices respectively. A great deal of work has been done to show that the behavior of the eigenvectors of (generalized) Wigner ensemble is consistent with QUE: delocalization of eigenvectors \cite{EYY12}, asymptotic normality and independence of finitely many deterministic projections of eigenvectors \cite{bourgade2017eigenvector,marcinek2022high}, the size and normality of other eigenvector statistics \cite{CES2020,cipolloni2022normal,benigni2022fluctuations,benigni2023fluctuations}. Similar results have been established for other Hermitian ensembles in \cite{bourgade2018random,bourgade2017sparse,aggarwal2021eigenvector}. 

In the case of non-Hermitian i.i.d. ensembles, less is known. The delocalization of the eigenvectors has been established in \cite{rudelson2015delocalization,rudelson2016no,alt2018local,alt2021spectral,luh2020eigenvector,lytova2020delocalization}. Some work has been done on the size and distribution of the overlaps $\mathcal{O}_{ij} = \langle u_j, u_i\rangle\langle v_i, v_j\rangle$, where $u_i, u_j$ and $v_i, v_j$ are the right and left eigenvectors of a non-Hermitian i.i.d. matrix $A$, see e.g. \cite{BD20,cipolloni2023optimal,EJ23}. Overlaps $\mathcal{O}_{ij}$ are of particular interest for the non-Hermitian models as they govern \cite{BD20,grela2018full} the evolution of eigenvalues and eigenvectors of $A_t$ under the flow $\d A_t = \frac{1}{\sqrt{N}}\d B_t$, where $B_t$ is a Brownian motion. In this work we obtain the Gaussianity of the deterministic projections of eigenvectors of $A$.

The standard process for proving Gaussianity of the eigenvector statistics for a Hermitian ensemble $H$ is the well-known three-step strategy, initially used for the local eigenvalue statistics in \cite{erdHos2010bulk}:
\begin{enumerate}
    \item Local laws for the resolvent of $H$;
    \item Gaussianity of the eigenvector statistics of a Gaussian divisible ensemble $H_t = H + \sqrt{t}B$, where $B$ is GOE;
    \item Comparison of the eigenvector statistics of the initial ensemble $H$ and $H_t$.
\end{enumerate}
The key component of this argument is the second step, which is done using the Eigenvector Moment Flow (EMF), first developed in \cite{bourgade2017eigenvector}. EMF is derived from the evolution of eigenvectors under the flow $\d H_t = \frac{1}{\sqrt{N}}\d B_t$, where $B_t$ is a standard Hermitian Brownian motion. The non-Hermitian analogue of this eigenvector evolution is a lot more complicated, see e.g. \cite{BD20,grela2018full}. 

In this paper we take a different approach to the second step of the three-step strategy. Our method is inspired by the supersymmetric approach recently used to prove the universality of local eigenvalue statistics of complex i.i.d. non-Hermitian matrices in the bulk in \cite{MO} and extended to the real case in \cite{dubova2024bulk,M2}. This method allows us to compute the moment generating functions of the eigenvector statistics of the Gaussian divisible ensemble $A+\sqrt{t}B$ directly through asymptotic analysis of the exact integral formulas for these functions at the time-scale $t=N^{-\frac13+\epsilon_0}$. 

Some of the details behind the third step for non-Hermitian matrices are also different from what is done the Hermitian case \cite{KY13}. As in \cite{MO}, this comparison is done through Girko's formula, and this requires translating eigenvector statistics into eigenvalue statistics in some sense. More is explained towards the end of the introduction.

Now we state the main result of this paper.
\begin{theorem}\label{theorem:main}
Let $A$ be an $N\times N$ complex random matrix with i.i.d. entries satisfying $\E A_{ij} = 0$, $\E\left|A_{ij}\right|^2 = N^{-1}$ and $\E\left|A_{ij}\right|^p \le C_p N^{-p/2}$. Fix $\epsilon,\tau>0$, $m_R, m_L\in \mathbb{Z}_+$, set $m=m_R+m_L$. Consider deterministic points in the complex plane $\left\{z^0_j = z^0_j(N)\right\}_{j=1}^{m}\subset\C$ such that $\left|z^0_j\right|<1-\tau$ for any $1\le j\le m$ and $\left|z^0_j-z^0_k\right|\ge N^{-1/2+\epsilon}$ for any distinct $1\le j,k\le m$. Let $(\lambda_j, u_j)$ denote an eigenvalue and the corresponding right eigenvector of $A$ for $1\le j\le m_R$ and $(\lambda_j, v_j)$ denote an eigenvalue and the corresponding left eigenvector of $A$ for $m_R+1\le j\le m$. Let $\left\{T_j = T_j(N)\right\}_{j=1}^{m}$ be a family of deterministic bounded rank matrices with bounded norm. Then for any test function $F\in C_c^\infty\left(\C^{m}\times \R^{m}\right)$, we have
\begin{align}
&\lim_{N\rightarrow\infty} \E F\left(\sqrt{N}(\lambda_1 - z^0_1),\ldots,\sqrt{N}(\lambda_m - z^0_{m}), \|T_1 u_1\|^2,\ldots, \|T_m v_m\|^2\right)\\
= &\lim_{N\rightarrow\infty} \E F\left(\sqrt{N}(\lambda_1 - z^0_1),\ldots,\sqrt{N}(\lambda_m - z^0_{m}), Z_1^2,\ldots, Z_m^2\right),
\end{align}
where $Z= (Z_1, \ldots, Z_N)$ consists of independent Gaussian random variables $Z_j \sim \mathcal{N}(0,\|T_j\|_F^2)$ which are independent of $A$. 
\end{theorem}
Theorem \ref{theorem:main} implies convergence in distribution of $\|T_{1}u_{1}\|^{2},\ldots,\|T_{m}v_{m}\|^{2}$ to squares of independent Gaussians. The point of including $\sqrt{N}(\lambda_{1}-z_{1}^{0}),\ldots,\sqrt{N}(\lambda_{m}-z_{m}^{0})$ is to pin down the corresponding eigenvalues near $z_{1}^{0},\ldots,z_{m}^{0}$. Based on the universality result for local eigenvalue statistics in Theorem 1.2 of \cite{MO} (and its proof), we anticipate that
\begin{align*}
&\lim_{N\rightarrow\infty} \E F\left(\sqrt{N}(\lambda_1 - z^0_1),\ldots,\sqrt{N}(\lambda_m - z^0_{m}), Z_1^2,\ldots, Z_m^2\right)\\
&=\int_{\C^{m}}\E F(w_{1},\ldots,w_{m},Z_{1}^{2},\ldots,Z_{m}^{2})\rho_{\mathrm{GinUE}}^{(m)}(w_{1},\ldots,w_{m})\d w_{1}\ldots\d w_{m},
\end{align*}
where the expectation in the second line is with respect to $Z_{1},\ldots,Z_{m}$, and
\begin{align*}
\rho_{\mathrm{GinUE}}^{(m)}(w_{1},\ldots,w_{m})=\det\left[\frac{1}{\pi}e^{-\frac12(|w_{j}|^{2}+|w_{\ell}|^{2})+w_{j}\overline{w}_{\ell}}\right]_{j,\ell=1}^{m}.
\end{align*}
(We can assume that $z_{j}^{0}$ converge as $N\to\infty$ by taking a subsequence and observe that the limiting $m$-point correlation function $\rho_{\mathrm{GinUE}}^{(m)}$ in Theorem 1.2 of \cite{MO} is independent of said limits of $z_{j}^{0}$.) This should follow exactly from the proof of Theorem 1.2 in \cite{MO}, but we have stated Theorem \ref{theorem:main} in the above way since this argument is not written down.

As mentioned above, the proof of Theorem \ref{theorem:main} follows the three-step strategy. The first step is carried out in \cite{alt2018local,cipolloni2023mesoscopic,cipolloni2021fluctuation,cipolloni2023optimal}, where the authors establish the local laws for the resolvent of the Hermitization of $A$:
\[
G_z(\eta) = \begin{pmatrix}-i\eta&A-z\\A^\ast-\bar{z}&-i\eta\end{pmatrix}^{-1}.
\]
In particular, the results we use in our proof are the averaged and isotropic local laws as well as the two-resolvent averaged local laws with possibly different shifts $z_1$ and $z_2$. The second step is covered in Sections \ref{sec:gauss-div}-\ref{sec:concentration}. Here we adapt the method of \cite{MO} to access left and right eigenvectors. The first technical aspect of this step is analyzing the resolvents $[(A-z)^{\ast}(A-z)+\eta^{2}+Y]^{-1}$ and $[(A-z)(A-z)^{\ast}+\eta^{2}+Y]^{-1}$, where $Y$ is finite rank, Hermitian, and possibly negative. The method of \cite{MO} depends crucially on various positivity properties and estimates for the resolvents when $Y=0$. We must ensure these properties are stable under perturbation by $Y$; this is why we require isotropic local laws in addition to the estimates used in \cite{MO}. Moreover, when studying multiple eigenvalues, we must also analyze these resolvents after projecting to the orthogonal complement of the span of finitely many eigenvectors. Isotropic local laws for deterministic vectors alone no longer suffice. We must also remove the projection using various perturbations and the concentration estimates in Section 6 of \cite{MO} (this is the content of Section \ref{sec:concentration}).

The third step, the comparison between the eigenvector statistics of a general i.i.d. matrix and a Gaussian divisible matrix, is done in Sections \ref{sec:comparison} and \ref{sec:green-function}. This argument follows the same framework as the eigenvector comparison argument of \cite{KY13} for the Hermitian matrices. The key novelty of our proof is the approximation of the eigenvector statistics with certain functions of $G^z$ in Lemma \ref{lemma:Vformula} and Girko's formula \cite{girko1984}. This approximation relies on the level repulsion estimate for the singular values of $A-z$ in Proposition \ref{prop:levelrepulsion} that we derive using Theorem 2.10 of \cite{EJ23} and Theorem 3.2 of \cite{CL19}. The rest of the argument is a standard Green function comparison, see e.g. \cite{KY16}.

\subsection{Notation}
We denote the standard basis vectors in $\C^m$ by $\mathbf{e}_{j,m}$ for $1\le j\le m$. Whenever the dimension of the space is clear from the context we omit it and write $\mathbf{e}_j = \mathbf{e}_{j,m}$. We write $\llbracket a, b\rrbracket = \mathbb{Z} \cap [a,b]$. Given a matrix $X$, normalized trace of $X$ is $\ntr{X} = N^{-1}\tr X$.

\subsection{Acknowledgements}
K.Y. is supported in part by NSF Grant No. DMS-2203075. H.-T. Y. is supported in part by NSF grant DMS-2153335. J. Y. is supported in part by the Simons Fellows in Mathematics.

\section{Gaussian divisible matrices}\label{sec:gauss-div}
In this section we state the main results for Gaussian divisible matrices.
For any $A\in M_N(\C)$ we define its Hermitization by
\[
\mathcal{H}_z = \begin{pmatrix}0& A-z\\A^\ast-\bar{z}&0\end{pmatrix},\,z\in\C.
\]
The resolvent of this Hermitization $G_z(\eta) = \left(\mathcal{H}_z-i\eta\right)^{-1}$ can be written as
\[
G_z(\eta) = \begin{pmatrix} 
i\eta\tilde{H}_z(\eta) & \tilde{H}_z(\eta)(A-z)\\
(A^\ast-\bar{z})\tilde{H}_z(\eta) & i\eta H_z(\eta)
\end{pmatrix},
\]
where
\begin{align*}
H_z(\eta) &= \left[(A-z)^\ast(A-z)+\eta^2\right]^{-1},\\
\tilde{H}_z(\eta) &= \left[(A-z)(A-z)^\ast+\eta^2\right]^{-1}.
\end{align*}
We will now condition on some assumptions on $A$; we explain shortly why they hold with probability $1-O(N^{-D})$ for any $D>0$. Consider $\epsilon_0>0$. We assume that for any $\eta, \eta_1, \eta_2, \in[N^{-\frac12+\epsilon_0}, 10]$, $\eta^\ast = \max\{\eta_1,\eta_2\}$ and $z,z_1,z_2\in \C$ such that $|z|,|z_{1}|,|z_{2}|<1-\tau$ for some fixed $\tau>0$, there exist constants $c = c(\epsilon_0)>0$ and $C=C(\epsilon_0)>0$ such that

\paragraph{A1:} We have
\begin{align}
c\eta^{-1} \le \ntr{H_{z}(\eta)} \le C\eta^{-1},\tag{A1.1}\label{assn:1.1}\\
c\eta^{-2} \le \ntr{H_{z}(\eta)\tilde{H}_z(\eta)} \le C\eta^{-2},\tag{A1.2}\label{assn:1.2}\\
\left|\ntr{H_{z}(\eta)^2(A-z)}\right| \le C\eta^{-1},\tag{A1.3}\label{assn:1.3}\\
c\eta^{-3} \le \ntr{H_{z}(\eta)^2} \le C\eta^{-3}.\tag{A1.4}\label{assn:1.4}
\end{align}
For
\[
B_j = \begin{pmatrix}0&1\\0&0\end{pmatrix}\otimes I_N \text{ or } \begin{pmatrix}0&0\\1&0\end{pmatrix}\otimes I_N
\]
we have
\begin{align}
\left|\ntr{G_{z}(\eta_1)B_1G_{z}(\eta_2)B_2}\right| &\le C.\tag{A1.5}\label{assn:1.5}
\end{align}

\paragraph{A2:} We have 
\begin{align}
\ntr{H_{z_1}(\eta_1)H_{z_2}(\eta_2)} &\ge c\left(\eta^\ast\right)^{-2},\tag{A2.1}\label{assn:2.1}\\
\ntr{\tilde{H}_{z_1}(\eta_1)\tilde{H}_{z_2}(\eta_2)} &\ge c\left(\eta^\ast\right)^{-2},\tag{A2.2}\label{assn:2.2}\\
\ntr{H_{z_1}(\eta_1)\tilde{H}_{z_2}(\eta_2)} &\ge \frac{c}{\eta^\ast\left(\eta^\ast+|z_1-z_2|^2\right)}.\tag{A2.3}\label{assn:2.3}
\end{align}
We fix $m$ families of finite rank deterministic matrices $T_j$, where $j\in\llbracket 1,m\rrbracket$ and $m\in\mathbb{Z}_+$. Then matrices $T_j T_j^\ast$ have spectral decomposition
\begin{align}\label{eq:T-spectral}
T_j T_j^\ast = \sum_{k=1}^{l_j} q_{j,k} \mathbf{w}_{j,k} \mathbf{w}_{j,k}^\ast,
\end{align}
where $q_{j,k}$ are positive constants and $\{\mathbf{w}_{j,k}\}_{k}$ are deterministic orthonormal vectors for each $j$. We make the following assumption on $A$ in relation to matrices $T_j$.

\paragraph{A3:} For any $\mathbf{w}_1,\mathbf{w}_2\in\left\{\mathbf{w}_{j,k_1},\mathbf{w}_{j,k_1}\pm\mathbf{w}_{j,k_2},\mathbf{w}_{j,k_1}\pm i\mathbf{w}_{j,k_2}\right\}_{j\in\llbracket1,m\rrbracket, k_1,k_2\in\llbracket1,l_j\rrbracket}$, we have 
\begin{align}
\left|\mathbf{w}_1^\ast H_z(\eta) \mathbf{w}_2 - \ntr{H_z(\eta)}\mathbf{w}_1^\ast\mathbf{w}_2\right| \le C\eta^{-\frac32}N^{-\frac12},\tag{A3.1}\label{assn:3.1}\\
\left|\mathbf{w}_1^\ast \tilde{H}_z(\eta) \mathbf{w}_2- \ntr{\tilde{H}_z(\eta)}\mathbf{w}_1^\ast\mathbf{w}_2\right| \le C\eta^{-\frac32}N^{-\frac12}.\tag{A3.2}\label{assn:3.2}
\end{align}

Let us explain why assumptions {\bf A1}, {\bf A2}, {\bf A3} hold for i.i.d. matrices $A$ with probability $1-O(N^{-D})$ for any $D>0$ fixed. All of {\bf A1} is explained in Section 7 of \cite{MO}; see after Lemma 7.1 in \cite{MO}. For {\bf A2}, see \eqref{eq:2G-LL-1}-\eqref{eq:2G-LL-3}. For {\bf A3}, see Proposition \ref{prop:locallaw}.

Now, following \cite{MO}, for any $z\in\C$ such that $|z|<1$ and $t>0$, we define $\eta_{z,t}>0$ such that $t\langle H_{z}(\eta_{z,t})\rangle=1$. By \eqref{assn:1.1}, we know $ct\leq \eta_{z,t}\leq Ct$ for finite constants $C,c>0$ for all $|z|<1-\tau$ and $t\geq N^{-1/2+\epsilon}$. (The existence and uniqueness of such $\eta_{z,t}$ for $|z|<1$ and $t\geq N^{-1/2+\epsilon}$ holds by Theorem 1.1 in \cite{MO}.)
\begin{theorem}\label{theorem:maingauss}
Suppose $t=N^{-\frac13+\epsilon_{0}}$ for some $\epsilon_{0}>0$ fixed. Let $m_R, m_L \in \mathbb{Z}_+$. For any $j\in \llbracket 1,m_R\rrbracket$ let $(\lambda_{j},u_j)$ denote an eigenvalue-right-eigenvector pair of $M_{t}$, and for any $j\in \llbracket m_R+1,m_R+m_L\rrbracket$ let $(\lambda_{j},v_j)$ denote an eigenvalue-left-eigenvector pair of $M_{t}$. Assume that there exists $\upsilon>0$ such that $|\lambda_{j}-\lambda_{k}|\geq N^{-1/2+\upsilon}$ for any distinct $j,k\in\llbracket1,m_R\rrbracket$ and for any distinct $j,k\in\llbracket m_R+1,m_R+m_L\rrbracket$. Let $\{T_{j}\}_{j\in\llbracket1,m_R+m_L\rrbracket}$ be deterministic, finite-rank matrices. There exists $\kappa>0$ and $\mathfrak{q}>0$ depending only on $T_{1},\ldots,T_{m_{R}+m_{L}}$ such that for all $q_{1},\ldots,q_{m_R+m_L}<\mathfrak{q}$, we have
\begin{align}
&\E_{\bm{\lambda}}\prod_{j=1}^{m_R}\exp[Nq_{j}\|T_{j}u_j\|^{2}]\prod_{j=m_R+1}^{m_R+m_L}\exp[Nq_{j}\|T_{j}v_j\|^{2}]\label{eq:thmmain}\\
&=\prod_{j=1}^m\det[1-tq_{j}H_{\lambda_{j}}(\eta_{\lambda_{j},t})T_{j}^{\ast}T_{j}]^{-1}\left[1+O(N^{-\kappa})\right].\nonumber
\end{align}
Above, $\E_{\bm{\lambda}}$ denotes expectation conditioning on $\lambda_{1},\ldots,\lambda_{m_{R}+m_{L}}$.
\end{theorem}
The choice $t=N^{-\frac13+\epsilon}$ can perhaps be improved to $t=N^{-\frac12+\epsilon}$, or even $t=N^{-1+\epsilon}$; see the end of Section 1 in \cite{MO}. For this reason, often in the proof we use only the lower bounds $t\geq N^{-1/2+\epsilon}$.

Theorem \ref{theorem:maingauss} gives joint normality in the large $N$ limit of finitely many components of finitely many left and right eigenvectors (subject to a separation condition on eigenvalues for eigenvectors on the same side, right or left). This is the content of the following result.
\begin{corollary}\label{corollary:maingauss}
The random vector $(N\|T_{j}u_{j}\|^{2},N\|T_{k}v_{k}\|^{2})_{j,k}$ converges in distribution to $(Z_{j}^{2},Z_{k}^{2})_{j,k}$, where $Z_{j}\sim N(0,\|T_{j}\|_{\mathrm{F}}^{2})$ are independent (here, $\|\cdot\|_{\mathrm{F}}$ is the Frobenius norm for square matrices).
\end{corollary}
\begin{proof}
By Theorem \ref{theorem:maingauss}, it suffices to show that for the above choices of $T_{j}$, we have 
\begin{align*}
\det[1-tq_{j}H_{\lambda_{j}}(\eta_{\lambda_{j},t})T_{j}^{\ast}T_{j}]^{-1}=[1-tq_{j}\|T_{j}\|_{\mathrm{F}}^{2}]^{-1}[1+O(N^{-\kappa})]
\end{align*}
for small enough $q_{j}>0$ and for some $\kappa>0$. To this end, we claim 
\begin{align*}
\det[1-tq_{j}H_{\lambda_{j}}(\eta_{\lambda_{j},t})T_{j}^{\ast}T_{j}]^{-1}&=\prod_{k=1}^{\ell_{j}}[1-tq_{j}q_{j,k}\mathbf{w}_{j,k}^{\ast}H_{\lambda_{j}}(\eta_{\lambda_{j},t})\mathbf{w}_{j,k}]^{-1}\left[1+O(N^{-\kappa})\right]\\
&=\prod_{k=1}^{\ell_{j}}[1-q_{j}q_{j,k}\|\mathbf{w}_{j,k}\|^{2}]^{-1}\left[1+O(N^{-\kappa})\right]\\
&=[1-tq_{j}\|T_{j}\|_{\mathrm{F}}^{2}]^{-1}[1+O(N^{-\kappa})],
\end{align*}
where vectors $\mathbf{w}_{j,k}$ and scalars $q_{j,k}$ are defined by \eqref{eq:T-spectral}. The second line holds by \eqref{assn:3.1} and $t=\langle H_{\lambda_{j}}(\eta_{\lambda_{j},t})\rangle$. It suffices to show the first line. To this end, for any $j$, define the vectors
\begin{align*}
\mathbf{x}_{j,k}=q_{j,k}^{\frac12}H_{\lambda_{j}}(\eta_{\lambda_{j},t})^{1/2}\mathbf{w}_{j,k}
\end{align*}
and let $\{\mathbf{y}_{j,k}\}$ be the collection of orthogonal vectors obtained by Gram-Schmidt applied to $\{\mathbf{x}_{j,k}\}_{k}$ without dividing by the norm. In particular, we have 
\begin{align*}
\det[1-tq_{j}H_{\lambda_{j}}(\eta_{\lambda_{j},t})T_{j}^{\ast}T_{j}]^{-1}&=\det\left[1-tq_{j}\sum_{k=1}^{\ell_{j}}\mathbf{x}_{j,k}\mathbf{x}_{j,k}^{\ast}\right]^{-1}=\prod_{k=1}^{\ell_{j}}[1-tq_{j}\mathbf{y}_{j,k}^{\ast}\mathbf{y}_{j,k}]^{-1}.
\end{align*}
By \eqref{assn:1.1} and \eqref{assn:3.1}, we know $\mathbf{x}_{j,k}^{\ast}\mathbf{x}_{j,m}=O(N^{-1/2}t^{-3/2})$ if $k\neq m$ and $C^{-1}t^{-1}\leq |\mathbf{x}_{j,k}|^{2}\leq Ct^{-1}$ for some $C>0$. It follows by Gram-Schmidt that $\mathbf{y}_{j,k}^{\ast}\mathbf{y}_{j,k}=\mathbf{x}_{j,k}^{\ast}\mathbf{x}_{j,k}[1+O(N^{-1/2}t^{-1/2})]$. Thus, 
\begin{align*}
\prod_{k=1}^{\ell_{j}}[1-tq_{j}\mathbf{y}_{j,k}^{\ast}\mathbf{y}_{j,k}]^{-1}&=\prod_{k=1}^{\ell_{j}}[1-tq_{j}\mathbf{x}_{j,k}^{\ast}\mathbf{x}_{j,k}(1+O(N^{-1/2}t^{-1/2})]^{-1}\\
&=\prod_{k=1}^{\ell_{j}}[1-tq_{j}q_{j,k}\mathbf{w}_{j,k}^{\ast}H_{\lambda_{j}}(\eta_{\lambda_{j},t})\mathbf{w}_{j,k}(1+O(N^{-1/2}t^{-1/2})]^{-1}.
\end{align*}
By \eqref{assn:3.1} and $t\geq N^{-1/3+\epsilon_{0}}$ and $t\langle H_{\lambda_{j}}(\eta_{\lambda_{j},t})\rangle=1$, we know for some $\kappa>0$ that
\begin{align*}
tq_{j}q_{j,k}\mathbf{w}_{j,k}^{\ast}H_{\lambda_{j}}(\eta_{\lambda_{j},t})\mathbf{w}_{j,k}&=tq_{j}q_{j,k}\left[\langle H_{\lambda_{j}}(\eta_{\lambda_{j},t})\rangle+O(N^{-1/2}t^{-3/2})\right]\\
&=q_{j}q_{j,k}\left[1+O(N^{-1/2}t^{-1/2})\right]=q_{j}q_{j,k}[1+O(N^{-\kappa})].
\end{align*}
In particular, since $q_{j},q_{j,k}\neq0$ and $|q_{j}|$ is small, this implies that 
\begin{align*}
&[1-tq_{j}q_{j,k}\mathbf{w}_{j,k}^{\ast}H_{\lambda_{j}}(\eta_{\lambda_{j},t})\mathbf{w}_{j,k}(1+O(N^{-1/2}t^{-1/2})]^{-1}\\
&=[1-tq_{j}q_{j,k}\mathbf{w}_{j,k}^{\ast}H_{\lambda_{j}}(\eta_{\lambda_{j},t})\mathbf{w}_{j,k}]^{-1}(1+O(N^{-1/2}t^{-1/2}),
\end{align*}
which completes the proof.
\end{proof}

The first main step to prove Theorem \ref{theorem:maingauss} is the case where $m_{R}+m_{L}=1$. The second is the case $q_{1},\ldots,q_{m_{R}+m_{L}}\leq q$, where we have a priori bounds of the form $\exp[Nq_{j}\|T_{j}u_{j}\|^{2}]\leq1$.
\begin{prop}\label{prop:single}
Suppose $t=N^{-\frac13+\epsilon_{0}}$ for some $\epsilon_{0}>0$ fixed. Let $(\lambda_{1},u)$ denote an eigenvalue-right-eigenvector pair of $M_{t}$, and let $T$ be a fixed deterministic finite rank matrix. There exists $\kappa>0$ and $q_{T}>0$ such that for all $q<q_{T}$, we have 
\begin{align}
\E_{\lambda_{1}}\exp[Nq\|Tu\|^{2}]=\frac{1}{\det[1-tqH_{\lambda_{1}}(\eta_{\lambda_{1},t})T^{\ast}T]}\left[1+O(N^{-\kappa})\right].\label{eq:singleright}
\end{align}
Let $(\lambda_{2},v)$ denote an eigenvalue-left-eigenvector pair of $M_{t}$, and let $T$ be a fixed deterministic finite rank matrix. There exists $\kappa>0$ and $q_{T}>0$ such that for all $q<q_{T}$, we have
\begin{align}
\E_{\lambda_{2}}\exp[Nq\|Tv\|^{2}]=\frac{1}{\det[1-tq\tilde{H}_{\lambda_{2}}(\eta_{\lambda_{2},t})T^{\ast}T]}\left[1+O(N^{-\kappa})\right].\label{eq:singleleft}
\end{align}
\end{prop}

\begin{prop}\label{prop:maingauss}
Retain the setting of Theorem \ref{theorem:maingauss}. Then \eqref{eq:thmmain} holds for all $q_{1},\dots,q_{m_{R}+m_{L}}\leq0$.
\end{prop}

\begin{proof}[Proof of Theorem \ref{theorem:maingauss}]
Proposition \ref{prop:single} implies the term inside the expectation in \eqref{eq:thmmain} (for small enough $q_{j}$) is tight as $N\to\infty$, and its expectation converges along subsequences to that of its subsequential limits. Proposition \ref{prop:maingauss} identifies these subsequential limits.
\end{proof}
\section{Change of variables}
In this section, we prepare the necessary change of variables for multiple left/right eigenvectors. For any positive integer $j$ consider a manifold 
\begin{align*}
\Omega_{j}=\C\times\mathbb{S}^{N-j}\times\C^{N-j}.
\end{align*}
We define Householder transformations $R_{j}:\mathbb{S}^{N-j}\to\mathbf{U}(N-j+1)$ via
\begin{align*}
R_j(u) &= I_{N-j+1} - 2\frac{(\mathbf{e}_1 - u)(\mathbf{e}_1 - u)^\ast}{\|\mathbf{e}_1-u\|^2}.
\end{align*}
Now, define the maps
\begin{align*}
\Phi_j^{R}, \Phi_j^{L}: \Omega_j \times M_{N-j}(\C) \rightarrow M_{N-j+1}(\C)
\end{align*}
given by the data
\begin{align*}
\Phi_j^{R}(\lambda_j, u^{(j-1)}, w_j, M^{(j)}) &= R_j(u^{(j-1)}) \begin{pmatrix}
    \lambda_j & w_j^\ast \\
    0 & M^{(j)}
\end{pmatrix} R_j(u^{(j-1)}),\\
\Phi_j^{L}(\lambda_j, v^{(j-1)}, w_j, M^{(j)}) &= R_j(v^{(j-1)}) \begin{pmatrix}
    \lambda_j & 0 \\
    w_j & M^{(j)}
\end{pmatrix} R_j(v^{(j-1)}).
\end{align*}
Finally, consider a map
\[
\Phi: \prod_{j=1}^{m_R+m_L} \Omega_j \times M_{N-m_R-m_L}(\C) \rightarrow M_N(\C)
\]
given by the composition of several $\Phi_j^{R}$ and $\Phi_j^{L}$ maps
\[
\Phi = \Phi_1^R \circ \left(\mathrm{Id} \times \Phi_2^R\right) \circ \ldots \circ \left(\mathrm{Id} \times \Phi_{m_R}^R\right) \circ \left(\mathrm{Id} \times \Phi_{m_R+1}^L\right) \circ \ldots \circ \left(\mathrm{Id} \times \Phi_{m_R+m_L}^L\right).
\]
For brevity we denote the total number of steps by $m = m_R+m_L$.
\begin{lemma}\label{lemma:jacobian}
The Jacobian of $\Phi$ is given by the following:
\begin{align*}
J(\Phi)=|\Delta(\bm{\lambda})|^{2}\prod_{j=1}^{m}|\det[\lambda_{j}-M^{(m)}]|^{2}.
\end{align*}
\end{lemma}
\begin{proof}
Since $\Phi$ is the composition of several $\Phi_{j}^R$ and  $\Phi_{j}^L$ maps, we have the chain rule 
\[
J(\Phi)=\prod_{j=1}^{m_R}J(\Phi_{j}^R) \prod_{j=m_R+1}^m J(\Phi_{j}^L).
\]
We now claim that 
\begin{align*}
J(\Phi_{j}^R) = J(\Phi_{j}^L) =\left|\det[\lambda_{j}-M^{(j)}]\right|
\end{align*}
for $j\in\llbracket 1,m\rrbracket$. For the maps $\Phi_{j}^R$, this holds by Lemma 3.1 of \cite{MO}. For the maps $\Phi_{j}^L$, we use Lemma 3.1 in \cite{MO} (this lemma holds for left eigenvectors as well, since one can always replace $M^{(j-1)}$, for which we apply Lemma 3.1 of \cite{MO}, by its adjoint). It remains to note 
\[
\left|\det[\lambda_{j}-M^{(j)}]\right|=\left(\prod_{k=j+1}^{m}\left|\lambda_{k}-\lambda_{j}\right|\right)\left|\det[\lambda_{j}-M^{(m)}]\right|.
\]
\end{proof}
We now apply this transformation to the ensemble $M_{t}=A+t^{\frac12}B$. Since $B$ is Gaussian, the distribution of $M_{t}$ is given by the following density with respect to flat measure on $M_{N}(\C)$, the space of $N\times N$ complex matrices:
\begin{align*}
\rho(M_{t})\d M_{t}=\left(\frac{N}{\pi t}\right)^{N^{2}}\exp\left\{-\frac{N}{t}\tr[(M_{t}-A)^{\ast}(M_{t}-A)]\right\}\d M_{t}.
\end{align*}
Now, by following the computations in Section 4 of \cite{MO} and combining them with Lemma \ref{lemma:jacobian}, we have the following formula, which uses notation $A^{(j)}$ and $b_{j}$ that we define afterwards:
\begin{align}
\rho(M_{t})\d M_{t}&=|\Delta(\bm{\lambda})|^{2}\prod_{j=1}^{m}|\det[\lambda_{j}-M^{(m)}]|^{2}\label{eq:changeofvariables}\\
&\times\left(\frac{N}{\pi t}\right)^{N^{2}}\exp\left\{-\frac{N}{t}\tr[(M_{t}^{(m)}-A^{(m)})^{\ast}(M_{t}^{(m)}-A^{(m)})]\right\}\nonumber\\
&\times\prod_{j=1}^{m_R}\exp\left\{-\frac{N}{t}\|(A^{(j-1)}-\lambda_{j})u^{(j-1)}\|^{2}\right\}\prod_{j=m_R+1}^m\exp\left\{-\frac{N}{t}\|(A^{(j-1)}-\lambda_{j})^{\ast}v^{(j-1)}\|^{2}\right\}\nonumber\\
&\times\prod_{j=1}^{m_R}\exp\left\{-\frac{N}{t}\|w_{j}-b_{j}\|^{2}\right\}\prod_{j=m_R+1}^{m}\exp\left\{-\frac{N}{t}\|w_{j}-c_{j}\|^{2}\right\}\nonumber\\
&\times\d\bm{\lambda}\prod_{j=1}^{m_R}\d u^{(j-1)}\prod_{j=m_R+1}^m\d v^{(j-1)} \prod_{j=1}^m\d w_{j}\d M_{t}^{(m)}\nonumber\\
&=:\tilde{\rho}_{\bm{\lambda}}(u,\ldots, u^{(m_R-1)}, v^{(m_R)}, \ldots, v^{(m-1)},w_{1},\ldots,w_{m},M_{t}^{(m)})\nonumber\\
&\times\d\bm{\lambda}\prod_{j=1}^{m_R}\d u^{(j-1)}\prod_{j=m_R+1}^m\d v^{(j-1)} \prod_{j=1}^m\d w_{j}\d M_{t}^{(m)}\nonumber
\end{align}
Here, $M_{t}^{(m)}$ is the same as $M^{(m)}$ from above after applying the change-of-variables procedure to $M_{t}$. The vectors $b_{j}\in \C^{N-j}$ and the matrices $A^{(j)}\in M_{N-j}(\C)$ are defined by the identities
\begin{align*}
R_{j}(u^{(j-1)})A^{(j-1)}R_{j}(u^{(j-1)})&=\begin{pmatrix} a_j &b_{j}^{\ast}\\c_{j}&A^{(j)}\end{pmatrix}.
\end{align*}
We clarify that \eqref{eq:changeofvariables} holds with the additional adjoint in the third line because the $(m_R+1)$-st through $m$-th steps are obtained by following Section 4 of \cite{MO} but for the adjoint of $A^{(j)}$ instead of $A^{(j)}$ itself. This is because we perform the change-of-variables with respect to the left eigenvector; see the proof of Lemma \ref{lemma:jacobian}. In the rest of the paper we will be computing expectations of functions of $u^{(j)}$, $v^{(j)}$, $w_j$ conditionally on $\bm{\lambda}$. Thus we can integrate out $M^{(m)}$ by following Section 5 of \cite{MO} verbatim. Ultimately, we have the following in which $C_{N,t,\bm{\lambda}}$ depends only on $N,t,\bm{\lambda_{1}}$ (and it may vary from line to line):
\begin{align*}
&\int_{M_{N-2}(\C)}\prod_{j=1}^m|\det[\lambda_{j}-M_{t}^{(m)}]|^{2}\exp\left\{-\frac{N}{t}\tr[(M_{t}^{(m)}-A^{(m)})^{\ast}(M_{t}^{(m)}-A^{(m)})]\right\}\d M_{t}^{(m)}\\
&=C_{N,t,\bm{\lambda}}\prod_{j=1}^m\det[H_{\lambda_{j}}(\eta_{\lambda_j,t})]^{-1}\prod_{j=1}^m|\det[V_{j}^{\ast}G_{\lambda_{j}}^{(j-1)}(\eta_{\lambda_{j},t})V_{j}]|^{m}\left[1+O(N^{-\delta})\right].
\end{align*}
Here, $V_{j}=I_{2}\otimes u^{(j-1)}$ for $j\in\llbracket 1,m_R\rrbracket$ and $V_{j}=I_{2}\otimes v^{(j-1)}$ for $j\in\llbracket m_R+1,m\rrbracket$, and $\delta>0$ is fixed. Also, $G_{z}^{(k)}(\eta)$ is the same as $G_{z}(\eta)$ but for the matrix $A^{(k)}$ instead of $A$. We can also integrate out $w_{j}$ variables because they are Gaussian. In particular, the marginal of $\tilde{\rho}_{\bm{\lambda}}$ after integrating out $w_{j}$ variables and $M_{t}^{(m)}$ is 
\begin{align}
\tilde{\rho}_{\bm{\lambda}}(u,\ldots, u^{(m_R-1)}, v^{(m_R)}, \ldots, v^{(m-1)})&=C_{N,t,\bm{\lambda}}\prod_{j=1}^{m}\det[H_{\lambda_{j}}(\eta_{\lambda_{j},t})]^{-1}\prod_{j=1}^{m}|\det[V_{j}^{\ast}G_{\lambda_{j}}^{(j-1)}(\eta_{\lambda_{j},t})V_{j}]|^{m}\nonumber\\
&\times\prod_{j=1}^{m_R}\exp\left\{-\frac{N}{t}\|(A^{(j-1)}-\lambda_{j})u^{(j-1)}\|^{2}\right\}\nonumber\\
&\times\prod_{j=m_R+1}^m\exp\left\{-\frac{N}{t}\|(A^{(j-1)}-\lambda_{j})^{\ast}v^{(j-1)}\|^{2}\right\}.\label{eq:changeofvariables1}
\end{align}
Additionally, we introduce the following measures on $\mathbb{S}^{N-j}$.
\begin{align*}
\nu_j(u^{(j-1)}) &= K_j(\lambda_j,A^{(j-1)})^{-1} \exp\left\{-\frac{N}{t}\|(A^{(j-1)}-\lambda_{j})u^{(j-1)}\|^{2}\right\}\d u^{(j-1)},\, j\in\llbracket 1,m_R\rrbracket,\\
\nu_j(v^{(j-1)}) &= K_j(\lambda_j,A^{(j-1)})^{-1} \exp\left\{-\frac{N}{t}\|(A^{(j-1)}-\lambda_{j})^\ast v^{(j-1)}\|^{2}\right\}\d v^{(j-1)},\, j\in\llbracket m_R+1,m\rrbracket,
\end{align*}
where
\begin{align*}
K_j(\lambda_j,A^{(j-1)}) &= \int_{\mathbb{S}^{N-j}} \exp\left\{-\frac{N}{t}\|(A^{(j-1)}-\lambda_{j})u^{(j-1)}\|^{2}\right\}\d u^{(j-1)},\, j\in\llbracket 1,m_R\rrbracket,\\
K_j(\lambda_j,A^{(j-1)}) &= \int_{\mathbb{S}^{N-j}}\exp\left\{-\frac{N}{t}\|(A^{(j-1)}-\lambda_{j})^\ast v^{(j-1)}\|^{2}\right\}\d v^{(j-1)},\, j\in\llbracket m_R+1,m\rrbracket.
\end{align*}
From Lemma 4.1 of \cite{MO}, we have
\begin{align*}
K_j(\lambda_j,A^{(j-1)}) = C_{N,t,\bm{\lambda}} \prod_{k=1}^{j-1} \left|\det[V_{j}^{\ast}G_{\lambda_{j}}^{(j-1)}(\eta_{\lambda_{j},t})V_{j}]\right|^{-1}(1+O(N^{-\delta})).
\end{align*}
The measures of the variables $w_j$ are Gaussian on $\C^{N-j}$ and given by
\begin{align*}
\d\omega_j(w_j) &= \left(\frac{N}{\pi t}\right)^{N-j}\exp\left\{-\frac{N}{t}\|w_{j}-b_{j}\|^{2}\right\} \d w_j, & j&\in\llbracket1,m_R\rrbracket,\\
\d\omega_j(w_j) &= \left(\frac{N}{\pi t}\right)^{N-j}\exp\left\{-\frac{N}{t}\|w_{j}-c_{j}\|^{2}\right\} \d w_j, & j&\in\llbracket m_R+1,m\rrbracket.
\end{align*}
The the marginal of $\tilde{\rho}_{\bm{\lambda}}$ after integrating out $M_t^{(m)}$ (but not $w_{j}$ variables) is  
\begin{align*}
&\tilde{\rho}_{\bm{\lambda}}(u,\ldots, u^{(m_R-1)}, v^{(m_R)}, \ldots, v^{(m-1)},w_{1},\ldots,w_{m})\prod_{j=1}^{m_R}\d u^{(j-1)}\prod_{j=m_R+1}^m\d v^{(j-1)} \prod_{j=1}^m\d w_{j}\\
&=C_{N,t,\bm{\lambda}}\prod_{j=1}^m\det[H_{\lambda_{j}}(\eta_{\lambda_{j},t})]^{-1}\left|\det[V_{j}^{\ast}G_{\lambda_{j}}^{(j-1)}(\eta_{\lambda_{j},t})V_{j}]\right|^{j}\\
&\times\prod_{j=1}^{m_R}\d \nu_j(u^{(j-1)})\prod_{j=m_R+1}^m\d \nu_j(v^{(j-1)}) \prod_{j=1}^m\d \omega_j(w_{j})\left[1+O(N^{-\delta})\right].
\end{align*}

\section{Proof of Proposition \ref{prop:single}}
We prove \eqref{eq:singleright}, since \eqref{eq:singleleft} follows by the same argument (just replace $A$ by $A^{\ast}$ to make left eigenvectors into right eigenvectors; in particular, this is why \eqref{eq:singleleft} has $\tilde{H}$ instead of $H$). Now, to compute $\E_{\lambda}\exp[Nq\|Tu\|^{2}]$, we use \eqref{eq:changeofvariables1} and integrate out everything except $u$ (equivalently, set $m=m_{R}=1$ in our application of \eqref{eq:changeofvariables1}). This gives the formula
\begin{align}
&\E_{\lambda_{1}}\exp[Nq\|Tu\|^{2}]\approx\det[H_{\lambda_{1}}(\eta_{\lambda_{1},t})]^{-1}\int_{\mathbb{S}^{N-1}}e^{Nq\|Tu\|^{2}}e^{-\frac{N}{t}\|(A-\lambda_{1})u\|^{2}}|\det[V_{1}^{\ast}G_{\lambda_{1}}(\eta_{\lambda_{1},t})V_{1}]|\d u\nonumber\\
&=e^{\frac{N}{t}\eta_{\lambda_{1},t}^{2}}\det[H_{\lambda_{1}}(\eta_{\lambda_{1},t})]^{-1}\int_{\mathbb{S}^{N-1}}e^{-\frac{N}{t}u^{\ast}[(A-\lambda_{1})^{\ast}(A-\lambda_{1})+\eta_{\lambda_{1},t}^{2}-tqT^{\ast}T]u}|\det[V_{1}^{\ast}G_{\lambda_{1}}(\eta_{\lambda_{1}})V_{1}]|\d u\nonumber\\
&=:e^{\frac{N}{t}\eta_{\lambda_{1},t}^{2}}\det[H_{\lambda_{1}}(\eta_{\lambda_{1},t})]^{-1}K_{q}\int_{\mathbb{S}^{N-1}}|\det[V_{1}^{\ast}G_{\lambda_{1}}(\eta_{\lambda_{1},t})V_{1}]|\d\mu_{q}(u),\label{eq:singlestart}
\end{align}
where $\approx$ means true up to a factor of $1+O(N^{-\delta})$ for $\delta>0$, and where 
\begin{align*}
K_{q}&:=\int_{\mathbb{S}^{N-1}}e^{-\frac{N}{t}u^{\ast}[(A-\lambda_{1})^{\ast}(A-\lambda_{1})+\eta_{\lambda_{1},t}^{2}-tqT^{\ast}T]u}\d u,\\
\d\mu_{q}(u)&:=K_{q}^{-1}e^{-\frac{N}{t}u^{\ast}[(A-\lambda_{1})^{\ast}(A-\lambda_{1})+\eta_{\lambda_{1},t}^{2}-tqT^{\ast}T]u}\d u.
\end{align*}
Before we can further analyze $K_{q}$ and the other remaining terms, we need a preliminary estimate. For convenience, let us the use the notation $H_{\lambda,q}(\eta):=[(A-\lambda)^{\ast}(A-\lambda)+\eta^{2}-tqT^{\ast}T]^{-1}$. Note that $H_{\lambda,0,T}(\eta):=H_{\lambda}(\eta)$. Let us also use the following more general notation, since it will be important for the proof of Proposition \ref{prop:maingauss}. For any $j\geq0$, recall $A^{(j)}$. Define 
\begin{align*}
H_{\lambda,q}^{(j)}(\eta)&:=[(A^{(j)}-\lambda)^{\ast}(A^{(j)}-\lambda)+\eta^{2}-tqT^{\ast}T]^{-1},\\
\tilde{H}_{\lambda,q,T}^{(j)}(\eta)&:=[(A^{(j)}-\lambda)(A^{(j)}-\lambda)^{\ast}+\eta^{2}-tqT^{\ast}T]^{-1}
\end{align*}
We also define $H_{\lambda}^{(j)}(\eta):=H_{\lambda,0,T}^{(j)}(\eta)$ and $\tilde{H}_{\lambda}^{(j)}(\eta):=H_{\lambda,0,T}^{(j)}(\eta)$. Here, we always assume that $T$ has the same dimension as $A^{(j)}$, and that $T$ is finite rank.
\begin{lemma}\label{lemma:Hqestimates}
Fix $\lambda\in\C$. There exists $q_{T}>0$ such that if $q<q_{T}$, then for some finite-rank, positive semi-definite Hermitian matrix $Y$ with operator norm $\|Y\|_{\mathrm{op}}=O(1)$, we have 
\begin{align*}
H_{\lambda,q}^{(j)}(\eta_{\lambda,t})=H_{\lambda}^{(j)}(\eta_{\lambda,t})^{\frac12}[I+qY]H_{\lambda}^{(j)}(\eta_{\lambda,t})^{\frac12}.
\end{align*}
Moreover, we have $\langle H_{\lambda,q}^{(j)}(\eta_{\lambda,t})\rangle=\langle H_{\lambda}^{(j)}(\eta_{\lambda,t})\rangle+O(N^{-1})$. The same is true for $\tilde{H}$ in place of $H$.
\end{lemma}
\begin{proof}
We prove the claim for $H$ and not $\tilde{H}$; for $\tilde{H}$, again, the same argument applies (just replace $A$ by $A^{\ast}$). Also, We first assume $j=0$; we comment on general $j$ at the end. Recall $C^{-1}t\leq\eta_{\lambda,t}\leq Ct$ for some $C>0$. By the Woodbury matrix identity, we have
\begin{align*}
H_{\lambda,q}(\eta_{\lambda,t})&=H_{\lambda}(\eta_{\lambda,t})+H_{\lambda}(\eta_{\lambda,t})T^{\ast}[t^{-1}q^{-1}-TH_{\lambda}(\eta_{\lambda,t})T^{\ast}]^{-1}TH_{\lambda}(\eta_{\lambda,t}).
\end{align*}
By \eqref{assn:1.1} and \eqref{assn:3.1}, since $T^{\ast}T$ is a finite rank projection with $O(1)$ operator norm, we have $\|TH_{\lambda}(\eta_{\lambda,t})T^{\ast}\|_{\mathrm{op}}=O(t^{-1})$. Hence, we can choose $q_{T}>0$ such that if $q<q_{T}$, then the term in square brackets on the RHS is a Hermitian operator that is bounded below by $Cq^{-1}t^{-1}$. In particular, the second term on the RHS of the above identity has the form $tqH_{\lambda}(\eta_{\lambda,t})T^{\ast}XTH_{\lambda}(\eta_{\lambda,t})$ with $X$ Hermitian, bounded, and positive, i.e.
\begin{align*}
H_{\lambda,q}(\eta_{\lambda,t})&=H_{\lambda}(\eta_{\lambda,t})+tqH_{\lambda}(\eta_{\lambda,t})T^{\ast}XTH_{\lambda}(\eta_{\lambda,t})\\
&=H_{\lambda}(\eta_{\lambda,t})^{\frac12}\left\{I+tqH_{\lambda}(\eta_{\lambda,t})^{\frac12}T^{\ast}XTH_{\lambda}(\eta_{\lambda,t})^{\frac12}\right\}H_{\lambda}(\eta_{\lambda,t})^{\frac12}.
\end{align*}
To prove the first estimate, it suffices to show that the second term inside the curly brackets without the factor $q$ has operator norm $O(1)$. This term is a finite rank, positive operator, so we can bound its operator norm by its trace. This gives
\begin{align*}
\|tH_{\lambda}(\eta_{\lambda,t})^{\frac12}T^{\ast}XTH_{\lambda}(\eta_{\lambda,t})^{\frac12}\|_{\mathrm{op}}&\leq t\tr H_{\lambda}(\eta_{\lambda,t})^{\frac12}T^{\ast}XTH_{\lambda}(\eta_{\lambda,t})^{\frac12}\\
&=t\tr TH_{\lambda}(\eta_{\lambda,t})T^{\ast}X\\
&\leq t\|X\|_{\mathrm{op}}\tr TH_{\lambda}(\eta_{\lambda,t})T^{\ast}.
\end{align*}
Again, we use \eqref{assn:1.1} and \eqref{assn:3.1} to get $t\tr TH_{\lambda}(\eta_{\lambda,t})T^{\ast}=O(1)$. Now use $\|X\|_{\mathrm{op}}=O(1)$ to get the first estimate. For the comparison of normalized traces, we have
\begin{align*}
\tr H_{\lambda}(\eta_{\lambda,t})^{\frac12}T^{\ast}XTH_{\lambda}(\eta_{\lambda,t})^{\frac12}&=\tr TH_{\lambda}(\eta_{\lambda,t})T^{\ast}X\\
&=O(1)\tr TH_{\lambda}(\eta_{\lambda,t})T^{\ast}=O(t^{-1});
\end{align*}
the last line holds by $\|X\|_{\mathrm{op}}=O(1)$, by \eqref{assn:1.1}, and \eqref{assn:3.1}, and by $\eta_{\lambda,t}^{-1}=O(t^{-1})$. Multiply by $O(N^{-1}t)$ to conclude. For general $j\neq0$, the same argument works. Indeed, all we need are
\begin{align*}
\|TH_{\lambda}^{(j)}(\eta_{\lambda,t})T^{\ast}\|_{\mathrm{op}}=O(t^{-1})\quad\mathrm{and}\quad \tr TH_{\lambda}^{(j)}(\eta_{\lambda,t})T^{\ast}=O(t^{-1}).
\end{align*} 
The former follows from the latter. The latter follows by interlacing (and the corresponding estimate without the $(j)$ superscript).
\end{proof}
We proceed to analyze $K_{q}$. By Lemma \ref{lemma:Hqestimates}, we know $C^{-1}H_{\lambda_{1}}(\eta_{\lambda_{1},t})\leq H_{\lambda_{1},q}(\eta_{\lambda_{1},t})\leq CH_{\lambda_{1}}(\eta_{\lambda_{1},t})$ for a finite, positive constant $C$. In particular, by Lemma 6.1 in \cite{MO}, we have
\begin{align*}
K_{q}&=C_{N,t}\det[H_{\lambda_{1},q}(\eta_{\lambda_{1},t})]\int_{\R}e^{i\frac{N}{t}p}\det[I+ipH_{\lambda_{1},q}(\eta_{\lambda_{1},t})]^{-1}\d p.
\end{align*}
The two-sided bound $C^{-1}H_{\lambda_{1}}(\eta_{\lambda_{1},t})\leq H_{\lambda_{1},q}(\eta_{\lambda_{1},t})\leq CH_{\lambda_{1}}(\eta_{\lambda_{1},t})$ is used in the proof of Lemma 6.1 of \cite{MO} only to guarantee that $H_{\lambda_{1},q}(\eta_{\lambda_{1},t})$ is strictly positive. We now use it to justify the following approximation (see immediately after the proof of Lemma 6.1 in \cite{MO}):
\begin{align*}
\int_{\R}e^{i\frac{N}{t}p}\det[I+ipH_{\lambda_{1},q}(\eta_{\lambda_{1},t})]^{-1}\d p&\approx\int_{\R}e^{i\frac{N}{t}p}\exp\left\{-iNp\langle H_{\lambda_{1},q}(\eta_{\lambda_{1},t})\rangle-\frac12Np^{2}\langle H_{\lambda_{1},q}(\eta_{\lambda_{1},t})^{2}\rangle\right\}\d p.
\end{align*}
By the trace estimate in Lemma \ref{lemma:Hqestimates}, we can continue as follows (below, $\kappa_{N,t,\lambda_{1}}=O(1)$):
\begin{align*}
&\int_{\R}e^{i\frac{N}{t}p}\exp\left\{-iNp\langle H_{\lambda_{1},q}(\eta_{\lambda_{1},t})\rangle-\frac12Np^{2}\langle H_{\lambda_{1},q}(\eta_{\lambda_{1},t})^{2}\rangle\right\}\d p\\
&=\int_{\R}e^{i\frac{N}{t}p[1-t\langle H_{\lambda_{1},q}(\eta_{\lambda_{1},t})\rangle]}\exp\left\{-\frac12Np^{2}\langle H_{\lambda_{1},q}(\eta_{\lambda_{1},t})^{2}\rangle\right\}\d p\\
&=\int_{\R}\exp\left\{i\kappa_{N,t,\lambda_{1}}p\right\}\exp\left\{-\frac12Np^{2}\langle H_{\lambda_{1},q}(\eta_{\lambda_{1},t})^{2}\rangle\right\}\d p\\
&=\frac{\sqrt{2\pi}}{N^{1/2}\langle H_{\lambda_{1},q}(\eta_{\lambda_{1},t})^{2}\rangle^{1/2}}\exp\left\{-\frac{\kappa_{N,t,\lambda_{1}}^{2}}{2N\langle H_{\lambda_{1},q}(\eta_{\lambda_{1},t})^{2}\rangle}\right\}.
\end{align*}
The last line follows by Gaussian Fourier transform. The bounds $C^{-1}H_{\lambda_{1}}(\eta_{\lambda_{1},t})\leq H_{\lambda_{1},q}(\eta_{\lambda_{1},t})$ and $\langle H_{\lambda_{1}}(\eta_{\lambda_{1},q})^{2}\rangle\geq C^{-1}\eta_{\lambda_{1},q}^{-3}$ and $\kappa_{N,t,\lambda_{1}}=O(1)$ show that the exponential in the last line is $\approx1$. Now, note that if $q\geq0$, then $H_{\lambda_{1},q}(\eta_{\lambda_{1},t})-H_{\lambda_{1}}(\eta_{\lambda_{1},t})\geq0$, and that if $q\leq0$, the reverse inequality holds. This follows by Lemma \ref{lemma:Hqestimates}. In particular, $H_{\lambda_{1},q}(\eta_{\lambda_{1},t})-H_{\lambda_{1}}(\eta_{\lambda_{1},t})$ is always positive or negative semi-definite; it cannot have both positive and negative eigenvalues. Using this, we have the following, where the last line follows by $H_{\lambda_{1},q}(\eta_{\lambda_{1},t})\leq CH_{\lambda_{1}}(\eta_{\lambda_{1},t})$ and Lemma \ref{lemma:Hqestimates}:
\begin{align*}
&|\langle H_{\lambda_{1},q}(\eta_{\lambda_{1},t})^{2}\rangle^{\frac12}-\langle H_{\lambda_{1},q}(\eta_{\lambda_{1},t})^{2}\rangle^{\frac12}|\lesssim|\langle H_{\lambda_{1},q}(\eta_{\lambda_{1},t})^{2}\rangle-\langle H_{\lambda_{1}}(\eta_{\lambda_{1},t})^{2}\rangle|^{\frac12}\\
&\lesssim(\|H_{\lambda_{1},q}(\eta_{\lambda_{1},t})\|_{\mathrm{op}}+\|H_{\lambda_{1}}(\eta_{\lambda_{1},t})\|_{\mathrm{op}})|\langle H_{\lambda_{1},q}(\eta_{\lambda_{1},t})-H_{\lambda_{1}}(\eta_{\lambda_{1},t})\rangle|\\
&\lesssim \|H_{\lambda_{1}}(\eta_{\lambda_{1},t})\|_{\mathrm{op}}||\langle H_{\lambda_{1},q}(\eta_{\lambda_{1},t})-H_{\lambda_{1}}(\eta_{\lambda_{1},t})\rangle|\lesssim N^{-1}t^{-2}.
\end{align*}
Because $C^{-1}H_{\lambda_{1}}(\eta_{\lambda_{1},t})\leq H_{\lambda_{1},q}(\eta_{\lambda_{1},t})$, we also have $\langle H_{\lambda_{1},q}(\eta_{\lambda_{1},t})^{2}\rangle^{1/2}\gtrsim\langle H_{\lambda_{1}}(\eta_{\lambda_{1},t})^{2}\rangle^{1/2}\gtrsim t^{-3/2}$. We deduce from this and the previous display that $\langle H_{\lambda_{1},q}(\eta_{\lambda_{1},t})^{2}\rangle^{1/2}\approx\langle H_{\lambda_{1}}(\eta_{\lambda_{1},t})^{2}\rangle^{1/2}$ since $t\ll1$. Ultimately, since $\langle H_{\lambda_{1}}(\eta_{\lambda_{1},t})^{2}\rangle$ depends only on $N,t$ (recall that $\lambda_{1}$ is fixed), we have 
\begin{align}
K_{q}&=C_{N,t}\det[H_{\lambda_{1},q}(\eta_{\lambda_{1},t})]\int_{\R}e^{i\frac{N}{t}p}\det[I+ipH_{\lambda_{1},q}(\eta_{\lambda_{1},t})]^{-1}\d p\approx C_{N,t}\det[H_{\lambda_{1},q}(\eta_{\lambda_{1},t})].\label{eq:kcomputation}
\end{align}
We now control the $\d\mu_{q}(u)$ integration. First, we record the following from Section 6 of \cite{MO}:
\begin{align*}
|\det[V_{1}^{\ast}G_{\lambda_{1}}(\eta_{\lambda_{1},t})V_{1}]|&=\eta_{\lambda_{1},t}^{2}(u^{\ast}H_{\lambda_{1}}(\eta_{\lambda_{1},t})u)(u^{\ast}\tilde{H}_{\lambda_{1}}(\eta_{\lambda_{1},t})u)+|u^{\ast}H_{\lambda_{1}}(\eta_{\lambda_{1},t})(A-\lambda_{1})u|^{2}.
\end{align*}
We claim the following concentration estimates:
\begin{align}
\mu_{q}\left(|\eta_{\lambda_{1},t}u^{\ast}H_{\lambda_{1}}(\eta_{\lambda_{1},t})u-t\eta_{\lambda_{1},t}^{-1}\langle\tilde{H}_{\lambda_{1}}(\eta_{\lambda_{1},t})H_{\lambda_{1}}(\eta_{\lambda_{1},t})\rangle|\geq\frac{\log N}{N^{1/2}t}\right)&\leq e^{-C\log^{2}N},\label{eq:singleconcentrate1}\\
\mu_{q}\left(|\eta_{\lambda_{1},t}u^{\ast}\tilde{H}_{\lambda_{1}}(\eta_{\lambda_{1},t})u-t\eta_{\lambda_{1},t}^{-1}\langle\tilde{H}_{\lambda_{1}}(\eta_{\lambda_{1},t})^{2}\rangle|\geq\frac{\log N}{N^{1/2}t^{3/2}}\right)&\leq e^{-C\log^{2}N},\label{eq:singleconcentrate2}\\
\mu_{q}\left(|\eta_{\lambda_{1},t}u^{\ast}H_{\lambda_{1}}(\eta_{\lambda_{1},t})(A-\lambda_{1})u-t\eta_{\lambda_{1},t}^{-1}\langle\tilde{H}_{\lambda_{1}}(\eta_{\lambda_{1},t})(A-\lambda_{1})\rangle|\geq\frac{\log N}{N^{1/2}t}\right)&\leq e^{-C\log^{2}N}.\label{eq:singleconcentrate3}
\end{align}
For $q=0$, these estimates are exactly the content of Lemma 6.2 in \cite{MO}. Assuming these estimates hold, we can finish the proof of Proposition \ref{prop:single} by following the calculation after the proof of Lemma 6.3 in \cite{MO} (in said calculation, we set $j=1$, so that Lemma 6.3 in \cite{MO} is unnecessary). In particular, this would give 
\begin{align*}
\int_{\mathbb{S}^{N-1}}|\det[V_{1}^{\ast}G_{\lambda_{1}}(\eta_{\lambda_{1},t})V_{1}]|\d\mu_{q}(u)&\approx C_{N,t,\lambda_{1}}.
\end{align*}
We can combine the previous display with our computation \eqref{eq:kcomputation} of $K_{q}$ and \eqref{eq:singlestart} to get 
\begin{align*}
\E_{\lambda_{1}}\exp[Nq\|Tu\|^{2}]\approx C_{N,t,\lambda_{1}}\det[H_{\lambda_{1}}(\eta_{\lambda_{1},t})]^{-1}\det[H_{\lambda_{1},q}(\eta_{\lambda_{1},t})]=C_{N,t,\lambda_{1}}\det[1-tqH_{\lambda_{1}}(\eta_{\lambda_{1},t})T^{\ast}T]^{-1},
\end{align*}
where the last identity follows from an elementary resolvent identity. To finish the proof of Proposition \ref{prop:single}, it suffices to note that the constant $C_{N,t,\lambda_{1}}$ on the far RHS is equal to $1+O(N^{-\kappa})$ for some $\kappa>0$. This can be verified by plugging in $q=0$ in the above identity.

We now show \eqref{eq:singleconcentrate1}-\eqref{eq:singleconcentrate3}. We give details for \eqref{eq:singleconcentrate1}; it amounts to adjustments in the proof of Lemma 6.2 in \cite{MO}. The other two estimates follow by the same adjustments. We first prove 
\begin{align}
\mu_{q}\left(\eta_{\lambda_{1},t}u^{\ast}H_{\lambda_{1}}(\eta_{\lambda_{1},t})u-t\eta_{\lambda_{1},t}^{-1}\langle\tilde{H}_{\lambda_{1}}(\eta_{\lambda_{1},t})H_{\lambda_{1}}(\eta_{\lambda_{1},t})\rangle\geq\frac{\log N}{N^{1/2}t}\right)&\leq e^{-C\log^{2}N}.\label{eq:singleconcentrate1a}
\end{align}
The bound \eqref{eq:singleconcentrate1} would then follow by proving the same but replacing $H_{\lambda_{1}}(\eta_{\lambda_{1},t})$ by its negative:
\begin{align*}
\mu_{q}\left(-\eta_{\lambda_{1},t}u^{\ast}H_{\lambda_{1}}(\eta_{\lambda_{1},t})u+t\eta_{\lambda_{1},t}^{-1}\langle\tilde{H}_{\lambda_{1}}(\eta_{\lambda_{1},t})H_{\lambda_{1}}(\eta_{\lambda_{1},t})\rangle\geq\frac{\log N}{N^{1/2}t}\right)&\leq e^{-C\log^{2}N}.
\end{align*}
As in Lemma 6.2 in \cite{MO}, this holds by the same argument as the proof of \eqref{eq:singleconcentrate1a}, so we focus on \eqref{eq:singleconcentrate1a}. By following the proof of Lemma 6.2 in \cite{MO}, we eventually apply Markov to the LHS of \eqref{eq:singleconcentrate1a}. In particular, let $Y$ be a Hermitian, semi-definite matrix. (Any matrix is a linear combination of such matrices $Y$.) We must control the following for $r>0$:
\begin{align*}
m_{Y}(r)&=e^{-\frac{rt}{N}\tr[\tilde{H}_{\lambda_{1}}(\eta_{\lambda_{1},t})Y]}\int_{\mathbb{S}^{N-1}}e^{ru^{\ast}Yu}\d\mu_{q}(u)\\
&=e^{-\frac{rt}{N}\tr[\tilde{H}_{\lambda_{1}}(\eta_{\lambda_{1},t})Y]}e^{\frac{N}{t}\eta_{\lambda_{1},t}^{2}}K_{q}^{-1}\int_{\mathbb{S}^{N-1}}e^{-\frac{N}{t}u^{\ast}[(A-\lambda_{1})^{\ast}(A-\lambda_{1})+\eta_{\lambda_{1},t}^{2}-tqT^{\ast}T-\frac{rt}{N}Y]u}\d u.
\end{align*}
For the choice of $Y$ which produces \eqref{eq:singleconcentrate1a}, as in the proof of Lemma 6.2 in \cite{MO}, we pick $r=N^{-1/2}\log N$. It turns out by inspection of the proof of Lemma 6.2 in \cite{MO} that we always have $r\leq N^{-1/2}t^{-1/2}\log N$, so the following argument will use only this. Our choices of $Y$ will also always satisfy $\|Y\|_{\mathrm{op}}=O(1)$. As in our computation of $K_{q}$ from earlier, we use Lemma 2.3 in \cite{MO} to get 
\begin{align*}
&\int_{\mathbb{S}^{N-1}}e^{-\frac{N}{t}u^{\ast}[(A-\lambda_{1})^{\ast}(A-\lambda_{1})+\eta_{\lambda_{1},t}^{2}-tqT^{\ast}T-\frac{rt}{N}Y]u}\d u\\
&=C_{N,t}\det[H_{\lambda_{1},q,r}(\eta_{\lambda_{1},t})]\int_{\R}e^{i\frac{N}{t}p}\det[I+ipH_{\lambda_{1},q,Y}(\eta_{\lambda_{1},t})]^{-1}\d p,
\end{align*}
where $H_{\lambda_{1},q,Y}(\eta_{\lambda_{1},t}):=[(A-\lambda_{1})^{\ast}(A-\lambda_{1})+\eta_{\lambda_{1},t}^{2}-tqT^{\ast}T-N^{-1}rtY]^{-1}$. (Technically, for this to apply, we need $H_{\lambda_{1},q,Y}(\eta_{\lambda_{1},t})>0$ in order to perform a Gaussian integration. This holds by Lemma \ref{lemma:Hqestimates} and the trivial bound $N^{-1}t\|Y\|_{\mathrm{op}}=O(N^{-1}t)$.) In fact, this gives $H_{\lambda_{1},q,r}(\eta_{\lambda_{1},t})\gtrsim H_{\lambda_{1},q}(\eta_{\lambda_{1},t})$. In particular, we can again approximate the determinant in the $\d p$ integration by a Gaussian density. It also allows us to restrict to the region $|p|\leq N^{-1/2}t^{3/2}\log N$. Ultimately, 
\begin{align*}
&\int_{\R}e^{i\frac{N}{t}p}\det[I+ipH_{\lambda_{1},q,Y}(\eta_{\lambda_{1},t})]^{-1}\d p\\
&\approx\int_{\R}e^{i\frac{N}{t}p}\exp\left\{-ipN\langle H_{\lambda_{1},q,Y}(\eta_{\lambda_{1},t})\rangle -\frac12p^{2}N\langle H_{\lambda_{1},q,Y}(\eta_{\lambda_{1},t})^{2}\rangle \right\}\d p\\
&\approx\int_{\R}e^{i\frac{N}{t}p}\exp\left\{-ipN\langle H_{\lambda_{1},q}(\eta_{\lambda_{1},t})\rangle -\frac12p^{2}N\langle H_{\lambda_{1},q}(\eta_{\lambda_{1},t})^{2}\rangle \right\}\d p.
\end{align*}
(The last line follows since $\langle H_{\lambda_{1},q,Y}(\eta_{\lambda_{1},t})\rangle=\langle H_{\lambda_{1},q}(\eta_{\lambda_{1},t})\rangle+O(N^{-1}t^{-2})$ by resolvent perturbation, the bound $N^{-1}t\|Y\|_{\mathrm{op}}=O(N^{-1}t)$, and $\langle H_{\lambda_{1},q}(\eta_{\lambda_{1},t})\rangle\lesssim\langle H_{\lambda_{1}}(\eta_{\lambda_{1},t})\rangle\lesssim t^{-1}$; see Lemma \ref{lemma:Hqestimates}. Multiply by $p=O(N^{-1/2}t^{3/2}\log N)$ to get an error inside the exponential of $O(N^{-1/2}t^{-1/2})$. A similar perturbation shows $p^{2}N\langle H_{\lambda_{1},q,Y}(\eta_{\lambda_{1},t})^{2}\rangle=p^{2}N\langle H_{\lambda_{1},q}(\eta_{\lambda_{1},t})^{2}\rangle+O(N^{-\delta})$ for some $\delta>0$, hence the last line follows.) On the other hand, we have 
\begin{align*}
\det[H_{\lambda_{1},q,r}(\eta_{\lambda_{1},t})]&=\det[H_{\lambda_{1},q}(\eta_{\lambda_{1},t})]\det[I-N^{-1}rtH_{\lambda_{1},q}(\eta_{\lambda_{1},t})^{1/2}YH_{\lambda_{1},q}(\eta_{\lambda_{1},t})^{1/2}]^{-1}\\
&=\det[H_{\lambda_{1},q}(\eta_{\lambda_{1},t})]\exp\left\{-\tr\log\left[1-N^{-1}rtH_{\lambda_{1},q}(\eta_{\lambda_{1},t})^{1/2}YH_{\lambda_{1},q}(\eta_{\lambda_{1},t})^{1/2}\right]\right\}.
\end{align*}
Ultimately, we have 
\begin{align*}
m_{Y}(r)&\approx C_{N,t,\lambda_{1}}K_{q}^{-1}\det[H_{\lambda_{1},q}(\eta_{\lambda_{1},t})]\int_{\R}e^{i\frac{N}{t}p}\exp\left\{-ipN\langle H_{\lambda_{1},q}(\eta_{\lambda_{1},t})-\frac12p^{2}N\langle H_{\lambda_{1},q}(\eta_{\lambda_{1},t})^{2}\rangle\right\}\d p\\
&\times \exp\left\{-\tr[H_{\lambda_{1},q}(\eta_{\lambda_{1},t})N^{-1}rtY]-\tr\log\left(1-N^{-1}rt[H_{\lambda_{1},q}(\eta_{\lambda_{1},t})^{\frac12}Y[H_{\lambda_{1},q}(\eta_{\lambda_{1},t})^{-\frac12}]\right)\right\}.
\end{align*}
By our asymptotics for $K_{q}$ from earlier, we deduce 
\begin{align*}
m_{Y}(r)&\approx C_{N,t,\lambda_{1}}\exp\left\{-\tr[H_{\lambda_{1},q}(\eta_{\lambda_{1},t})N^{-1}rtY]-\tr\log\left(1-N^{-1}rt[H_{\lambda_{1},q}(\eta_{\lambda_{1},t})^{\frac12}Y[H_{\lambda_{1},q}(\eta_{\lambda_{1},t})^{-\frac12}]\right)\right\}.
\end{align*}
The exponential on the RHS is equal to $1$ when we set $r=0$; this means $C_{N,t,\lambda_{1}}\approx1$. At this point, we can now follow the proof of Lemma 6.2 and use elementary log-inequalities to get
\begin{align*}
m_{Y}(r)&\lesssim \exp\left\{\frac{\tr[H_{\lambda_{1},q}(\eta_{\lambda_{1},t})^{\frac12}Y[H_{\lambda_{1},q}(\eta_{\lambda_{1},t})^{-\frac12}]^{2}}{1-\|N^{-1}rtH_{\lambda_{1},q}(\eta_{\lambda_{1},t})^{\frac12}Y[H_{\lambda_{1},q}(\eta_{\lambda_{1},t})^{-\frac12}\|_{\mathrm{op}}}\right\}.
\end{align*}
We must now control the terms inside the exponential. This is done for $q=0$ in the proof of Lemma 6.2. To inherit the estimates for all $q<q_{T}$ (with $q_{T}>0$ small enough), it suffices to show that 
\begin{align*}
\|\tilde{H}_{\lambda_{1},q}(\eta_{\lambda_{1},t})^{\frac12}Y\tilde{H}_{\lambda_{1},q}(\eta_{\lambda_{1},t})^{\frac12}\|_{\mathrm{op}}&\lesssim \|\tilde{H}_{\lambda_{1}}(\eta_{\lambda_{1},t})^{\frac12}Y\tilde{H}_{\lambda_{1}}(\eta_{\lambda_{1},t})^{\frac12}\|_{\mathrm{op}}\\
\tr \tilde{H}_{\lambda_{1},q}(\eta_{\lambda_{1},t})^{\frac12}Y\tilde{H}_{\lambda_{1},q}(\eta_{\lambda_{1},t})Y\tilde{H}_{\lambda_{1},q}(\eta_{\lambda_{1},t})^{\frac12}&\lesssim \tr \tilde{H}_{\lambda_{1}}(\eta_{\lambda_{1},t})^{\frac12}Y\tilde{H}_{\lambda_{1}}(\eta_{\lambda_{1},t})Y\tilde{H}_{\lambda_{1}}(\eta_{\lambda_{1},t})^{\frac12}
\end{align*}
for $\tilde{H}_{\lambda,q}(\eta):=[(A-\lambda)(A-\lambda)^{\ast}+\eta^{2}-tqT^{\ast}T]$. Both follow by semi-definiteness of $Y$ and $\tilde{H}_{\lambda_{1},q}(\eta_{\lambda_{1},t})\leq C\tilde{H}_{\lambda_{1}}(\eta_{\lambda_{1},t})$; this can be shown by using the exact same proof of Lemma \ref{lemma:Hqestimates} (but with $A-\lambda$ replaced by its adjoint). The bound \eqref{eq:singleconcentrate1a} follows, and the proof is finished. \qed

\section{Proof of Proposition \ref{prop:maingauss}}

For any $j\in\llbracket1,m\rrbracket$ consider the composition of first $j$ Householder transforms
\begin{align*}
U_j &= \prod_{k=1}^{j}\left(I_{k-1}\oplus R_k(u^{(k-1)})\right),&\, j&\in\llbracket 1,m_R\rrbracket;\\
U_j &= \prod_{k=1}^{m_R}\left(I_{k-1}\oplus R_k(u^{(k-1)})\right)\prod_{k=m_R+1}^{j}\left(I_{k-1}\oplus R_k(v^{(k-1)})\right),&\, j&\in\llbracket m_R+1,m\rrbracket.
\end{align*}
We start by expressing the finite-rank projections of left and right eigenvectors $u_j$, $v_j$ through the integration variables $u^{(j-1)}, v^{(j-1)}$. This is the content of the following lemma.
\begin{lemma}\label{lemma:simplify}
Set $\epsilon_1 = \min\{\epsilon_0, \upsilon\}$. Then for $j\in\llbracket1,m_R\rrbracket$ we have
\begin{align*}
|T_j|u_j = |T_j|U_{j-1}\begin{pmatrix}0\\u^{(j-1)}\end{pmatrix} + \oh_{\mu}\left(N^{-\frac12-\epsilon_1}\right).
\end{align*}
For $j\in\llbracket m_R+1,m\rrbracket$ we have
\begin{align*}
|T_j|v_j = |T_j|U_{j-1}\begin{pmatrix}0\\v^{(j-1)}\end{pmatrix} + \oh_{\mu}\left(N^{-\frac12-\epsilon_1}\right).
\end{align*}
\end{lemma}
The error terms here are bounded in the sense of stochastic domination, defined below, with respect to the measure $\mu$, where
\[
\d \mu = \prod_{j=1}^m \d \nu_j \prod_{j=1}^m \d \omega_j.
\]
\begin{defn}[Stochastic domination]
Suppose 
\[
X = \left\{X_N(s) : N\in \zz_+, s\in S_N\right\} \text{ and }Y = \left\{Y_N(s) : N\in \zz_+, s\in S_N\right\}
\]
are sequences of random variables, possibly parametrized by $s$. We say that $X$ is stochastically dominated by $Y$ uniformly in $s$ with respect to measure $\mu$ and write $X\prec_{\mu} Y$ or $X = \oh_{\mu}(Y)$ if for any $\varepsilon, D>0$ we have
\[
\sup_{S_N} \pp_{\mu}\left(X_N(s) > N^{\varepsilon} Y_N(s)\right) < N^{-D}
\]
for large enough $N$.
\end{defn}
We prove Lemma \ref{lemma:simplify} in section \ref{sec:concentration}, where we collect all technical high probability estimates with respect to the measure $\mu$.
In view of Lemma \ref{lemma:simplify}, consider the event $\mathcal{E}:=\cap_{j=1}^{m}\mathcal{E}_{j}$, where for some $\kappa>0$ small,
\begin{align*}
\mathcal{E}_{j}&:=\left\{\|T_{j}u_{j}\|^{2}=\left\||T_{j}|U_{j-1}\begin{pmatrix}0\\u^{(j-1)}\end{pmatrix}\right\|^{2}+O(N^{-1-\kappa})\right\}\quad\mathrm{if} \ j\leq m_{R},\\
\mathcal{E}_{j}&:=\left\{\|T_{j}v_{j}\|^{2}=\left\||T_{j}|U_{j-1}\begin{pmatrix}0\\v^{(j-1)}\end{pmatrix}\right\|^{2}+O(N^{-1-\kappa})\right\}\quad\mathrm{if} \ j\geq m_{R}+1.
\end{align*}
By Lemma \ref{lemma:simplify} and a union bound, we know that $\E_{\bm{\lambda}}\mathbf{1}_{\mathcal{E}^{C}}=O(N^{-D})$ for large $D>0$. Since $q_{j}\leq0$ for all $j$ by assumption, this gives 
\begin{align}
&\E_{\bm{\lambda}}\left\{\prod_{j=1}^{m_{R}}\exp[Nq_{j}\|T_{j}u_{j}\|^{2}]\prod_{j=m_{R}+1}^{m_{R}+m_{L}}\exp[Nq_{j}\|T_{j}v_{j}\|^{2}]\right\}\label{eq:mainstart1}\\
&\approx\E_{\bm{\lambda}}\left\{\prod_{j=1}^{m_{R}}\exp\left[Nq_{j}\left\||T_{j}|U_{j-1}\begin{pmatrix}0\\u^{(j-1)}\end{pmatrix}\right\|^{2}\right]\prod_{j=m_{R}+1}^{m_{R}+m_{L}}\exp\left[Nq_{j}\left\||T_{j}|U_{j-1}\begin{pmatrix}0\\v^{(j-1)}\end{pmatrix}\right\|^{2}\right]\right\}+O(N^{-D}),\nonumber
\end{align}
where, as before, $\approx$ means true up to a factor of $1+O(N^{-\delta})$ for some $\delta>0$. Now use \eqref{eq:changeofvariables1}:
\begin{align*}
&\E_{\bm{\lambda}}\left\{\prod_{j=1}^{m_{R}}\exp\left[Nq_{j}\left\||T_{j}|U_{j-1}\begin{pmatrix}0\\u^{(j-1)}\end{pmatrix}\right\|^{2}\right]\prod_{j=m_{R}+1}^{m_{R}+m_{L}}\exp\left[Nq_{j}\left\||T_{j}|U_{j-1}\begin{pmatrix}0\\v^{(j-1)}\end{pmatrix}\right\|^{2}\right]\right\}\\
&\approx C_{N,t,\bm{\lambda}}\prod_{j=1}^{m}\det[H_{\lambda_{j}}(\eta_{\lambda_{j},t})]^{-1}\int_{\mathbb{S}^{N-1}}\ldots\int_{\mathbb{S}^{N-m}}\prod_{j=1}^{m}|\det[V_{j}^{\ast}G_{\lambda_{j}}^{(j-1)}(\eta_{\lambda_{j},t})V_{j}]|^{m}\\
&\times\prod_{j=1}^{m_{R}}e^{-\frac{N}{t}u^{(j-1),\ast}\left[(A^{(j-1)}-\lambda_{j})^{\ast}(A^{(j-1)}-\lambda_{j})+\eta_{\lambda_{j},t}^{2}-tq_{j}T_{j}^{(j-1),\ast}T_{j}^{(j-1)}\right]u^{(j-1)}}\d u^{(j-1)}\\
&\times\prod_{j=1}^{m_{R}}e^{-\frac{N}{t}v^{(j-1),\ast}\left[(A^{(j-1)}-\lambda_{j})(A^{(j-1)}-\lambda_{j})^{\ast}+\eta_{\lambda_{j},t}^{2}-tq_{j}T_{j}^{(j-1),\ast}T_{j}^{(j-1)}\right]v^{(j-1)}}\d v^{(j-1)}.
\end{align*}
Above, $T_{j}^{(j-1)}$ is the restriction of $T_{j}U_{j-1}$ to the orthogonal complement of $\mathbf{e}_{1},\ldots,\mathbf{e}_{j-1}$. For $j\leq m_{R}$, we set
\begin{align*}
K_{j,q_{j},T_{j}}&:=\int_{\mathbb{S}^{N-j}}e^{-\frac{N}{t}u^{(j-1),\ast}\left[(A^{(j-1)}-\lambda_{j})^{\ast}(A^{(j-1)}-\lambda_{j})+\eta_{\lambda_{j},t}^{2}-tq_{j}T_{j}^{(j-1),\ast}T_{j}^{(j-1)}\right]u^{(j-1)}}\d u^{(j-1)},\\
\d\nu_{j,q_{j},T_{j}}(u^{(j-1)})&:=K_{j,q_{j},T_{j}}^{-1}e^{-\frac{N}{t}u^{(j-1),\ast}\left[(A^{(j-1)}-\lambda_{j})^{\ast}(A^{(j-1)}-\lambda_{j})+\eta_{\lambda_{j},t}^{2}-tq_{j}T_{j}^{(j-1),\ast}T_{j}^{(j-1)}\right]u^{(j-1)}}\d u^{(j-1)},
\end{align*}
and for $j\geq m_{R}+1$, we define 
\begin{align*}
K_{j,q_{j},T_{j}}&:=\int_{\mathbb{S}^{N-j}}e^{-\frac{N}{t}v^{(j-1),\ast}\left[(A^{(j-1)}-\lambda_{j})(A^{(j-1)}-\lambda_{j})^{\ast}+\eta_{\lambda_{j},t}^{2}-tq_{j}T_{j}^{(j-1),\ast}T_{j}^{(j-1)}\right]v^{(j-1)}}\d v^{(j-1)},\\
\d\nu_{j,q_{j},T_{j}}(v^{(j-1)})&:=K_{j,q_{j},T_{j}}^{-1}e^{-\frac{N}{t}v^{(j-1),\ast}\left[(A^{(j-1)}-\lambda_{j})(A^{(j-1)}-\lambda_{j})^{\ast}+\eta_{\lambda_{j},t}^{2}-tq_{j}T_{j}^{(j-1),\ast}T_{j}^{(j-1)}\right]v^{(j-1)}}\d v^{(j-1)}.
\end{align*}
With this notation, we can write
\begin{align}
&\E_{\bm{\lambda}}\left\{\prod_{j=1}^{m_{R}}\exp\left[Nq_{j}\left\||T_{j}|U_{j-1}\begin{pmatrix}0\\u^{(j-1)}\end{pmatrix}\right\|^{2}\right]\prod_{j=m_{R}+1}^{m_{R}+m_{L}}\exp\left[Nq_{j}\left\||T_{j}|U_{j-1}\begin{pmatrix}0\\v^{(j-1)}\end{pmatrix}\right\|^{2}\right]\right\}\label{eq:mainstart2}\\
&\approx C_{N,t,\bm{\lambda}}\prod_{j=1}^{m}\det[H_{\lambda_{j}}(\eta_{\lambda_{j},t})]^{-1}K_{j,q_{j},T_{j}}\nonumber\\
&\times\int_{\mathbb{S}^{N-1}}\ldots\int_{\mathbb{S}^{N-m}}\prod_{j=1}^{m}|\det[V_{j}^{\ast}G_{\lambda_{j}}^{(j-1)}(\eta_{\lambda_{j},t})V_{j}]|^{j}\prod_{j=1}^{m_{R}}\d\nu_{j,q_{j},T_{j}}(u^{(j-1)})\prod_{j=m_{R}+1}^{m}\d\nu_{j,q_{j},T_{j}}(v^{(j-1)}).\nonumber
\end{align}
Similar to the proof of Proposition \ref{prop:single}, we now focus on the following two lemmas. The first computes $\det[H_{\lambda_{j}}(\eta_{\lambda_{j},t})]^{-1}K_{j,q_{j},T_{j}}$ for all $j$. The second computes the remaining spherical integrals.
\begin{lemma}\label{lemma:kjcomputation}
There exists an event $\mathcal{E}$ such that $\E_{\bm{\lambda}}\mathbf{1}_{\mathcal{E}}=O(N^{-D})$ for any $D>0$ fixed and such that on $\mathcal{E}$, we have the following for all $j=1,\ldots,m$:
\begin{align}
&\det[H_{\lambda_{j}}(\eta_{\lambda_{j},t})]^{-1}K_{j,q_{j},T_{j}}\nonumber\\
&\approx \det\left[I-tq_{j}\begin{pmatrix}0&0\\0&H_{\lambda_{j}}^{(j-1)}(\eta_{\lambda_{j},t})\end{pmatrix}U_{j-1}^{\ast}T_{j}^{\ast}T_{j}U_{j-1}\right]^{-1}\prod_{\ell=1}^{j-1}|\det[V_{\ell}^{\ast}G_{\lambda_{\ell}}^{(\ell-1)}(\eta_{\lambda_{\ell},t})V_{\ell}]|^{-1}\label{eq:kjcomputation1}\\
&\approx\det[I-tq_{j}H_{\lambda_{j}}(\eta_{\lambda_{j},t})T_{j}^{\ast}T_{j}]^{-1}\prod_{\ell=1}^{j-1}|\det[V_{\ell}^{\ast}G_{\lambda_{\ell}}^{(\ell-1)}(\eta_{\lambda_{\ell},t})V_{\ell}]|^{-1}.\label{eq:kjcomputation2}
\end{align}
\end{lemma}
\begin{lemma}\label{lemma:concentrationallj}
For all $j=1,\ldots,m$, there exists a constant $C_{N,t,j}$ such that 
\begin{align*}
\int_{\mathbb{S}^{N-j}}|\det[V_{j}^{\ast}G_{\lambda_{j}}^{(j-1)}(\eta_{\lambda_{j},t})V_{j}]|^{j}\d\nu_{j,q_{j},T_{j}}(u^{(j-1)})\approx C_{N,t,j}.
\end{align*}
\end{lemma}
Assuming Lemmas \ref{lemma:kjcomputation} and \ref{lemma:concentrationallj}, since $q_{j}\leq0$, we can deduce from \eqref{eq:mainstart2} that 
\begin{align*}
&\E_{\bm{\lambda}}\left\{\prod_{j=1}^{m_{R}}\exp\left[Nq_{j}\left\||T_{j}|U_{j-1}\begin{pmatrix}0\\u^{(j-1)}\end{pmatrix}\right\|^{2}\right]\prod_{j=m_{R}+1}^{m_{R}+m_{L}}\exp\left[Nq_{j}\left\||T_{j}|U_{j-1}\begin{pmatrix}0\\v^{(j-1)}\end{pmatrix}\right\|^{2}\right]\right\}\\
&\approx C_{N,t,\bm{\lambda}}\prod_{j=1}^{m}\det[I-tq_{j}H_{\lambda_{j}}(\eta_{\lambda_{j},t})T_{j}^{\ast}T_{j}]^{-1}+O(N^{-D}).
\end{align*}
We set $q=0$ to get $C_{N,t,\bm{\lambda}}\approx1$. Next, we use the finite rank property of $T_{2}$ to get the trivial bound $\det[I-tq_{j}H_{\lambda_{j}}(\eta_{\lambda_{j},t})T_{j}^{\ast}T_{j}]\lesssim \|H_{\lambda_{j}}(\eta_{\lambda_{j},t)})\|_{\mathrm{op}}^{C}\lesssim t^{-C}$ for some $C=O(1)$. This is much bigger than $N^{-D}$ if $D>0$ is large enough, so the second term in the second line of \eqref{eq:mainstart1} is much smaller than the first term therein. We deduce
\begin{align*}
&\E_{\bm{\lambda}}\left\{\prod_{j=1}^{m_{R}}\exp[Nq_{j}\|T_{j}u_{j}\|^{2}]\prod_{j=m_{R}+1}^{m_{R}+m_{L}}\exp[Nq_{j}\|T_{j}v_{j}\|^{2}]\right\}\label{eq:mainstart1}\\
&\approx\E_{\bm{\lambda}}\left\{\prod_{j=1}^{m_{R}}\exp\left[Nq_{j}\left\||T_{j}|U_{j-1}\begin{pmatrix}0\\u^{(j-1)}\end{pmatrix}\right\|^{2}\right]\prod_{j=m_{R}+1}^{m_{R}+m_{L}}\exp\left[Nq_{j}\left\||T_{j}|U_{j-1}\begin{pmatrix}0\\v^{(j-1)}\end{pmatrix}\right\|^{2}\right]\right\},
\end{align*}
at which point it suffices to combine the previous two displays to conclude the proof of Proposition \ref{prop:maingauss}. Thus, the finish the proof of Proposition \ref{prop:maingauss}, we must prove Lemmas \ref{lemma:kjcomputation} and \ref{lemma:concentrationallj}. We prove the latter first, since it is short. 
\begin{proof}[Proof of Lemma \ref{lemma:concentrationallj}]
In Lemmas 6.2 and 6.3 of \cite{MO}, it is shown that there is a decomposition $\mathbb{S}^{N-j}=\mathbb{G}_{j}\cup\mathbb{G}_{j}^{C}$ such that the following hold. 
\begin{enumerate}
\item On $\mathbb{G}_{j}$, we have $|\det[V_{j}^{\ast}G_{\lambda_{j}}^{(j-1)}(\eta_{\lambda_{j},t})V_{j}]|^{j}\approx C_{N,t,j}$ for some $C_{N,t,j}\gtrsim N^{-\upsilon}$ with $\upsilon=O(1)$.
\item We have $\nu_{j,0,T_{j}}(\mathbb{G}_{j}^{C})=O(N^{-D})$ for any large, fixed $D\geq0$. This and $|\det[V_{j}^{\ast}G_{\lambda_{j}}^{(j-1)}(\eta_{\lambda_{j},t})V_{j}]|=O(t^{-C})$ for some $C=O(1)$ (along with $t\geq N^{-1})$ give the following for any $D\geq0$ fixed:
\begin{align*}
\int_{\mathbb{G}_{j}^{C}}|\det[V_{j}^{\ast}G_{\lambda_{j}}^{(j-1)}(\eta_{\lambda_{j},t})V_{j}]|^{j}\d\nu_{j,q_{j},T_{j}}(u^{(j-1)})\lesssim N^{-D}
\end{align*}
\end{enumerate}
Thus, it suffices to show that for any $q_{j}\leq0$ independent of $N$, we have the inequality of measures $\d\nu_{j,q_{j},T_{j}}\leq N^{C}\d\nu_{j,0,T_{j}}$ for some $C=O(1)$ (note that $\nu_{j,0,T_{j}}$ has no dependence on $T_{j}$). To this end, note that by construction and the assumption $q_{j}\leq0$, we have $\d\nu_{j,q_{j},T_{j}}\leq K_{j,q_{j},T_{j}}^{-1}K_{j,0,T_{j}}\d\nu_{j,0,T_{j}}$. By Lemma \ref{lemma:kjcomputation}, we have 
\begin{align*}
K_{j,q_{j},T_{j}}^{-1}K_{j,0,T_{j}}&\approx\det[I-tq_{j}H_{\lambda_{j}}(\eta_{\lambda_{j},t})T_{j}^{\ast}T_{j}].
\end{align*}
Because $T_{j}^{\ast}T_{j}$ is finite rank by assumption, the determinant on the RHS is bounded above by a finite power of the operator norm of $tq_{j}H_{\lambda_{j}}(\eta_{\lambda_{j},t})$, which is at most $O(t^{-1})$. Since $t\geq N^{-1}$, we obtain $K_{j,q_{j},T_{j}}^{-1}K_{j,0,T_{j}}=O(t^{-C})$ as desired.
\end{proof}
The rest of this section is dedicated towards the proof of Lemma \ref{lemma:kjcomputation}. We assume $j=1,\ldots,m_{R}$; for $j=m_{R}+1,\ldots,m$, just replace $A^{(j-1)}-\lambda_{j}$ by its adjoint. 
\subsection{Proof of \eqref{eq:kjcomputation1}}
In this step, there is no need to restrict to an event $\mathcal{E}$. Throughout the proof of \eqref{eq:kjcomputation1}, we will adopt the notation
\begin{align*}
H_{\lambda_{j},q_{j}}^{(j-1)}(\eta_{\lambda_{j},t})&:=[(A^{(j-1)}-\lambda_{j})^{\ast}(A^{(j-1)}-\lambda_{j})+\eta_{\lambda_{j},t}^{2}-tq_{j}T_{j}^{(j-1),\ast}T_{j}^{(j-1)}];
\end{align*}
here $T_{j}^{(j-1)}$ is the restriction of $T_{j}U_{j-1}$ (see Lemma \ref{lemma:simplify}) to the orthogonal complement of $\mathbf{e}_{1},\ldots,\mathbf{e}_{j-1}$. In particular, we choose $T=T_{j}^{(j-1)}$ in the context of Lemma \ref{lemma:Hqestimates}.

The argument is similar to the proof of \eqref{eq:kcomputation}. Lemma \ref{lemma:Hqestimates} gives 
\begin{align*}
C^{-1}H_{\lambda_{j}}^{(j-1)}(\eta_{\lambda_{j},t})\leq H_{\lambda_{j},q_{j}}^{(j-1)}(\eta_{\lambda_{j},t})\leq CH_{\lambda_{j}}^{(j-1)}(\eta_{\lambda_{j},t}),
\end{align*}
where $C>0$ is a fixed constant. This gives $H_{\lambda_{j},q_{j}}^{(j-1)}(\eta_{\lambda_{j},t})>0$, and we can follow the proof of Lemma 6.1 in \cite{MO} verbatim to get
\begin{align*}
K_{j,q_{j},T_{j}}&=C_{N,t,j}\det[H_{\lambda_{j},q_{j}}^{(j-1)}(\eta_{\lambda_{j},t})]\int_{\R}e^{i\frac{N}{t}p}\det[I+ipH_{\lambda_{j},q_{j}}^{(j-1)}(\eta_{\lambda_{j},t})]^{-1}\d p.
\end{align*}
The above two-sided resolvent bound also lets us approximate the last determinant in the previous display by a Gaussian factor, as in the discussion after Lemma 6.1 of \cite{MO}. In particular, with more explanation after, we have 
\begin{align*}
&\int_{\R}e^{i\frac{N}{t}p}\det[I+ipH_{\lambda_{j},q_{j}}^{(j-1)}(\eta_{\lambda_{j},t})]^{-1}\d p\\
&=\int_{\R}e^{i\frac{N}{t}p}\exp\left\{-\tr\log\left[I+ipH_{\lambda_{j},q_{j}}^{(j-1)}(\eta_{\lambda_{j},t})\right]\right\}\d p\\
&\approx\int_{|p|\leq N^{\delta}N^{-1/2}t^{3/2}}e^{i\frac{N}{t}p}\exp\left\{-iNp\langle H_{\lambda_{j},q_{j}}^{(j-1)}(\eta_{\lambda_{j},t})\rangle-\frac12Np^{2}\langle H_{\lambda_{j},q_{j}}^{(j-1)}(\eta_{\lambda_{j},t})^{2}\rangle\right\}\d p.
\end{align*}
Indeed, to show that the last line holds, we use the inequality $\mathrm{Re}\log[1+ix]\geq Cx^{2}$ for a fixed constant $C>0$ along with $\tr H_{\lambda_{j},q_{j}}^{(j-1)}(\eta_{\lambda_{j},t})^{2}=N\langle H_{\lambda_{j},q_{j}}^{(j-1)}(\eta_{\lambda_{j},t})^{2}\rangle\gtrsim Nt^{-3}$. This lets us restrict $\d p$ integration from $\R$ to $|p|\leq N^{\delta}N^{-1/2}t^{3/2}$. After this, we Taylor expand and control the third-order error term by interlacing $\langle H_{\lambda_{j},q_{j}}^{(j-1)}(\eta_{\lambda_{j},t})^{k}\rangle=\langle H_{\lambda_{j},q_{j}}(\eta_{\lambda_{j},t})^{k}\rangle+O(N^{-1}t^{-2k})$ and \eqref{assn:1.4} to get $\langle H_{\lambda_{j},q_{j}}(\eta_{\lambda_{j},t})^{k}\rangle\lesssim t^{-2k+2}\langle H_{\lambda_{j},q_{j}}(\eta_{\lambda_{j},t})\rangle\lesssim t^{-2k+1}$:
\begin{align*}
&\tr\log\left[I+ipH_{\lambda_{j},q_{j}}^{(j-1)}(\eta_{\lambda_{j},t})\right]\\
&=iNp\langle H_{\lambda_{j},q_{j}}^{(j-1)}(\eta_{\lambda_{j},t})\rangle+\frac12Np^{2}\langle H_{\lambda_{j},q_{j}}^{(j-1)}(\eta_{\lambda_{j},t})^{2}\rangle+O(Np^{3}t^{-5})+O(p^{3}t^{-6})\\
&=iNp\langle H_{\lambda_{j},q_{j}}^{(j-1)}(\eta_{\lambda_{j},t})\rangle+\frac12Np^{2}\langle H_{\lambda_{j},q_{j}}^{(j-1)}(\eta_{\lambda_{j},t})^{2}\rangle+O(N^{3\delta}N^{-1/2}t^{-1/2})+O(N^{3\delta}N^{-3/2}t^{-3/2}).
\end{align*}
Now use $t\gg N^{-1/2}$. Finally, by the same token, we can again remove the constraint on $p$ in the $\R$-integration after controlling the Taylor expansion above. In particular, we have 
\begin{align*}
&\int_{\R}e^{i\frac{N}{t}p}\det[I+ipH_{\lambda_{j},q_{j}}^{(j-1)}(\eta_{\lambda_{j},t})]^{-1}\d p\\
&\approx\int_{\R}e^{i\frac{N}{t}p}\exp\left\{-iNp\langle H_{\lambda_{j},q_{j}}^{(j-1)}(\eta_{\lambda_{j},t})\rangle-\frac12Np^{2}\langle H_{\lambda_{j},q_{j}}^{(j-1)}(\eta_{\lambda_{j},t})^{2}\rangle\right\}\d p\\
&=\int_{\R}e^{i\frac{N}{t}p[1-t\langle H_{\lambda_{j},q_{j}}^{(j-1)}(\eta_{\lambda_{j},t})\rangle]}\exp\left\{-\frac12Np^{2}\langle H_{\lambda_{j},q_{j}}^{(j-1)}(\eta_{\lambda_{j},t})^{2}\rangle\right\}\d p.
\end{align*}
By interlacing and Lemma \ref{lemma:Hqestimates}, we get $t\langle H_{\lambda_{j},q_{j}}^{(j-1)}(\eta_{\lambda_{j},t})\rangle=t\langle H_{\lambda_{j},q_{j}}(\eta_{\lambda_{j},t})\rangle+O(N^{-1}t^{-1})=t\langle H_{\lambda_{j}}(\eta_{\lambda_{j},t})\rangle+O(N^{-1}t^{-1})$. By definition of $\eta_{\lambda_{j},t}$, we also have $1-t\langle H_{\lambda_{j}}(\eta_{\lambda_{j},t})\rangle=0$. Hence, for some quantity $\kappa_{N,t,j}=O(N^{-1}t^{-1}))$, we have 
\begin{align*}
&\int_{\R}e^{i\frac{N}{t}p[1-t\langle H_{\lambda_{j},q_{j}}^{(j-1)}(\eta_{\lambda_{j},t})\rangle]}\exp\left\{-\frac12Np^{2}\langle H_{\lambda_{j},q_{j}}^{(j-1)}(\eta_{\lambda_{j},t})^{2}\rangle\right\}\d p\\
&=\int_{\R}e^{i\frac{N}{t}p\kappa_{N,t,j}}\exp\left\{-\frac12Np^{2}\langle H_{\lambda_{j},q_{j}}^{(j-1)}(\eta_{\lambda_{j},t})^{2}\rangle\right\}\d p\\
&=\frac{\sqrt{2\pi}}{N^{1/2}\langle H_{\lambda_{j},q_{j}}^{(j-1)}(\eta_{\lambda_{j},t})^{2}\rangle^{1/2}}\exp\left\{-\frac{\kappa_{N,t,j}^{2}}{2N\langle H_{\lambda_{j},q_{j}}^{(j-1)}(\eta_{\lambda_{j},t})^{2}\rangle}\right\},
\end{align*}
where the last line follows by Gaussian integration. Now, we use $\langle H_{\lambda_{j},q_{j}}^{(j-1)}(\eta_{\lambda_{j},t})^{2}\rangle\lesssim t^{-4}$, so that
\begin{align*}
\exp\left\{-\frac{\kappa_{N,t,j}^{2}}{2N\langle H_{\lambda_{j},q_{j}}^{(j-1)}(\eta_{\lambda_{j},t})^{2}\rangle}\right\}&= \exp\left\{O(N^{-3}t^{-6})\right\}\approx1,
\end{align*}
where the last bound follows by $t\gg N^{-1/2}$. On the other hand, by interlacing once again, we have 
\begin{align*}
N^{\frac12}\langle H_{\lambda_{j},q_{j}}^{(j-1)}(\eta_{\lambda_{j},t})^{2}\rangle^{\frac12}=N^{\frac12}\left\{\langle H_{\lambda_{j},q_{j}}(\eta_{\lambda_{j},t})^{2}\rangle+O(N^{-1}t^{-4})\right\}^{\frac12}.
\end{align*}
By \eqref{assn:1.4}, we have $\langle H_{\lambda_{j},q_{j}}(\eta_{\lambda_{j},t})^{2}\rangle\gtrsim t^{-3}$. Hence, if we now use $t\gg N^{-1}$, the RHS of the previous display is $\approx N^{1/2}\langle H_{\lambda_{j},q_{j}}(\eta_{\lambda_{j},t})^{2}\rangle^{1/2}$. At this point, we can follow the display before \eqref{eq:kcomputation} to get
\begin{align*}
\frac{\sqrt{2\pi}}{N^{1/2}\langle H_{\lambda_{j},q_{j}}^{(j-1)}(\eta_{\lambda_{j},t})^{2}\rangle^{1/2}}\exp\left\{-\frac{\kappa_{N,t,j}^{2}}{2N\langle H_{\lambda_{j},q_{j}}^{(j-1)}(\eta_{\lambda_{j},t})^{2}\rangle}\right\}\approx C_{N,t}.
\end{align*}
By combining our computations thus far, we deduce
\begin{align*}
K_{j,q_{j},T_{j}}&=C_{N,t,j}\det[H_{\lambda_{j},q_{j}}^{(j-1)}(\eta_{\lambda_{j},t})]\int_{\R}e^{i\frac{N}{t}p}\det[I+ipH_{\lambda_{j},q_{j}}^{(j-1)}(\eta_{\lambda_{j},t})]^{-1}\d p\\
&\approx C_{N,t,j}\det[H_{\lambda_{j},q_{j}}^{(j-1)}(\eta_{\lambda_{j},t})].
\end{align*}
By the previous display and elementary resolvent identities (see the display after (6.3) in \cite{MO} for the second line below), we have
\begin{align*}
&\det[H_{\lambda_{j}}(\eta_{\lambda_{j},t})]^{-1}K_{j,q_{j},T_{j}}\approx C_{N,t,j}\det[H_{\lambda_{j}}(\eta_{\lambda_{j},t})]^{-1}\det[H_{\lambda_{j},q_{j}}^{(j-1)}(\eta_{\lambda_{j},t})]\\
&\approx C_{N,t,j}\prod_{\ell=1}^{j-1}|\det[V_{\ell}^{\ast}G_{\lambda_{\ell}}^{(\ell-1)}(\eta_{\lambda_{\ell},t})V_{\ell}]|^{-1}\det[H_{\lambda_{j}}^{(j-1)}(\eta_{\lambda_{j},t})]^{-1}\det[H_{\lambda_{j},q_{j}}^{(j-1)}(\eta_{\lambda_{j},t})],
\end{align*}
as well as 
\begin{align*}
\det[H_{\lambda_{j}}^{(j-1)}(\eta_{\lambda_{j},t})]^{-1}\det[H_{\lambda_{j},q_{j}}^{(j-1)}(\eta_{\lambda_{j},t})]&=\det[I-tq_{j}H_{\lambda_{j}}^{(j-1)}(\eta_{\lambda_{j},t})T_{j}^{(j-1),\ast}T_{j}^{(j-1)}]^{-1}\\
&=\det\left[I-tq_{j}\begin{pmatrix}0&0\\0&H_{\lambda_{j}}^{(j-1)}(\eta_{\lambda_{j},t})\end{pmatrix}U_{j-1}^{\ast}T_{j}^{\ast}T_{j}U_{j-1}\right],
\end{align*}
where the last line follows from recalling that $T_{j}^{(j-1)}$ is the restriction of $T_{j}U_{j-1}$ to the orthogonal complement of $\mathbf{e}_{1},\ldots,\mathbf{e}_{j-1}$. This completes the proof of \eqref{eq:kjcomputation1}. \qed
\subsection{Proof of \eqref{eq:kjcomputation2}}
By the spectral theorem and finite rank property of $T_{j}$, we can write $T_{j}^{\ast}T_{j}=\sum_{k=1}^{\ell_{j}}q_{j,k}\mathbf{w}_{j,k}\mathbf{w}_{j,k}^{\ast}$ for $q_{j,k}>0$. If we define
\begin{align*}
\mathbf{x}_{j,k}^{(j-1)}&:=\begin{pmatrix}0&0\\0&H_{\lambda_{j}}^{(j-1)}(\eta_{\lambda_{j},t})^{\frac12}\end{pmatrix}U_{j-1}^{\ast}\mathbf{w}_{j,k},
\end{align*}
then we have 
\begin{align*}
&\det\left[I-tq_{j}\begin{pmatrix}0&0\\0&H_{\lambda_{j}}^{(j-1)}(\eta_{\lambda_{j},t})\end{pmatrix}U_{j-1}^{\ast}T_{j}^{\ast}T_{j}U_{j-1}\right]=\det\left[I-t\sum_{k=1}^{\ell_{j}}q_{j}q_{j,k}\mathbf{x}_{j,k}^{(j-1)}\mathbf{x}_{j,k}^{(j-1),\ast}\right].
\end{align*}
Lemma \ref{lem:wHw} and polarization give $\mathbf{x}_{j,k}^{(j-1),\ast}\mathbf{x}_{j,m}^{(j-1)}=\mathbf{x}_{j,k}^{\ast}\mathbf{x}_{j,m}+\mathcal{O}(N^{-1}t^{-3})$, where 
\begin{align*}
\mathbf{x}_{j,k}&:=H_{\lambda_{j}}(\eta_{\lambda_{j},t})^{\frac12}\mathbf{w}_{j,k},
\end{align*}
We now argue as in the proof of Corollary \ref{corollary:maingauss}. By \eqref{assn:1.1} and \eqref{assn:3.1}, we have $\mathbf{x}_{j,k}^{\ast}\mathbf{x}_{j,m}=O(N^{-1/2}t^{-3/2})$ if $k\neq m$ and $C^{-1}t^{-1}\leq |\mathbf{x}_{j,k}|^{2}\leq Ct^{-1}$ for some $C>0$. In particular, following the proof of Corollary \ref{corollary:maingauss}, we have the following on an event $\mathcal{E}$ satisfying $\E_{\bm{\lambda}}\mathbf{1}_{\mathcal{E}}=O(N^{-D})$ for any $D>0$ fixed:
\begin{align*}
\det\left[I-t\sum_{k=1}^{\ell_{j}}q_{j}q_{j,k}\mathbf{x}_{j,k}^{(j-1)}\mathbf{x}_{j,k}^{(j-1),\ast}\right]^{-1}&=\prod_{k=1}^{\ell_{j}}\left[1-tq_{j}q_{j,k}|\mathbf{x}_{j,k}^{(j-1)}|^{2}+O(N^{-1}t^{-3})\right]^{-1}\\
&\approx\prod_{k=1}^{\ell_{j}}\left[1-tq_{j}q_{j,k}|\mathbf{x}_{j,k}^{(j-1)}|^{2}\right]^{-1}\\
&=\prod_{k=1}^{\ell_{j}}\left[1-tq_{j}q_{j,k}\mathbf{w}_{j,k}^{\ast}U_{j-1}\begin{pmatrix}0&0\\0&H_{\lambda_{j}}^{(j-1)}(\eta_{\lambda_{j},t})\end{pmatrix}U_{j-1}^{\ast}\mathbf{w}_{j,k}\right]^{-1}\\
&=\prod_{k=1}^{\ell_{j}}\left[1-tq_{j}q_{j,k}\mathbf{w}_{j,k}^{\ast}H_{\lambda_{j}}(\eta_{\lambda_{j},t})\mathbf{w}_{j,k}+\mathcal{O}(N^{-1}t^{-3})\right]^{-1}.
\end{align*}
The second line follows by $tq_{j}q_{j,k}|\mathbf{x}_{j,k}^{(j-1)}|^{2}\gtrsim1$ that we justified in the previous paragraph and the assumption $t\geq N^{-1/3+\epsilon}$ for some $\epsilon>0$. The last bound follows again by Lemma \ref{lem:wHw}. Again, we can remove the $\mathcal{O}$-term by using \eqref{assn:1.1} and \eqref{assn:3.1}; this gives $\mathbf{w}_{j,k}^{\ast}H_{\lambda_{j}}(\eta_{\lambda_{j},t})\mathbf{w}_{j,k}=\langle H_{\lambda_{j}}(\eta_{\lambda_{j},t})\rangle[1+O(N^{-1/2}t^{-3/2})]\gtrsim t^{-1}$. The same reasoning that gave the first line in the above display yields 
\begin{align*}
\prod_{k=1}^{\ell_{j}}\left[1-tq_{j}q_{j,k}\mathbf{w}_{j,k}^{\ast}H_{\lambda_{j}}(\eta_{\lambda_{j},t})\mathbf{w}_{j,k}+\mathcal{O}(N^{-1}t^{-2})\right]^{-1}&\approx\det[I-tq_{j}H_{\lambda_{j}}(\eta_{\lambda_{j},t})T_{j}^{\ast}T_{j}]^{-1},
\end{align*}
which completes the proof. \qed

\section{Concentration}\label{sec:concentration}

In this section we prove the concentration estimates with respect to the measures $\nu_j$ and $\omega_j$. First, we state a master concentration inequality below. This was essentially proved in Lemma 6.2 of \cite{MO}. We state it in a more general form.

\begin{lemma}[Master concentration inequality]\label{lem:master-concentration}
For any $j\in\llbracket 1,m_R\rrbracket$ and any Hermitian matrix $B\in M_{N-j+1}(\C)$ such that 
\[
\left\|\sqrt{H^{(j-1)}_{\lambda_j}(\eta_{\lambda_j,t})}B\sqrt{H^{(j-1)}_{\lambda_j}(\eta_{\lambda_j,t})}\right\|\le K \text{ and } \tr\left(H^{(j-1)}_{\lambda_j}(\eta_{\lambda_j,t})B\right)^2 \le S,
\]
we have
\begin{equation}
\left|u^{(j-1),\ast} B u^{(j-1)} - \frac{t}{N}\tr H^{(j-1)}_{\lambda_j}(\eta_{\lambda_j,t})B\right| \prec_{\nu_j} \frac{t}{N}\max\left\{\sqrt{S}, K\right\}.  
\end{equation}
In particular, in case of rank 1 matrix $B$, we have
\begin{equation}
u^{(j-1),\ast} B u^{(j-1)} \prec_{\nu_j} \frac{t}{N}\tr H^{(j-1)}_{\lambda_j}(\eta_{\lambda_j,t})B.  
\end{equation}
For any $j\in\llbracket m_R+1, m\rrbracket$ and any Hermitian matrix $B\in M_{N-j+1}(\C)$ such that 
\[
\left\|\sqrt{\tilde{H}^{(j-1)}_{\lambda_j}(\eta_{\lambda_j,t})}B\sqrt{\tilde{H}^{(j-1)}_{\lambda_j}(\eta_{\lambda_j,t})}\right\|\le K \text{ and } \tr\left(\tilde{H}^{(j-1)}_{\lambda_j}(\eta_{\lambda_j,t})B\right)^2 \le S,
\]
we have
\begin{equation}
\left|v^{(j-1),\ast} B v^{(j-1)} - \frac{t}{N}\tr \tilde{H}^{(j-1)}_{\lambda_j}(\eta_{\lambda_j,t})B\right| \prec_{\nu_j} \frac{t}{N}\max\left\{\sqrt{S}, K\right\}.  
\end{equation}
In particular, in case of rank 1 matrix $B$, we have
\begin{equation}
v^{(j-1),\ast} B v^{(j-1)} \prec_{\nu_j} \frac{t}{N}\tr \tilde{H}^{(j-1)}_{\lambda_j}(\eta_{\lambda_j,t})B.  
\end{equation}
\end{lemma}

\begin{proof}
Consider the right eigenvector case of $j\in\llbracket 1,m_R\rrbracket$. The case of the left eigenvector step $j\in\llbracket m_R+1, m\rrbracket$ is proved similarly after replacing $A$ with its adjoint. In the right eigenvector case we show that 
\[
\mu\left(u^{(j-1),\ast} B u^{(j-1)} - \frac{t}{N}\tr H^{(j-1)}_{\lambda_j}(\eta_{\lambda_j,t})B>N^{\delta_1} \frac{t}{N}\max\left\{\sqrt{S}, KN^{\delta_1}\right\}\right) \le e^{-CN^{\delta_1}}.
\]
The proof of the other side of the inequality is analogous. Define the moment generating function
\[
m(r, B) = \E_{\mu} \expb{r\left(u^{(j-1),\ast} B u^{(j-1)} - \frac{t}{N}\tr H^{(j-1)}_{\lambda_j}(\eta_{\lambda_j,t})B\right)}.
\]
We take 
\[
r = \left(\max\left\{\frac{t\sqrt{S}}{N}, \frac{tKN^{\delta_1}}{N}\right\}\right)^{-1} = \min\left\{\frac{N}{t\sqrt{S}}, \frac{N}{N^{\delta_1}t K}\right\},
\]
By Markov's inequality, we have
\begin{equation}\label{eq:markov}
\mu\left(r\left(u^{(j-1),\ast} B u^{(j-1)} - \frac{t}{N}\tr H^{(j-1)}_{\lambda_j}(\eta_{\lambda_j,t})B\right)>N^{\delta_1}\right) \le e^{-N^{\delta_1}}m(r, B).
\end{equation}
In the proof Lemma 6.2 of \cite{MO} it is shown that if $\frac{rt}{N}\|B\| < 1-c$ for $0<c<1$, then
\[
m(r, B) \le \expb{c^{-1}\frac{r^2t^2}{N^2} \tr\left(H^{(j-1)}_{\lambda_j}(\eta_{\lambda_j,t})B\right)^2}.
\]
In our case, since $r \le N^{-\delta_1}\frac{N}{tK}$,
\[
\frac{rt}{N}\|B\| \le N^{-\delta_1} < 1-c.
\]
Hence, we have
\[
m(r, B) \le \expb{c^{-1}\frac{r^2t^2S}{N^2}} \le \expb{c^{-1}}.
\]
Plugging this into \eqref{eq:markov} and dividing by $r$ on the left, we get the desired bound.
\end{proof}

\begin{lemma}\label{lem:two-res-removing-j}
For any $j\in\llbracket1,m-1\rrbracket$ and $k, l\in\llbracket 1,m\rrbracket$, we have
\begin{align*}
\ntr{H^{(j)}_{\lambda_k}\left(\eta_{\lambda_k,t}\right)H^{(j)}_{\lambda_l}\left(\eta_{\lambda_l,t}\right)} &= \ntr{H_{\lambda_k}\left(\eta_{\lambda_k,t}\right)H_{\lambda_l}\left(\eta_{\lambda_l,t}\right)} + O\left(\frac{1}{Nt^4}\right),\\
\ntr{H^{(j)}_{\lambda_k}\left(\eta_{\lambda_k,t}\right)\tilde{H}^{(j)}_{\lambda_l}\left(\eta_{\lambda_l,t}\right)} &= \ntr{H_{\lambda_k}\left(\eta_{\lambda_k,t}\right)\tilde{H}_{\lambda_l}\left(\eta_{\lambda_l,t}\right)} + O\left(\frac{1}{Nt^4}\right),\\
\ntr{\tilde{H}^{(j)}_{\lambda_k}\left(\eta_{\lambda_k,t}\right)\tilde{H}^{(j)}_{\lambda_l}\left(\eta_{\lambda_l,t}\right)} &= \ntr{\tilde{H}_{\lambda_k}\left(\eta_{\lambda_k,t}\right)\tilde{H}_{\lambda_l}\left(\eta_{\lambda_l,t}\right)} + O\left(\frac{1}{Nt^4}\right).
\end{align*}
In particular, these together with {\bf A2} imply
\begin{align*}
\ntr{H^{(j)}_{\lambda_k}\left(\eta_{\lambda_k,t}\right)H^{(j)}_{\lambda_l}\left(\eta_{\lambda_l,t}\right)} &\gtrsim t^{-2},\\
\ntr{H^{(j)}_{\lambda_k}\left(\eta_{\lambda_k,t}\right)\tilde{H}^{(j)}_{\lambda_l}\left(\eta_{\lambda_l,t}\right)} &\gtrsim t^{-1}\left(t+|\lambda_k-\lambda_l|^2\right)^{-1},\\
\ntr{\tilde{H}^{(j)}_{\lambda_k}\left(\eta_{\lambda_k,t}\right)\tilde{H}^{(j)}_{\lambda_l}\left(\eta_{\lambda_l,t}\right)} &\gtrsim t^{-2}.
\end{align*}
\end{lemma}

\begin{proof}
We prove the first estimate, the other two a proved similarly. It is sufficient to show that we can reduce the index $j$ by $1$ with the error $O\left(\frac{1}{Nt^4}\right)$, i.e.
\[
\ntr{H^{(j)}_{\lambda_k}\left(\eta_{\lambda_k,t}\right)H^{(j)}_{\lambda_l}\left(\eta_{\lambda_l,t}\right)} = \ntr{H^{(j-1)}_{\lambda_k}\left(\eta_{\lambda_k,t}\right)H^{(j-1)}_{\lambda_l}\left(\eta_{\lambda_l,t}\right)} + O\left(\frac{1}{Nt^4}\right).
\]
We will consider $j\in\llbracket1,m_R\rrbracket$, the other case is proved the same way, the only modification is the replacement of $u^{(j-1)}$ with $v^{(j-1)}$. For $i=k,l$ consider
\begin{equation*}
\hat{H}^{(j-1)}_{\lambda_i}(\eta_{\lambda_i,t}) = \left[\left(A^{(j-1)}-\lambda_i\right)^\ast(I-u^{(j-1)}u^{(j-1),\ast})\left(A^{(j-1)}-\lambda_i\right) + \eta_{\lambda_i,t}^2\right]^{-1}.
\end{equation*}
By Lemma 2.1 of \cite{MO}, we have
\begin{align*}
\ntr{H^{(j)}_{\lambda_k}\left(\eta_{\lambda_k,t}\right)H^{(j)}_{\lambda_l}\left(\eta_{\lambda_l,t}\right)} &= \ntr{\hat{H}^{(j-1)}_{\lambda_k}\left(\eta_{\lambda_k,t}\right)\hat{H}^{(j-1)}_{\lambda_l}\left(\eta_{\lambda_l,t}\right)} \\
&- \frac{u^{(j-1),\ast}\hat{H}^{(j-1)}_{\lambda_k}(\eta_{\lambda_k,t})\hat{H}^{(j-1)}_{\lambda_l}(\eta_{\lambda_l,t})\hat{H}^{(j-1)}_{\lambda_k}(\eta_{\lambda_k,t})u^{(j-1)}}{Nu^{(j-1),\ast}\hat{H}^{(j-1)}_{\lambda_k}(\eta_{\lambda_k,t})u^{(j-1)}} \\
&- \frac{u^{(j-1),\ast}\hat{H}^{(j-1)}_{\lambda_l}(\eta_{\lambda_l,t})\hat{H}^{(j-1)}_{\lambda_k}(\eta_{\lambda_k,t})\hat{H}^{(j-1)}_{\lambda_l}(\eta_{\lambda_l,t})u^{(j-1)}}{Nu^{(j-1),\ast}\hat{H}^{(j-1)}_{\lambda_l}(\eta_{\lambda_l,t})u^{(j-1)}} \\
&+ \frac{\left|u^{(j-1),\ast}\hat{H}^{(j-1)}_{\lambda_k}(\eta_{\lambda_k,t})\hat{H}^{(j-1)}_{\lambda_l}(\eta_{\lambda_l,t})u^{(j-1)}\right|^2}{Nu^{(j-1),\ast}\hat{H}^{(j-1)}_{\lambda_k}(\eta_{\lambda_k,t})u^{(j-1)}u^{(j-1),\ast}\hat{H}^{(j-1)}_{\lambda_l}(\eta_{\lambda_l,t})u^{(j-1)}}.
\end{align*}
To bound the second and the third term we note that
\[
u^{(j-1),\ast}\hat{H}^{(j-1)}_{\lambda_k}(\eta_{\lambda_k,t})\hat{H}^{(j-1)}_{\lambda_l}(\eta_{\lambda_l,t})\hat{H}^{(j-1)}_{\lambda_k}(\eta_{\lambda_k,t})u^{(j-1)} \lesssim t^{-4} u^{(j-1),\ast}\hat{H}^{(j-1)}_{\lambda_k}(\eta_{\lambda_k,t})u^{(j-1)}
\]
and to bound the last term we note that
\begin{align*}
&\left|u^{(j-1),\ast}\hat{H}^{(j-1)}_{\lambda_k}(\eta_{\lambda_k,t})\hat{H}^{(j-1)}_{\lambda_l}(\eta_{\lambda_l,t})u^{(j-1)}\right|^2 \le \left\|\hat{H}^{(j-1)}_{\lambda_k}(\eta_{\lambda_k,t})u^{(j-1)}\right\|^2 \left\|\hat{H}^{(j-1)}_{\lambda_l}(\eta_{\lambda_l,t})u^{(j-1)}\right\|^2\\
&\lesssim t^{-4} u^{(j-1),\ast}\hat{H}^{(j-1)}_{\lambda_k}(\eta_{\lambda_k,t})u^{(j-1)}u^{(j-1),\ast}\hat{H}^{(j-1)}_{\lambda_l}(\eta_{\lambda_l,t})u^{(j-1)}.
\end{align*}
Then
\begin{align*}
\ntr{H^{(j)}_{\lambda_k}\left(\eta_{\lambda_k,t}\right)H^{(j)}_{\lambda_l}\left(\eta_{\lambda_l,t}\right)} &= \ntr{\hat{H}^{(j-1)}_{\lambda_k}\left(\eta_{\lambda_k,t}\right)\hat{H}^{(j-1)}_{\lambda_l}\left(\eta_{\lambda_l,t}\right)} + O\left(\frac{1}{Nt^4}\right)
\end{align*}
Now we use Woodbury's identity to get
\begin{align*}
&\ntr{\hat{H}^{(j-1)}_{\lambda_k}\left(\eta_{\lambda_k,t}\right)\hat{H}^{(j-1)}_{\lambda_l}\left(\eta_{\lambda_l,t}\right)} = \ntr{H^{(j-1)}_{\lambda_k}\left(\eta_{\lambda_k,t}\right)H^{(j-1)}_{\lambda_l}\left(\eta_{\lambda_l,t}\right)}\\
&+ \frac{u^{(j-1),\ast}(A^{(j-1)}-\lambda_k)H^{(j-1)}_{\lambda_k}(\eta_{\lambda_k,t})H^{(j-1)}_{\lambda_l}(\eta_{\lambda_l,t})H^{(j-1)}_{\lambda_k}(\eta_{\lambda_k,t})(A^{(j-1)}-\lambda_k)^\ast u^{(j-1)}}{N\eta_{\lambda_k,t}^2 u^{(j-1),\ast} \tilde{H}^{(j-1)}_{\lambda_k}\left(\eta_{\lambda_k,t}\right) u^{(j-1)}}\\
&+ \frac{u^{(j-1),\ast}(A^{(j-1)}-\lambda_l)H^{(j-1)}_{\lambda_l}(\eta_{\lambda_l,t})H^{(j-1)}_{\lambda_k}(\eta_{\lambda_k,t})H^{(j-1)}_{\lambda_l}(\eta_{\lambda_l,t})(A^{(j-1)}-\lambda_l)^\ast u^{(j-1)}}{N\eta_{\lambda_l,t}^2 u^{(j-1),\ast} \tilde{H}^{(j-1)}_{\lambda_l}\left(\eta_{\lambda_l,t}\right) u^{(j-1)}}\\
&+ \frac{\left|u^{(j-1),\ast}(A^{(j-1)}-\lambda_k)H^{(j-1)}_{\lambda_k}(\eta_{\lambda_k,t})H^{(j-1)}_{\lambda_l}(\eta_{\lambda_l,t})(A^{(j-1)}-\lambda_l)^\ast u^{(j-1)}\right|^2}{N\eta_{\lambda_k,t}^2\eta_{\lambda_l,t}^2 u^{(j-1),\ast} \tilde{H}^{(j-1)}_{\lambda_k}\left(\eta_{\lambda_k,t}\right) u^{(j-1)} u^{(j-1),\ast} \tilde{H}^{(j-1)}_{\lambda_l}\left(\eta_{\lambda_l,t}\right) u^{(j-1)}}.
\end{align*}
To bound the first two error terms we note that
\begin{align*}
&u^{(j-1),\ast}(A^{(j-1)}-\lambda_k)H^{(j-1)}_{\lambda_k}(\eta_{\lambda_k,t})H^{(j-1)}_{\lambda_l}(\eta_{\lambda_l,t})H^{(j-1)}_{\lambda_k}(\eta_{\lambda_k,t})(A^{(j-1)}-\lambda_k)^\ast u^{(j-1)}\\
&\lesssim t^{-2} u^{(j-1),\ast}(A^{(j-1)}-\lambda_k)H^{(j-1)}_{\lambda_k}(\eta_{\lambda_k,t})^2(A^{(j-1)}-\lambda_k)^\ast u^{(j-1)}\\
&\le t^{-2} u^{(j-1),\ast}\tilde{H}^{(j-1)}_{\lambda_k}(\eta_{\lambda_k,t}) u^{(j-1)}.
\end{align*}
To bound the final error term we use Cauchy-Schwarz as follows.
\begin{align*}
&\left|u^{(j-1),\ast}(A^{(j-1)}-\lambda_k)H^{(j-1)}_{\lambda_k}(\eta_{\lambda_k,t})H^{(j-1)}_{\lambda_l}(\eta_{\lambda_l,t})(A^{(j-1)}-\lambda_k)^\ast u^{(j-1)}\right|^2 \\
&\le \left\|H^{(j-1)}_{\lambda_l}(\eta_{\lambda_k,t})(A^{(j-1)}-\lambda_k)^\ast u^{(j-1)}\right\|^2 \left\|H^{(j-1)}_{\lambda_l}(\eta_{\lambda_l,t})(A^{(j-1)}-\lambda_l)^\ast u^{(j-1)}\right\|^2 \\
&\le u^{(j-1),\ast}\tilde{H}^{(j-1)}_{\lambda_k}(\eta_{\lambda_k,t}) u^{(j-1)} u^{(j-1),\ast}\tilde{H}^{(j-1)}_{\lambda_l}(\eta_{\lambda_l,t}) u^{(j-1)}.
\end{align*}
This concludes the proof.
\end{proof}

\begin{corollary}\label{cor:conc}
For any $j\in\llbracket 1,m_R\rrbracket$ and any $k\in\llbracket 1,m\rrbracket$ we have
\begin{align}
&\left|u^{(j-1),\ast} H^{(j-1)}_{\lambda_k}(\eta_{\lambda_k,t}) u^{(j-1)} - t\ntr{H^{(j-1)}_{\lambda_j}(\eta_{\lambda_j,t})H^{(j-1)}_{\lambda_k}(\eta_{\lambda_k,t})}\right|\nonumber \\
&\prec_{\nu_j} \frac{1}{\sqrt{N}t}\ntr{H^{(j-1)}_{\lambda_j}(\eta_{\lambda_j,t})H^{(j-1)}_{\lambda_k}(\eta_{\lambda_k,t})}^{\frac12},\label{eq:conc1}\\
&\left|u^{(j-1),\ast} \tilde{H}^{(j-1)}_{\lambda_k}(\eta_{\lambda_k,t}) u^{(j-1)} - t\ntr{H^{(j-1)}_{\lambda_j}(\eta_{\lambda_j,t})\tilde{H}^{(j-1)}_{\lambda_k}(\eta_{\lambda_k,t})}\right| \nonumber\\
&\prec_{\nu_j} \frac{1}{\sqrt{N}t}\ntr{H^{(j-1)}_{\lambda_j}(\eta_{\lambda_j,t})\tilde{H}^{(j-1)}_{\lambda_k}(\eta_{\lambda_k,t})}^{\frac12}.\label{eq:conc3}
\end{align}
For any $j\in\llbracket m_R+1,m\rrbracket$ and any $k\in\llbracket 1,m\rrbracket$ we have
\begin{align*}
&\left|u^{(j-1),\ast} H^{(j-1)}_{\lambda_k}(\eta_{\lambda_k,t}) u^{(j-1)} - t\ntr{\tilde{H}^{(j-1)}_{\lambda_j}(\eta_{\lambda_j,t})H^{(j-1)}_{\lambda_k}(\eta_{\lambda_k,t})}\right|\nonumber \\
&\prec_{\nu_j} \frac{1}{\sqrt{N}t}\ntr{\tilde{H}^{(j-1)}_{\lambda_j}(\eta_{\lambda_j,t})H^{(j-1)}_{\lambda_k}(\eta_{\lambda_k,t})}^{\frac12},\\
&\left|u^{(j-1),\ast} \tilde{H}^{(j-1)}_{\lambda_k}(\eta_{\lambda_k,t}) u^{(j-1)} - t\ntr{\tilde{H}^{(j-1)}_{\lambda_j}(\eta_{\lambda_j,t})\tilde{H}^{(j-1)}_{\lambda_k}(\eta_{\lambda_k,t})}\right| \nonumber\\
&\prec_{\nu_j} \frac{1}{\sqrt{N}t}\ntr{\tilde{H}^{(j-1)}_{\lambda_j}(\eta_{\lambda_j,t})\tilde{H}^{(j-1)}_{\lambda_k}(\eta_{\lambda_k,t})}^{\frac12}.
\end{align*}
As a consequence,
\begin{align}
u^{(j-1),\ast} H^{(j-1)}_{\lambda_k}(\eta_{\lambda_k,t}) u^{(j-1)} &\succ_{\nu_j} t^{-1},  &
u^{(j-1),\ast} \tilde{H}^{(j-1)}_{\lambda_k}(\eta_{\lambda_k,t}) u^{(j-1)}&\succ_{\nu_j} \left(t+|\lambda_k-\lambda_l|^2\right)^{-1}, \label{eq:uHu-lower-bound}\\
v^{(j-1),\ast} \tilde{H}^{(j-1)}_{\lambda_k}(\eta_{\lambda_k,t}) v^{(j-1)}&\succ_{\nu_j} t^{-1},  & v^{(j-1),\ast} H^{(j-1)}_{\lambda_k}(\eta_{\lambda_k,t}) v^{(j-1)} &\succ_{\nu_j} \left(t+|\lambda_k-\lambda_l|^2\right)^{-1}.\label{eq:vHv-lower-bound}
\end{align}
Additionally, for any deterministic unit vector $\mathbf{w}\in\mathbb{S}^{N-j}$, we have
\begin{align}\label{eq:conc2}
&\left|\mathbf{w}^\ast \tilde{H}^{(j-1)}_{\lambda_k}(\eta_{\lambda_k,t})(A^{(j-1)}-\lambda_k)u^{(j-1)}\right|^2\prec_{\nu_j} \frac{1}{Nt}\mathbf{w}^\ast \tilde{H}^{(j-1)}_{\lambda_k}(\eta_{\lambda_k,t})\mathbf{w},\\
&\left|\mathbf{w}^\ast H^{(j-1)}_{\lambda_k}(\eta_{\lambda_k,t})(A^{(j-1)}-\lambda_k)u^{(j-1)}\right|^2\prec_{\nu_j} \frac{1}{Nt}\mathbf{w}^\ast H^{(j-1)}_{\lambda_k}(\eta_{\lambda_k,t})\mathbf{w},\\
&\left|\mathbf{w}^\ast \tilde{H}^{(j-1)}_{\lambda_k}(\eta_{\lambda_k,t})u^{(j-1)}\right|^2\prec_{\nu_j} \frac{1}{Nt^3}\mathbf{w}^\ast \tilde{H}^{(j-1)}_{\lambda_k}(\eta_{\lambda_k,t})\mathbf{w},\label{eq:conc5}\\
&\left|\mathbf{w}^\ast H^{(j-1)}_{\lambda_k}(\eta_{\lambda_k,t})u^{(j-1)}\right|^2\prec_{\nu_j} \frac{1}{Nt^3}\mathbf{w}^\ast H^{(j-1)}_{\lambda_k}(\eta_{\lambda_k,t})\mathbf{w},\\
&\left|\mathbf{w}^\ast u^{(j-1)}\right|^2 \prec \frac{t}{N}\mathbf{w}^\ast H_{\lambda_j}^{(j-1)}(\eta_{\lambda_j,t})\mathbf{w}\label{eq:conc4}
\end{align}
for $j\in\llbracket 1,m_R\rrbracket$ uniformly in $\mathbf{w}$ and 
\begin{align}\label{eq:conc2v}
&\left|\mathbf{w}^\ast H^{(j-1)}_{\lambda_k}(\eta_{\lambda_k,t})(A^{(j-1)}-\lambda_k)^\ast v^{(j-1)}\right|^2\prec_{\nu_j} \frac{1}{Nt}\mathbf{w}^\ast H^{(j-1)}_{\lambda_k}(\eta_{\lambda_k,t})\mathbf{w},\\
&\left|\mathbf{w}^\ast \tilde{H}^{(j-1)}_{\lambda_k}(\eta_{\lambda_k,t})(A^{(j-1)}-\lambda_k)^\ast v^{(j-1)}\right|^2\prec_{\nu_j} \frac{1}{Nt}\mathbf{w}^\ast \tilde{H}^{(j-1)}_{\lambda_k}(\eta_{\lambda_k,t})\mathbf{w},\\
&\left|\mathbf{w}^\ast H^{(j-1)}_{\lambda_k}(\eta_{\lambda_k,t})v^{(j-1)}\right|^2\prec_{\nu_j} \frac{1}{Nt^3}\mathbf{w}^\ast H^{(j-1)}_{\lambda_k}(\eta_{\lambda_k,t})\mathbf{w},\label{eq:conc5v}\\
&\left|\mathbf{w}^\ast \tilde{H}^{(j-1)}_{\lambda_k}(\eta_{\lambda_k,t})v^{(j-1)}\right|^2\prec_{\nu_j} \frac{1}{Nt^3}\mathbf{w}^\ast \tilde{H}^{(j-1)}_{\lambda_k}(\eta_{\lambda_k,t})\mathbf{w},\\
&\left|\mathbf{w}^\ast v^{(j-1)}\right|^2 \prec \frac{t}{N}\mathbf{w}^\ast \tilde{H}_{\lambda_j}^{(j-1)}(\eta_{\lambda_j,t})\mathbf{w}\label{eq:conc4v}
\end{align}
for $j\in\llbracket m_R+1,m\rrbracket$ uniformly in $\mathbf{w}$.
\end{corollary}

\begin{proof}
We prove the concentration bounds with respect to the right eigenvector measures $\nu_j$ for $j\in\llbracket1,m_R\rrbracket$. The concentration bounds with respect to the left eigenvectors are proved the same way after replacing $A$ with $A^\ast$. To show \eqref{eq:conc1}, we note that 
\[
\left\|\sqrt{H^{(j-1)}_{\lambda_j}(\eta_{\lambda_j,t})}H^{(j-1)}_{\lambda_k}(\eta_{\lambda_k,t})\sqrt{H^{(j-1)}_{\lambda_j}(\eta_{\lambda_j,t})}\right\| \le \eta_{\lambda_j,t}^{-2}\eta_{\lambda_k,t}^{-2}\lesssim t^{-4}
\]
and
\[
\tr\left(H^{(j-1)}_{\lambda_j}(\eta_{\lambda_j,t})H^{(j-1)}_{\lambda_k}(\eta_{\lambda_k,t})\right)^2 \lesssim \frac{N}{t^4}\ntr{H^{(j-1)}_{\lambda_j}(\eta_{\lambda_j,t})H^{(j-1)}_{\lambda_k}(\eta_{\lambda_k,t})}.
\]
Then \eqref{eq:conc1} follows by Lemma \ref{lem:master-concentration} and Lemma \ref{lem:two-res-removing-j}. The proof of \eqref{eq:conc3} is analogous.
We now need to address 
\begin{align*}
&\left|\mathbf{w}^\ast \tilde{H}^{(j-1)}_{\lambda_k}(\eta_{\lambda_k,t})(A^{(j-1)}-\lambda_k)u^{(j-1)}\right|^2\\
&=u^{(j-1),\ast}(A^{(j-1)}-\lambda_k)^{\ast}\tilde{H}^{(j-1)}_{\lambda_k}(\eta_{\lambda_k,t})\mathbf{w}\mathbf{w}^{\ast}\tilde{H}^{(j-1)}_{\lambda_k}(\eta_{\lambda_k,t})(A^{(j-1)}-\lambda_k)u^{(j-1)}.
\end{align*}
We use Lemma \ref{lem:master-concentration} for the rank $1$ matrix
\begin{equation*}
B =(A-\lambda_{2})^{\ast}\tilde{H}_{\lambda_{2}}(\eta_{\lambda_{2},t})\mathbf{w}\mathbf{w}^{\ast}\tilde{H}_{\lambda_{2}}(\eta_{\lambda_{2},t})(A-\lambda_{2}).
\end{equation*}
It remains to compute
\begin{align*}
&\frac{t}{N}\tr H^{(j-1)}_{\lambda_j}(\eta_{\lambda_j,t})B \\
&=\frac{t}{N}\mathbf{w}^{\ast}\tilde{H}^{(j-1)}_{\lambda_k}(\eta_{\lambda_k,t})(A^{(j-1)}-\lambda_k) H^{(j-1)}_{\lambda_j}(\eta_{\lambda_j,t})(A^{(j-1)}-\lambda_k)^{\ast}\tilde{H}^{(j-1)}_{\lambda_k}(\eta_{\lambda_k,t})\mathbf{w}\\
&\lesssim \frac{1}{Nt}\mathbf{w}^{\ast}\tilde{H}^{(j-1)}_{\lambda_k}(\eta_{\lambda_k,t})(A^{(j-1)}-\lambda_k)(A^{(j-1)}-\lambda_k)^{\ast}\tilde{H}^{(j-1)}_{\lambda_k}(\eta_{\lambda_k,t})\mathbf{w}\\
&=\frac{1}{Nt}\mathbf{w}^{\ast}\tilde{H}^{(j-1)}_{\lambda_k}(\eta_{\lambda_k,t})\mathbf{w}-\frac{1}{Nt}\eta_{\lambda_{k},t}^{2}\mathbf{w}^{\ast}\tilde{H}^{(j-1)}_{\lambda_k}(\eta_{\lambda_k,t})^{2}\mathbf{w}\le \frac{1}{Nt}\mathbf{w}^{\ast}\tilde{H}^{(j-1)}_{\lambda_k}(\eta_{\lambda_k,t})\mathbf{w}.
\end{align*}
Thus, Lemma \ref{lem:master-concentration} implies \eqref{eq:conc2}.
Now, we deal with 
\begin{align*}
\left|\mathbf{w}^\ast \tilde{H}^{(j-1)}_{\lambda_k}(\eta_{\lambda_k,t})u^{(j-1)}\right|^2 =u^{(j-1),\ast}\tilde{H}^{(j-1)}_{\lambda_k}(\eta_{\lambda_k,t})\mathbf{w}\mathbf{w}^{\ast}\tilde{H}^{(j-1)}_{\lambda_k}(\eta_{\lambda_k,t})u^{(j-1)}.
\end{align*}
We consider a rank $1$ matrix
\[
B = \tilde{H}^{(j-1)}_{\lambda_k}(\eta_{\lambda_k,t})\mathbf{w}\mathbf{w}^{\ast}\tilde{H}^{(j-1)}_{\lambda_k}(\eta_{\lambda_k,t}).
\]
Then it is easy to check that
\begin{align*}
\frac{t}{N}\tr H^{(j-1)}_{\lambda_j}(\eta_{\lambda_j,t}) B &= \frac{t}{N} \mathbf{w}^{\ast}\tilde{H}^{(j-1)}_{\lambda_k}(\eta_{\lambda_k,t}) H^{(j-1)}_{\lambda_j}(\eta_{\lambda_j,t}) \tilde{H}^{(j-1)}_{\lambda_k}(\eta_{\lambda_k,t})\mathbf{w}\\
&\lesssim \frac{1}{Nt^3} \mathbf{w}^{\ast}\tilde{H}^{(j-1)}_{\lambda_k}(\eta_{\lambda_k,t})\mathbf{w}
\end{align*}
Plugging this into Lemma \ref{lem:master-concentration} gives us \eqref{eq:conc5}. Finally, the bound \eqref{eq:conc4} follows directly from Lemma \ref{lem:master-concentration} with $B = \mathbf{w}\mathbf{w}^\ast$.
\end{proof}

\begin{lemma}\label{lem:wHw}
For any $j\in\llbracket 1,m-1\rrbracket$, any $k\in\llbracket 1,m\rrbracket$ and any deterministic vector $\mathbf{w}\in \C^N$ we have 
\begin{align}
\mathbf{w}^\ast U_{j}^\ast \begin{pmatrix}0&0\\0&H^{(j)}_{\lambda_k}(\eta_{\lambda_k,t})\end{pmatrix}U_{j}\mathbf{w} &= \mathbf{w}^\ast H_{\lambda_k}(\eta_{\lambda_k,t})\mathbf{w}\left(1+ \oh_{\nu_{j\rightarrow 1}} \left(\frac{1}{Nt^3}\right)\right),\\
\mathbf{w}^\ast U_{j}^\ast \begin{pmatrix}0&0\\0&\tilde{H}^{(j)}_{\lambda_k}(\eta_{\lambda_k,t})\end{pmatrix}U_{j}\mathbf{w} &= \mathbf{w}^\ast\tilde{H}_{\lambda_k}(\eta_{\lambda_k,t})\mathbf{w}\left(1+ \oh_{\nu_{j\rightarrow 1}} \left(\frac{1}{Nt^3}\right)\right).
\end{align}
\end{lemma}

\begin{proof}
We show the first estimate by induction over $j$. The other estimate is proved similarly. It is sufficient to show that for any $\mathbf{w}\in \mathbb{S}^{N-j}$ we have
\begin{align*}
\mathbf{w}^\ast R_{j}\left(u^{(j-1)}\right) \begin{pmatrix}0&0\\0&H^{(j)}_{\lambda_k}(\eta_{\lambda_k,t})\end{pmatrix}R_{j}\left(u^{(j-1)}\right)\mathbf{w} = \mathbf{w}^\ast H^{(j-1)}_{\lambda_k}(\eta_{\lambda_k,t})\mathbf{w}\left(1+ \oh_{\nu_j} \left(\frac{1}{Nt^3}\right)\right)
\end{align*}
uniformly in $\mathbf{w}$. Define
\begin{equation*}
\hat{H}^{(j-1)}_{\lambda_k}(\eta_{\lambda_k,t}) = \left[\left(A^{(j-1)}-\lambda_k\right)^\ast(I-u^{(j-1)}u^{(j-1),\ast})\left(A^{(j-1)}-\lambda_k\right) + \eta_{\lambda_k,t}^2\right]^{-1}.
\end{equation*}
By Lemma 2.1 of \cite{MO}, we have
\begin{align*}
&\mathbf{w}^\ast R_{j}\left(u^{(j-1)}\right) \begin{pmatrix}0&0\\0&H^{(j)}_{\lambda_k}(\eta_{\lambda_k,t})\end{pmatrix}R_{j}\left(u^{(j-1)}\right) \mathbf{w} \\
&= \mathbf{w}^\ast\hat{H}^{(j-1)}_{\lambda_k}(\eta_{\lambda_k,t})\mathbf{w} - \frac{\left|\mathbf{w}^\ast\hat{H}^{(j-1)}_{\lambda_k}(\eta_{\lambda_k,t})u^{(j-1)}\right|^2}{u^{(j-1),\ast}\hat{H}^{(j-1)}_{\lambda_k}(\eta_{\lambda_k,t})u^{(j-1)}}.
\end{align*}
We can use the Woodbury's identity to express $\hat{H}^{(j-1)}_{\lambda_k}(\eta_{\lambda_k,t})$ in terms of $H^{(j-1)}_{\lambda_k}(\eta_{\lambda_k,t})$ as follows.
\begin{align}
&\hat{H}^{(j-1)}_{\lambda_k}(\eta_{\lambda_k,t}) \nonumber\\
&= H^{(j-1)}_{\lambda_k}(\eta_{\lambda_k,t}) + \frac{H^{(j-1)}_{\lambda_k}(\eta_{\lambda_k,t})\left(A^{(j-1)}-\lambda_k\right)^\ast u^{(j-1)}u^{(j-1),\ast}\left(A^{(j-1)}-\lambda_k\right)H^{(j-1)}_{\lambda_k}(\eta_{\lambda_k,t})}{1-u^{(j-1),\ast}\left(A^{(j-1)}-\lambda_k\right)H^{(j-1)}_{\lambda_k}(\eta_{\lambda_k,t})\left(A^{(j-1)}-\lambda_k\right)^\ast u^{(j-1)}}\nonumber\\
&= H^{(j-1)}_{\lambda_k}(\eta_{\lambda_k,t}) + \frac{H^{(j-1)}_{\lambda_k}(\eta_{\lambda_k,t})\left(A^{(j-1)}-\lambda_k\right)^\ast u^{(j-1)}u^{(j-1),\ast}\left(A^{(j-1)}-\lambda_k\right)H^{(j-1)}_{\lambda_k}(\eta_{\lambda_k,t})}{\eta_{\lambda_k,t}^2 u^{(j-1),\ast}\tilde{H}^{(j-1)}_{\lambda_k}(\eta_{\lambda_k,t}) u^{(j-1)}}.\label{eq:woodbury}
\end{align}
Using this identity and Corollary \ref{cor:conc} we see that
\begin{align*}
\mathbf{w}^\ast \hat{H}^{(j-1)}_{\lambda_k}(\eta_{\lambda_k,t})\mathbf{w}- \mathbf{w}^\ast H^{(j-1)}_{\lambda_k}(\eta_{\lambda_k,t})\mathbf{w}&= \frac{\left|\mathbf{w}^\ast H^{(j-1)}_{\lambda_k}(\eta_{\lambda_k,t})\left(A^{(j-1)}-\lambda_k\right)^\ast u^{(j-1)}\right|^2}{\eta_{\lambda_k,t}^2 u^{(j-1),\ast}\tilde{H}^{(j-1)}_{\lambda_k}(\eta_{\lambda_k,t}) u^{(j-1)}}\\
&\prec_{\nu_j} \frac{1}{Nt^3} \mathbf{w}^\ast H^{(j-1)}_{\lambda_k}(\eta_{\lambda_k,t})\mathbf{w}.
\end{align*}
Using \eqref{eq:woodbury} again we get
\begin{align*}
&\left|\mathbf{w}^\ast\hat{H}^{(j-1)}_{\lambda_k}(\eta_{\lambda_k,t})u^{(j-1)}\right|^2 \lesssim \left|\mathbf{w}^\ast H^{(j-1)}_{\lambda_k}(\eta_{\lambda_k,t})u^{(j-1)}\right|^2 \\
&+ \frac{\left|\mathbf{w}^\ast H^{(j-1)}_{\lambda_k}(\eta_{\lambda_k,t})\left(A^{(j-1)}-\lambda_k\right)^\ast u^{(j-1)}\right|^2 \left|u^{(j-1),\ast} H^{(j-1)}_{\lambda_k}(\eta_{\lambda_k,t})\left(A^{(j-1)}-\lambda_k\right)^\ast u^{(j-1)} \right|^2}{\eta_{\lambda_k,t}^4 \left(u^{(j-1),\ast}\tilde{H}^{(j-1)}_{\lambda_k}(\eta_{\lambda_k,t}) u^{(j-1)}\right)^2} 
\end{align*}
and
\begin{align*}
&u^{(j-1),\ast}\hat{H}^{(j-1)}_{\lambda_k}(\eta_{\lambda_k,t})u^{(j-1)}\\
&\ge \max\left\{u^{(j-1),\ast}H^{(j-1)}_{\lambda_k}(\eta_{\lambda_k,t})u^{(j-1)}; \frac{\left|u^{(j-1),\ast} H^{(j-1)}_{\lambda_k}(\eta_{\lambda_k,t})\left(A^{(j-1)}-\lambda_k\right)^\ast u^{(j-1)} \right|^2}{\eta_{\lambda_k,t}^2 u^{(j-1),\ast}\tilde{H}^{(j-1)}_{\lambda_k}(\eta_{\lambda_k,t}) u^{(j-1)}} \right\}.
\end{align*}
Thus, combining these two bounds with Corollary \ref{cor:conc}, we get
\begin{align*}
\frac{\left|\mathbf{w}^\ast\hat{H}^{(j-1)}_{\lambda_k}(\eta_{\lambda_k,t})u^{(j-1)}\right|^2}{u^{(j-1),\ast}\hat{H}^{(j-1)}_{\lambda_k}(\eta_{\lambda_k,t})u^{(j-1)}} &\lesssim \frac{\left|\mathbf{w}^\ast H^{(j-1)}_{\lambda_k}(\eta_{\lambda_k,t})u^{(j-1)}\right|^2}{u^{(j-1),\ast}H^{(j-1)}_{\lambda_k}(\eta_{\lambda_k,t})u^{(j-1)}} \\
&+ \frac{\left|\mathbf{w}^\ast H^{(j-1)}_{\lambda_k}(\eta_{\lambda_k,t})\left(A^{(j-1)}-\lambda_k\right)^\ast u^{(j-1)}\right|^2}{\eta_{\lambda_k,t}^2 u^{(j-1),\ast}\tilde{H}^{(j-1)}_{\lambda_k}(\eta_{\lambda_k,t}) u^{(j-1)}}\\
&\prec_{\nu_j} \frac{1}{Nt^3} \mathbf{w}^\ast H^{(j-1)}_{\lambda_k}(\eta_{\lambda_k,t})\mathbf{w},
\end{align*}
which completes the proof.
\end{proof}

\begin{corollary}\label{cor:vector-bound-first-entry}
Suppose $\mathbf{w}\in\C^N$ is a deterministic unit vector. Then for any $j\in\llbracket 1,m_R\rrbracket$ we have
\begin{align*}
\left|\mathbf{w}^\ast U_{j-1} \begin{pmatrix}0\\u^{(j-1)}\end{pmatrix}\right|^2 \prec \frac{t}{N}\mathbf{w}^\ast H_{\lambda_j}(\eta_{\lambda_j,t})\mathbf{w}.
\end{align*}
For any $j\in\llbracket m_R+1,m\rrbracket$ we have
\begin{align*}
\left|\mathbf{w}^\ast U_{j-1} \begin{pmatrix}0\\v^{(j-1)}\end{pmatrix}\right|^2 \prec \frac{t}{N}\mathbf{w}^\ast \tilde{H}_{\lambda_j}(\eta_{\lambda_j,t})\mathbf{w}.
\end{align*}
\end{corollary}

\begin{proof}
Follows directly from the bounds \eqref{eq:conc4} and \eqref{eq:conc4v} from Corollary \ref{cor:conc} and Lemma \ref{lem:wHw}.
\end{proof}
\subsection{Proof of Lemma \ref{lemma:simplify}}
First, we get the high probability bounds on some quantities involved in the proof of Lemma \ref{lemma:simplify}.

\begin{lemma}\label{lem:evec-first-entry-error}
For any $k,j\in\llbracket 1,m_R\rrbracket$ such that $k\le j$ we have
\begin{align*}
\left|w_k^\ast \prod_{l=k+1}^{j-1} \left(I_{l-k-1}\oplus R_l\left(u^{(l-1)}\right)\right)\begin{pmatrix}0\\u^{(j-1)}\end{pmatrix}\right|^2 \prec_{\omega_k,\nu_j}  \frac{t+|\lambda_k-\lambda_j|^2}{Nt}.
\end{align*}
For any $k,j\in\llbracket m_R+1,m\rrbracket$ such that $k\le j$ we have
\begin{align*}
\left|w_k^\ast \prod_{l=k+1}^{j-1} \left(I_{l-k-1}\oplus R_l\left(v^{(l-1)}\right)\right)\begin{pmatrix}0\\v^{(j-1)}\end{pmatrix}\right|^2 \prec_{\omega_k,\nu_j} \frac{t+|\lambda_k-\lambda_j|^2}{Nt}.
\end{align*}
\end{lemma}

\begin{proof}
We prove the first right eigenvector bound. The other bound is proved similarly. Since the measure $\omega_k(w_k)$ is Gaussian with mean $b_k$ and variance $\frac{t}{N}$, we have $\|w_k-b_k\|\prec t^{\frac12}N^{-\frac12}$. Thus
\begin{align*}
&\left|w_k^\ast \prod_{l=k+1}^{j-1} \left(I_{l-k-1}\oplus R_l\left(u^{(l-1)}\right)\right)\begin{pmatrix}0\\u^{(j-1)}\end{pmatrix}\right| \\
&= \left|b_k^\ast \prod_{l=k+1}^{j-1} \left(I_{l-k-1}\oplus R_l\left(u^{(l-1)}\right)\right)\begin{pmatrix}0\\u^{(j-1)}\end{pmatrix}\right| + \oh_{\omega_k}(t^{\frac12}N^{-\frac12})\\
&= \left|u^{(k-1),\ast}A^{(k-1)}R_{k}(u^{(k-1)})\prod_{l=k+1}^{j-1} \left(I_{l-k}\oplus R_l\left(u^{(l-1)}\right)\right)\begin{pmatrix}0\\u^{(j-1)}\end{pmatrix}\right| + \oh_{\omega_k}(t^{\frac12}N^{-\frac12})\\
&= \left|u^{(k-1),\ast}\left(A^{(k-1)}-\lambda_j\right)\prod_{l=k}^{j-1} \left(I_{l-k}\oplus R_l\left(u^{(l-1)}\right)\right)\begin{pmatrix}0\\u^{(j-1)}\end{pmatrix}\right| + \oh_{\omega_k}(t^{\frac12}N^{-\frac12}).
\end{align*}
For brevity we denote 
\[
U_{k,j-1} = \prod_{l=k}^{j-1} \left(I_{l-k}\oplus R_l\left(u^{(l-1)}\right)\right).
\]
Now we use Lemma \ref{lem:master-concentration} with
\[
B = U_{k,j-1}^\ast \left(A^{(k-1)}-\lambda_j\right)^\ast u^{(k-1)} u^{(k-1),\ast}\left(A^{(k-1)}-\lambda_j\right) U_{k,j-1}.
\]
Since this is a rank $1$ matrix, by Lemma \ref{lem:master-concentration} we have
\begin{align*}
&\left|u^{(k-1),\ast}\left(A^{(k-1)}-\lambda_j\right)U_{k,j-1}\begin{pmatrix}0\\u^{(j-1)}\end{pmatrix}\right|^2 \prec_{\nu_j} \frac{t}{N}\tr B\begin{pmatrix}0&0\\0& H^{(j-1)}_{\lambda_j}\left(\eta_{\lambda_j,t}\right)\end{pmatrix} \\
&= \frac{t}{N}u^{(k-1),\ast}\left(A^{(k-1)}-\lambda_j\right) U_{k,j-1} \begin{pmatrix}0&0\\0& H^{(j-1)}_{\lambda_j}\left(\eta_{\lambda_j,t}\right)\end{pmatrix} U_{k,j-1}^\ast \left(A^{(k-1)}-\lambda_j\right)^\ast u^{(k-1)} \\
&\le \frac{t}{N}u^{(k-1),\ast}\left(A^{(k-1)}-\lambda_j\right) U_{k,j-2} \begin{pmatrix}0&0\\0& H^{(j-2)}_{\lambda_j}\left(\eta_{\lambda_j,t}\right)\end{pmatrix} U_{k,j-2}^\ast \left(A^{(k-1)}-\lambda_j\right)^\ast u^{(k-1)} \\
&+ \frac{t\left|u^{(k-1),\ast}\left(A^{(k-1)}-\lambda_j\right) U_{k,j-2} \begin{pmatrix}0&0\\0& H^{(j-2)}_{\lambda_j}\left(\eta_{\lambda_j,t}\right)\left(A^{(j-2)}-\lambda_j\right)^\ast u^{(j-2)}\end{pmatrix}\right|^2}{N\eta_{\lambda_j,t}^2 u^{(j-2),\ast}\tilde{H}^{(j-2)}_{\lambda_j}\left(\eta_{\lambda_j,t}\right)u^{(j-2)}}.
\end{align*}
where in the last step we used Lemma 2.1 of \cite{MO} and Woodbury's identity. As long as $j\ge k+2$, we can use concentration in $u^{(j-2)}$ to bound the numerator of the final term by
\begin{align*}
&\frac{t}{N} u^{(k-1),\ast}\left(A^{(k-1)}-\lambda_j\right) U_{k,j-2} \\
&\times\begin{pmatrix}0&0\\0& H^{(j-2)}_{\lambda_j}\left(\eta_{\lambda_j,t}\right)\left(A^{(j-2)}-\lambda_j\right)^\ast H^{(j-2)}_{\lambda_{j-1}}\left(\eta_{\lambda_{j-1},t}\right)\left(A^{(j-2)}-\lambda_j\right) H^{(j-2)}_{\lambda_j}\left(\eta_{\lambda_j,t}\right) \end{pmatrix} \\
&\times U_{k,j-2}^\ast \left(A^{(k-1)}-\lambda_j\right)^\ast u^{(k-1)}
\end{align*}
in the sense of stochastic domination with respect to $\nu_{j-1}$. We deal with the resolvent product in the middle by bounding the middle $H$ with a trivial norm bound as follows.
\begin{align*}
&H^{(j-2)}_{\lambda_j}\left(\eta_{\lambda_j,t}\right)\left(A^{(j-2)}-\lambda_j\right)^\ast H^{(j-2)}_{\lambda_{j-1}}\left(\eta_{\lambda_{j-1},t}\right)\left(A^{(j-2)}-\lambda_j\right) H^{(j-2)}_{\lambda_j}\left(\eta_{\lambda_j,t}\right)\\
&\le \eta_{\lambda_{j-1},t}^{-2} H^{(j-2)}_{\lambda_j}\left(\eta_{\lambda_j,t}\right)\left(A^{(j-2)}-\lambda_j\right)^\ast \left(A^{(j-2)}-\lambda_j\right) H^{(j-2)}_{\lambda_j}\left(\eta_{\lambda_j,t}\right)\\
&\le \eta_{\lambda_{j-1},t}^{-2} H^{(j-2)}_{\lambda_j}\left(\eta_{\lambda_j,t}\right).
\end{align*}
Then putting the previous three displays together we get
\begin{align*}
&\left|u^{(k-1),\ast}\left(A^{(k-1)}-\lambda_j\right)U_{k,j-1}\begin{pmatrix}0\\u^{(j-1)}\end{pmatrix}\right|^2 \prec_{\nu_j,\nu_{j-1}}\\
&\frac{t}{N}u^{(k-1),\ast}\left(A^{(k-1)}-\lambda_j\right) U_{k,j-2} \begin{pmatrix}0&0\\0& H^{(j-2)}_{\lambda_j}\left(\eta_{\lambda_j,t}\right)\end{pmatrix} U_{k,j-2}^\ast \left(A^{(k-1)}-\lambda_j\right)^\ast u^{(k-1)}\\
&\times \left(1+ N^{-1}t^{-3}\left(u^{(j-2),\ast}\tilde{H}^{(j-2)}_{\lambda_j}\left(\eta_{\lambda_j,t}\right)u^{(j-2)}\right)^{-1}\right).
\end{align*}
Now using the concentration lower bound \eqref{eq:uHu-lower-bound}, we get
\begin{align*}
&\left|u^{(k-1),\ast}\left(A^{(k-1)}-\lambda_j\right)U_{k,j-1}\begin{pmatrix}0\\u^{(j-1)}\end{pmatrix}\right|^2 \prec_{\nu_{j},\nu_{j-1}}\\
&\frac{t}{N}u^{(k-1),\ast}\left(A^{(k-1)}-\lambda_j\right) U_{k,j-2} \begin{pmatrix}0&0\\0& H^{(j-2)}_{\lambda_j}\left(\eta_{\lambda_j,t}\right)\end{pmatrix} U_{k,j-2}^\ast \left(A^{(k-1)}-\lambda_j\right)^\ast u^{(k-1)}.
\end{align*}
Now we repeat these steps until we get
\begin{align*}
&\left|u^{(k-1),\ast}\left(A^{(k-1)}-\lambda_j\right)U_{k,j-1}\begin{pmatrix}0\\u^{(j-1)}\end{pmatrix}\right|^2 \prec_{\nu_{j\rightarrow k+1}}\\
&\frac{t}{N}u^{(k-1),\ast}\left(A^{(k-1)}-\lambda_j\right) R_k\left(u^{(k-1)}\right) \begin{pmatrix}0&0\\0& H^{(k)}_{\lambda_j}\left(\eta_{\lambda_j,t}\right)\end{pmatrix} R_k\left(u^{(k-1)}\right) \left(A^{(k-1)}-\lambda_j\right)^\ast u^{(k-1)}.
\end{align*}
We use Lemma 2.1 of \cite{MO} and Woodbury's identity again to get
\begin{align*}
&\left|u^{(k-1),\ast}\left(A^{(k-1)}-\lambda_j\right)U_{k,j-1}\begin{pmatrix}0\\u^{(j-1)}\end{pmatrix}\right|^2 \\
&\prec_{\nu_{j\rightarrow k+1}}\frac{t}{N}u^{(k-1),\ast}\left(A^{(k-1)}-\lambda_j\right) H^{(k-1)}_{\lambda_j}\left(\eta_{\lambda_j,t}\right) \left(A^{(k-1)}-\lambda_j\right)^\ast u^{(k-1)}\\
&+ \frac{t\left(u^{(k-1),\ast}\left(A^{(k-1)}-\lambda_j\right) H^{(k-1)}_{\lambda_j}\left(\eta_{\lambda_j,t}\right) \left(A^{(k-1)}-\lambda_j\right)^\ast u^{(k-1)}\right)^2}{N\eta_{\lambda_j,t}^2 u^{(k-1),\ast}\tilde{H}^{(k-1)}_{\lambda_j}\left(\eta_{\lambda_j,t}\right)u^{(k-1)}}\\
&\prec_{\nu_{j\rightarrow k}} \frac{t+|\lambda_k-\lambda_j|^2}{Nt}. 
\end{align*}
This concludes the proof.
\end{proof}
Now we are ready to prove Lemma \ref{lemma:simplify}.

\begin{proof}[Proof of Lemma \ref{lemma:simplify}]
First, we note that if $M^{(j)}$ has eigenvalue $\lambda$ with right eigenvector $\tilde{u}_j$ and left eigenvector $\tilde{v}_j$, then $M^{(j-1)}$ also has eigenvalue $\lambda$ with right and left eigenvectors given by
\begin{align}
\tilde{u}_{j-1} &= R_j\left(u^{(j-1)}\right)\begin{pmatrix}\frac{w_j^\ast \tilde{u}_j}{\lambda-\lambda_j}\\ \tilde{u}_j\end{pmatrix}, & \tilde{v}_{j-1} &= R_j\left(u^{(j-1)}\right)\begin{pmatrix}0\\ \tilde{v}_j\end{pmatrix}, & j&\in\llbracket 1,m_R\rrbracket;\label{eq:evec-conversion-1}\\
\tilde{u}_{j-1} &= R_j\left(v^{(j-1)}\right)\begin{pmatrix}0\\ \tilde{u}_j\end{pmatrix}, & \tilde{v}_{j-1} &= R_j\left(v^{(j-1)}\right)\begin{pmatrix}\frac{w_j^\ast \tilde{v}_j}{\lambda-\lambda_j}\\ \tilde{v}_j\end{pmatrix}, & j&\in\llbracket m_R+1,m\rrbracket.\label{eq:evec-conversion-2}
\end{align}
Using this observation, we can express the left and right eigenvectors of $M$ denoted by $u_1,\ldots,u_{m_R}$, $v_{m_R+1},\ldots, v_{m}$ through $u, \ldots, u^{(m_R-1)}, v^{(m_R)}, \ldots, v^{(m-1)}$. For instance, take $j\in \llbracket 1,m_R\rrbracket$. We consecutively apply the identities \eqref{eq:evec-conversion-1} to compute the right eigenvectors of $M^{(j-2)},\ldots,M^{(1)},M$ corresponding to the eigenvalue $\lambda_j$. In the end, we get
\begin{align}\label{eq:uj-initial}
u_j = U_{j-1}\begin{pmatrix}0\\u^{(j-1)}\end{pmatrix} + \sum_{s=1}^{j-1} \sum_{k=1}^{j-s} \sum_{s= i_1<i_2<\ldots<i_{k+1}=j} a(i_1,\ldots, i_{k+1}) U_{s-1}\begin{pmatrix}0\\u^{(s-1)}\end{pmatrix},
\end{align}
where the scalar coefficients are given by
\begin{align*}
a(i_1,\ldots, i_{k+1}) = \prod_{n=1}^{k} (\lambda_{i_n}-\lambda_{i_{k+1}})^{-1} w_{i_n}^\ast U_{i_n,i_{n+1}-1}\begin{pmatrix}0\\u^{(i_{n+1}-1)}\end{pmatrix}.
\end{align*}
Using Lemma \ref{lem:evec-first-entry-error} to bound the coefficients $a(i_1,\ldots, i_{k+1})$, we get
\begin{align*}
\left|a(i_1,\ldots, i_{k+1})\right|^2 \prec (Nt)^{-k}\prod_{n=1}^k (t + |\lambda_{i_n}-\lambda_{i_{n+1}}|^2)|\lambda_{i_n} - \lambda_{j}|^{-2},
\end{align*}
where $i_{k+1} = j$. We split the set of indices $I = \{i_1,\ldots, i_k\}$ into two disjoint subsets $I = I_1 \cup I_2$ such that
\begin{align*}
I_1 &= \{i\in I:\, |\lambda_i - \lambda_j| > t^{1/2}\};\\
I_2 &= \{i\in I:\, N^{-1/2+\upsilon} \le |\lambda_i - \lambda_j| \le t^{1/2}\}.
\end{align*}
Then 
\[
\prod_{n=1}^k |\lambda_{i_n} - \lambda_{j}|^{-2} \le t^{-|I_1|}N^{|I_2|(1-2\upsilon)}.
\]
If $|I_1|\ge|I_2|$, we bound each term $(t + |\lambda_{i_n}-\lambda_{i_{n+1}}|^2)$ in the product above by a constant and write $(Nt)^{-k} = (Nt)^{-|I_1|-|I_2|}$. Then 
\begin{align*}
\left|a(i_1,\ldots, i_{k+1})\right|^2 \prec N^{-|I_1|-2\upsilon |I_2| }t^{-2|I_1|-|I_2|} \le \left(\frac{1}{Nt^3}\right)^{|I_1|} = N^{-3|I_1|\epsilon_0} \le N^{-\epsilon_0}.
\end{align*}
If $|I_1|<|I_2|$, we note that for the distance between to eigenvalues to be larger that $2t^{1/2}$, at least one of them has to be outside of the set $I_2$, thus
\[
\left|\left\{n\in \llbracket1,k\rrbracket:\, |\lambda_{i_n}-\lambda_{i_{n+1}}| > 2t^{1/2}\right\} \right|\le 2|I_1|.
\]
Then for any $n$ outside of this set we can bound $t + |\lambda_{i_n}-\lambda_{i_{n+1}}|^2 \le 5t$, and for any $n$ in this set $t + |\lambda_{i_n}-\lambda_{i_{n+1}}|^2 \le C$. Then the product is bounded by
\[
\prod_{n=1}^k (t + |\lambda_{i_n}-\lambda_{i_{n+1}}|^2) \lesssim t^{k-2|I_1|} = t^{|I_2|-|I_1|}.
\]
Hence,
\begin{align*}
\left|a(i_1,\ldots, i_{k+1})\right|^2 \prec N^{-|I_1|-2\upsilon |I_2|}t^{-3|I_1|}  \le N^{-\upsilon}.
\end{align*}
Plugging these coefficient bounds into \eqref{eq:uj-initial}, we get
\begin{align*}
|T_j|u_j &= |T_j|U_{j-1}\begin{pmatrix}0\\u^{(j-1)}\end{pmatrix} + \sum_{s=1}^{j-1} \oh(N^{-\epsilon_1}) |T_j|U_{s-1}\begin{pmatrix}0\\u^{(s-1)}\end{pmatrix}.
\end{align*}
Here $\epsilon_1 = \min\{\epsilon_0, \upsilon\}$. Now we recall that
\[
|T_j| = \sum_{s=1}^{l_j}q_{j,s}^{\frac12}\mathbf{w}_{j,k}\mathbf{w}_{j,k}^\ast.
\]
Thus it remains to show that 
\[
\left|\mathbf{w}_{j,k}^\ast U_{s-1}\begin{pmatrix}0\\u^{(s-1)}\end{pmatrix}\right| \prec N^{-\frac12}.
\]
This follows by Corollary \ref{cor:vector-bound-first-entry} and the assumption {\bf A3}. The proof in the left eigenvector case is similar.
\end{proof}

\section{Proof of Theorem \ref{theorem:main}}\label{sec:comparison}
Given Theorem \ref{theorem:maingauss}, the first step is the following moment matching, which is Lemma 3.4 in \cite{EYY11}.
\begin{lemma}\label{lemma:ex35M}
Fix $\beta>0$ and $t=N^{-\beta}$. There exists a matrix $\tilde{A}$ such that the following hold.
\begin{enumerate}
\item The entries of $\tilde{A}$ satisfy $\E\tilde{A}_{ij}=0$ and $\E|\tilde{A}_{ij}|^{2}=N^{-1}$ and $\E|\tilde{A}_{ij}|^{p}\leq C_{p}N^{-p/2}$.
\item Consider $A$ and $\tilde{M}_{t}=(1+t)^{-1/2}(\tilde{A}+t^{1/2}B)$, where $B$ is sampled from an independent Gaussian ensemble. Set $m_{k}(i,j)=\E A_{ij}^{k}$ and $\tilde{m}_{k}(i,j)=\E (\tilde{M}_{t})_{ij}^{k}$. Then for all $i,j$ and $k=1,2,3$ and for some $\delta_{t}>0$, we have 
\begin{align*}
m_{k}(i,j)=\tilde{m}_{k}(i,j)\quad\mathrm{and}\quad m_{4}(i,j)=\tilde{m}_{4}(i,j)+O(N^{-\delta_{t}})
\end{align*}
\item The matrices $A$ and $\tilde{M}_{t}$ are independent.
\end{enumerate}
\end{lemma}
(Condition 2 is sometimes referred to as ``matching up to three-and-a-half moments".) Theorem \ref{theorem:main} follows by Theorem \ref{theorem:maingauss} and the comparison result below for matrices which match up to three-and-a-half moments. 
\begin{theorem}\label{theorem:comparisonwitht}
Take $\tilde{A}$ and $\tilde{M}_{t}$ from Lemma \ref{lemma:ex35M}. Fix positive integers $m_{R},m_{L}$. Fix deterministic $\left\{z_{j}^{0} = z_{j}^{0}(N)\right\}_{j = 1}^{m_{R}+m_{L}}\subset\C$ such that $|z_{j}^{0}|<1-\tau$ for each $j$ and for some $\tau>0$ independent of $N$. Next:
\begin{itemize}
\item For any $j\in \llbracket 1,m_R\rrbracket$ let $(\lambda_{j},u_j)$ denote an eigenvalue-right-eigenvector pair of $A$, and for any $j\in \llbracket m_R+1,m_R+m_L\rrbracket$ let $(\lambda_{j},v_j)$ denote an eigenvalue-left-eigenvector pair of $A$. 
\item Similarly, let $(\lambda_{j}(t),u_{j}(t))$ denote an eigenvalue-right-eigenvector pair of $\tilde{M}_{t}$ for $j\in\llbracket1,m_{R}\rrbracket$, and $(\lambda_{j}(t),v_j(t))$ denote an eigenvalue-left-eigenvector pair of $\tilde{M}_{t}$ for $j\in\llbracket m_{R}+1,m_{R}+m_{L}\rrbracket$.
\item For any $j$, let $T_{j}$ be a deterministic, finite-rank matrix.
\item For any $j$, write $\lambda_{j}=z_{j}^{0}+N^{-1/2}z_{j}$ and $\lambda_{j}(t)=z_{j}^{0}+N^{-1/2}z_{j}(t)$.
\item Set $\mathbf{z}=(z_{j})_{j=1}^{m_{R}+m_{L}}$ and $\mathbf{T}:=(N\|T_{j}u_{j}\|^{2},N\|T_{k}v_{k}\|^{2})_{j,k}$. Similarly, set $\mathbf{z}(t)=(z_{j}(t))_{j=1}^{m_{R}+m_{L}}$ and $\mathbf{T}(t):=(N\|T_{j}u_{j}(t)\|^{2},N\|T_{k}v_{k}(t)\|^{2})_{j,k}$. 
\end{itemize}
For any test function $F\in C^{\infty}_{c}(\C^{m_{R}+m_{L}}\times\R^{m_{R}+m_{L}})$, we have 
\begin{align*}
\lim_{N\rightarrow\infty}\E F(\mathbf{z},\mathbf{T}) = \lim_{N\rightarrow\infty}\E F(\mathbf{z}(t),\mathbf{T}(t)),
\end{align*}
where the expectation is over all the randomness in $A$ and $\tilde{M}_{t}$, respectively.
\end{theorem}
The rest of this section is dedicated to the proof of Theorem \ref{theorem:comparisonwitht}. The strategy is based on Girko's formula, and for this reason, we must express eigenvector statistics in terms of eigenvalues. We start with the following construction.
\begin{defn}\label{defn:Vfunction}
Fix $\delta_{V}>0$ small, and set $\eta_{V}:=N^{-1-\delta_{V}}$. Fix any finite-rank matrix $T\in M_{2N}(\C)$, and define 
\begin{align*}
V(z,T):=N\eta_{V}\tr[T^{\ast}T\mathrm{Im}G_{z}(\eta_{V})]=\sum_{k}\frac{N\eta_{V}^{2}}{(\xi^{z}_{k})^{2}+\eta_{V}^{2}}\|T \mathbf{u}_{k}^{z}\|^{2}.
\end{align*}
Above, $\xi_{k}^{z}$ are eigenvalues of the Hermitization $\mathcal{H}_{z}$, and $\mathbf{u}_{k}^{z}$ are the corresponding eigenvalues. We adopt the convention $\xi_{k}^{z}\geq0$ for all $k\in\llbracket1,N\rrbracket$, and $\xi_{k}^{z}=-\xi_{-k}^{z}$ for any $k\in\llbracket-N,-1\rrbracket$. We also impose $\xi_{k}^{z}\leq\xi_{\ell}^{z}$ for $k\leq\ell$. The sum on the far RHS of the previous display is over all $k\in\llbracket-N,1\rrbracket\cup\llbracket1,N\rrbracket$.
\end{defn}
Fix any right eigenvalue-eigenvector pairs $(\lambda_{1},u_{1}),\ldots,(\lambda_{m_{R}},u_{m_{R}})$ and left eigenvalue-eigenvector pairs $(\lambda_{m_{R}+1},v_{m_{R}+1}),\ldots,(\lambda_{m_{R}+m_{L}},v_{m_{R}+m_{L}})$. Now, fix $j=1,\ldots,m_{R}$. Since $(A-\lambda_{j})u_{j}=0$, we have 
\begin{align*}
\mathcal{H}_{\lambda_{j}}\begin{pmatrix}0\\u_{j}\end{pmatrix}=0.
\end{align*}
Similarly, for any $j=m_{R}+1,\ldots,m_{R}+m_{L}$, since $u_{j}^{\ast}(A-\lambda_{j})=\lambda_{j}u_{j}^{\ast}$, we have 
\begin{align*}
\mathcal{H}_{\lambda_{j}}\begin{pmatrix}v_{j}\\0\end{pmatrix}=0.
\end{align*}
We now provide the following, which essentially compares $V(\lambda_{j},T)$ to the corresponding eigenvector information we are interested in. We explain its utility afterwards.
\begin{lemma}\label{lemma:Vformula}
We have 
\begin{align*}
V(z,T)=\sum_{k=\pm1}\frac{N\eta_{V}^{2}}{(\xi_{k}^{z})^{2}+\eta_{V}^{2}}\|T\mathbf{u}_{k}^{z}\|^{2}+\mathbf{1}\left(\xi_{2}^{z}\leq N^{-1-\delta_{V}/2}\right)\mathcal{O}(1)+\mathcal{O}(N^{-\delta_{V}/2}).
\end{align*}
Above, $\mathcal{O}$ is with respect to the randomness of the matrix $A$. Now, fix any $j=1,\ldots,m_{R}+m_{L}$. We have
\begin{align*}
V(\lambda_{j},T)=\sum_{k=\pm1}N\|T\mathbf{u}_{k}^{\lambda_{j}}\|^{2}+\mathbf{1}\left(\xi_{2}^{\lambda_{j}}\leq N^{-1-\delta_{V}/2}\right)\mathcal{O}(1)+\mathcal{O}(N^{-\delta_{V}/2}).
\end{align*}
\end{lemma}
\begin{proof}
To prove the second claim, we set $z=\lambda_{j}$ and use $\xi_{\pm1}^{\lambda_{j}}=0$. To prove the first claim, it suffices to show 
\begin{align*}
\sum_{k\neq\pm1}\frac{N\eta_{V}^{2}}{(\xi^{z}_{k})^{2}+\eta_{V}^{2}}\|T \mathbf{u}_{k}^{z}\|^{2}=\mathbf{1}\left(\xi_{2}^{z}\leq N^{-1-\delta_{V}/2}\right)\mathcal{O}(1)+\mathcal{O}(N^{-\delta_{V}/2}).
\end{align*}
We first write, for some $c>0$ small and independent of $N$,
\begin{align*}
\sum_{k\neq\pm1}\frac{N\eta_{V}^{2}}{(\xi^{z}_{k})^{2}+\eta_{V}^{2}}\|T \mathbf{u}_{k}^{z}\|^{2}=\sum_{\substack{k\neq\pm1\\|\xi_{k}^{z}|\leq c}}\frac{N\eta_{V}^{2}}{(\xi^{z}_{k})^{2}+\eta_{V}^{2}}\|T \mathbf{u}_{k}^{z}\|^{2}+\sum_{\substack{k\neq\pm1\\|\xi_{k}^{z}|> c}}\frac{N\eta_{V}^{2}}{(\xi^{z}_{k})^{2}+\eta_{V}^{2}}\|T \mathbf{u}_{k}^{z}\|^{2}.
\end{align*}
Using the trivial bound $\|T\mathbf{u}_{k}^{z}\|=O(1)$ and $\eta_{V}^{2}=N^{-2-2\delta_{V}}$, we have 
\begin{align*}
\sum_{\substack{k\neq\pm1\\|\xi_{k}^{z}|> c}}\frac{N\eta_{V}^{2}}{(\xi^{z}_{k})^{2}+\eta_{V}^{2}}\|T \mathbf{u}_{k}^{z}\|^{2}\lesssim N^{-1-2\delta_{V}}\sum_{\substack{k\neq\pm1\\|\xi_{k}^{z}|> c}}\frac{1}{(\xi^{z}_{k})^{2}+\eta_{V}^{2}}\lesssim N^{-2\delta_{V}}.
\end{align*}
Next, for $|\xi_{k}^{z}|<c$, we can use delocalization (see Proposition \ref{prop:rigiditydeloc}) and the finite-rank property of $T$ to get $N\|T\mathbf{u}_{k}^{z}\|^{2}=O(1)$; note that delocalization only holds for bulk singular values in Proposition \ref{prop:rigiditydeloc}. Now, we write
\begin{align}
\sum_{\substack{k\neq\pm1\\|\xi_{k}^{z}|\leq c}}\frac{\eta_{V}^{2}}{(\xi_{k}^{z})^{2}+\eta_{V}^{2}}&=\sum_{\substack{k\neq\pm1\\|\xi_{k}^{z}|\leq c\\|k|\leq N^{\epsilon}}}\frac{\eta_{V}^{2}}{(\xi_{k}^{z})^{2}+\eta_{V}^{2}}+\sum_{\substack{k\neq\pm1\\|\xi_{k}^{z}|\leq c\\|k|> N^{\epsilon}}}\frac{\eta_{V}^{2}}{(\xi_{k}^{z})^{2}+\eta_{V}^{2}}.\label{eq:Vformuladecomp}
\end{align}
The first term on the RHS is bounded deterministically by $N^{\epsilon}$. For the second term on the RHS, we use rigidity (Proposition \ref{prop:rigiditydeloc}) to bound it by $\mathcal{O}(N^{-\epsilon})$. In particular, this gives $|\xi_{k}^{z}|\geq N^{-1}|k|+\mathcal{O}(N^{-1})\geq \frac12N^{-1}|k|$ since $|k|>N^{\epsilon}$, with which we can bound the second term on the RHS. Since $\epsilon>0$ is arbitrary, we get the a priori bound 
\begin{align*}
\sum_{k\neq\pm1}\frac{N\eta_{V}^{2}}{(\xi^{z}_{k})^{2}+\eta_{V}^{2}}\|T \mathbf{u}_{k}^{z}\|^{2}=\mathcal{O}(1).
\end{align*}
It remains to show that
\begin{align*}
\mathbf{1}\left(\xi_{2}^{z}\geq N^{-1-\delta_{V}/2}\right)\sum_{\substack{k\neq\pm1\\ |\xi_{k}^{z}|\leq c}}\frac{N\eta_{V}^{2}}{(\xi^{z}_{k})^{2}+\eta_{V}^{2}}\|T \mathbf{u}_{k}^{z}\|^{2}=\mathcal{O}(N^{-\delta_{V}/2}).
\end{align*}
We again use $N\|T\mathbf{u}_{k}^{z}\|^{2}=\mathcal{O}(1)$ and \eqref{eq:Vformuladecomp}. By the assumption $\xi_{2}^{z}\geq N^{-1-\delta_{V}/2}$, we have 
\begin{align*}
\sum_{\substack{k\neq\pm1\\|k|\leq N^{\epsilon}}}\frac{\eta_{V}^{2}}{(\xi_{k}^{z})^{2}+\eta_{V}^{2}}\leq \sum_{\substack{k\neq\pm1\\|k|\leq N^{\epsilon}}}\frac{\eta_{V}^{2}}{(\xi_{k}^{z})^{2}}\leq N^{-\delta_{V}+\epsilon}.
\end{align*}
Also, we can again use rigidity to show the second term on the RHS of \eqref{eq:Vformuladecomp} is $\mathcal{O}(N^{-\epsilon})$. Now choose $\epsilon=\delta_{V}/2$ to conclude.
\end{proof}
If we choose $T$ appropriately, then our computations before Lemma \ref{lemma:Vformula} show that $N\|T\mathbf{u}_{\pm1}^{\lambda_{j}}\|^{2}=N\|T_{j}u_{j}\|^{2}$ for $j\leq m_{R}$, and $N\|T\mathbf{u}_{\pm1}^{\lambda_{j}}\|^{2}=N\|T_{j}v_{j}\|^{2}$ for $j\geq m_{R}+1$. We expand on this more later when it is more relevant. First, we present the following technical result, which is important to various estimates in the proof of Theorem \ref{theorem:comparisonwitht}.
\begin{lemma}\label{lemma:gradientestimates}
Fix Hermitian $T\in M_{2N}(\C)$, and assume $\|T\|_{\mathrm{op}}=O(1)$. Fix $\eta\leq N^{-1}$. We have 
\begin{align}
\sup_{|z|\leq1-\tau}|\nabla_{z}\tr TG_{z}(\eta)|&=\mathcal{O}(N^{-\frac32}\eta^{-2}\tr\sqrt{T^{\ast}T}),\label{res:nz1}\\
\sup_{|z|\leq1-\tau}|\nabla_{z}^{2}\tr TG_{z}(\eta)|&=\mathcal{O}(N^{-2}\eta^{-3}\tr\sqrt{T^{\ast}T}).\label{res:nz2}
\end{align}
The $\mathcal{O}$ is with respect to randomness of $A$. In particular, we have 
\begin{align}
\nabla_{z}^{m}V(z,T)&=\mathcal{O}(N^{-1-m/2}\eta_{V}^{-1-m}),\quad m=1,2.\label{dev_V}
\end{align}
\end{lemma}
\begin{proof}
By union bound over a very fine net and elementary resolvent bounds for $G_{z}(\eta)$, it is enough to prove the proposed bounds for a fixed deterministic $|z|\leq1-\tau$. We first write (for some $c>0$ small and independent of $N$)
\begin{align*}
G_{z}(\eta):=G+\widehat{G}:=\sum_{\alpha:|\xi_{\alpha}^{z}|\leq c}\frac{\mathbf{u}_{\alpha}\mathbf{u}_{\alpha}^{\ast}}{\xi_{\alpha}^{z}-i\eta}+\sum_{\alpha:|\xi_{\alpha}^{z}|> c}\frac{\mathbf{u}_{\alpha}\mathbf{u}_{\alpha}^{\ast}}{\xi_{\alpha}^{z}-i\eta}
\end{align*}
Bounds for $|\nabla_{z}\tr T\widehat{G}|$ and $|\nabla_{z}^{2}\tr T\widehat{G}|$ are simple to establish since $\|\widehat{G}\|_{\mathrm{op}}=O(1)$. Thus, we prove bounds for $|\nabla_{z}\tr TG|$ and $|\nabla_{z}^{2}\tr TG|$. The advantage of this decomposition is the use of delocalization estimates (Proposition \ref{prop:rigiditydeloc}) for all relevant eigenvectors.

Fix any $Y,Y_{1},Y_{2}\in\{E_{12},E_{21}\}$, where $E_{12}$ and $E_{21}$ are the $2N\times2N$ block matrices
\begin{align*}
E_{12}:=\begin{pmatrix}0&I_{N}\\0&0\end{pmatrix},\quad E_{21}:=\begin{pmatrix}0&0\\I_{N}&0\end{pmatrix}.
\end{align*}
By resolvent perturbation identities, it suffices to prove 
\begin{align}
\tr GTGY&=\mathcal{O}(N^{-3/2}\eta^{-2}\tr\sqrt{T^{\ast}T}),\label{res:nz3},\\
\tr GTGY_{1}GY_{2}&=\mathcal{O}(N^{-2}\eta^{-3}\tr\sqrt{T^{\ast}T}).\label{res:nz4}
\end{align}
Now, for any vectors $\mathbf{u},\mathbf{v}$ and any matrix $X$, we write $X_{\mathbf{u}\mathbf{v}}:=\mathbf{u}^{\ast}X\mathbf{v}$. We have 
\begin{align*}
\tr GTGY&=\sum_{\alpha,\beta}\frac{T_{\mathbf{u}_{\alpha}^{z}\mathbf{u}_{\beta}^{z}}Y_{\mathbf{u}_{\beta}^{z}\mathbf{u}_{\alpha}^{z}}}{(\xi_{\alpha}^{z}-i\eta)(\xi_{\beta}^{z}-i\eta)}.
\end{align*}
By Proposition \ref{prop:rigiditydeloc} applied to $\mathbf{u}_{\alpha}^{z}$, because $T$ is deterministic and Hermitian, we have $|T_{\mathbf{u}_{\alpha}^{z}\mathbf{u}_{\beta}^{z}}|=\mathcal{O}(N^{-1}\tr\sqrt{T^{\ast}T})$. By Proposition \ref{prop:eth}, we also have $|Y_{\mathbf{u}_{\beta}^{z}\mathbf{u}_{\alpha}^{z}}|=\mathcal{O}(N^{-1/2}+\delta_{|\alpha|=|\beta|}N^{-1}|\alpha|)$. If we plug these bounds into the previous display and use rigidity of $\xi_{\alpha}^{z}$, we have 
\begin{align*}
|\tr GTGY|&=N^{-1}\tr\sqrt{T^{\ast}T}\mathcal{O}\left(\sum_{\alpha,\beta}\frac{N^{-\frac12}+|\xi_{\alpha}^{z}|\delta_{|\alpha|=|\beta|}}{(\xi_{\alpha}^{z}-i\eta)(\xi_{\beta}^{z}-i\eta)}\right)=\mathcal{O}(N^{-3/2}\eta^{-2}\tr\sqrt{T^{\ast}T}).
\end{align*}
(The last estimate also uses rigidity of $\xi_{\alpha}^{z}$.) This proves \eqref{res:nz3}. To show \eqref{res:nz4}, we similarly have 
\begin{align*}
\tr GTGY_{1}GY_{2}&=\sum_{\alpha,\beta,\gamma}\frac{T_{\mathbf{u}_{\alpha}^{z}\mathbf{u}_{\beta}^{z}}(Y_{1})_{\mathbf{u}_{\beta}^{z}\mathbf{u}_{\gamma}^{z}}(Y_{2})_{\mathbf{u}_{\gamma}^{z}\mathbf{u}_{\alpha}^{z}}}{(\xi_{\alpha}^{z}-i\eta)(\xi_{\beta}^{z}-i\eta)(\xi_{\gamma}^{z}-i\eta)}\\
&=N^{-1}\tr\sqrt{T^{\ast}T}\mathcal{O}\left(\sum_{\alpha,\beta,\gamma}\frac{(N^{-\frac12}+|\xi_{\alpha}^{z}|\delta_{|\alpha|=|\beta|})(N^{-\frac12}+|\xi_{\beta}^{z}|\delta_{|\beta|=|\gamma|})}{(\xi_{\alpha}^{z}-i\eta)(\xi_{\beta}^{z}-i\eta)(\xi_{\gamma}^{z}-i\eta)}\right),
\end{align*}
at which point we again use rigidity of $\xi^{z}_{\alpha}$ to conclude \eqref{res:nz4}.
\end{proof}
We now reduce comparison of $N\|T\mathbf{u}_{\pm1}^{\lambda_{j}}\|^{2}$ to that of $V(\lambda_{j},T)$ via level repulsion.
\begin{prop}\label{prop:levelrepulsion}
Fix any $|z|\leq1-\tau$ with $\tau>0$ fixed. If $\delta>0$ is small enough, then we have
\begin{align*}
\mathbb{P}\left(\xi_{2}^{z}\leq N^{-1-\delta}\right)\leq N^{-2.1\delta}.
\end{align*}
\end{prop}
We defer the proof of this to the end of this entire section. The coefficient $2.1$ in the exponent is not important; it just needs to be strictly bigger than $2$.

Note that Proposition \ref{prop:levelrepulsion} is for \emph{fixed} $|z|<1$, whereas we will need it for random $z$. To this end, we use a net argument; we ultimately conclude the following.
\begin{lemma}\label{lemma:levelrepulsion}
Fix any $z_{0}\in\C$ such that $|z_{0}|\leq1-\tau$ with $\tau>0$ independent of $N$. Choose any $R>0$ independent of $N$, and let $B_{z_{0}}(RN^{-1/2})\subseteq\C$ be the ball of radius $RN^{-1/2}$ around $z_{0}$. If $\delta>0$ is small enough, then there exists $c=c(\delta)>0$ such that 
\begin{align*}
\mathbb{P}\left(\exists\lambda\in\mathrm{Spec}(A)\cap B_{z_{0}}(RN^{-1/2}): \ \xi_{2}^{\lambda}\leq N^{-1-\delta}\right)=O(N^{-c}).
\end{align*}
\end{lemma}
\begin{proof}
Throughout this argument, we set $\eta=N^{-1-\delta_{1}}$, where $0<\delta_{1}<\delta$ (and $\delta>0$ will be a small parameter chosen shortly). Let $\phi:\R\to\R_{\geq0}$ be a smooth test function such that 
\begin{equation}\label{def_phi}
    \phi(x)= \begin{cases}  1, & x \ge 1.9 \\ 0, & x \leq 1.1 \end{cases}
\end{equation} 
For convenience, set $B:=B_{z_{0}}(RN^{-1/2})$. We claim that for some $c>0$, we have
\begin{align}
\E\left[\phi\left(\max_{z\in B}\frac{\eta}{2}\mathrm{Im}\tr G_{z}(\eta)\right)\right]\leq N^{-c},\label{eq:randomlevelrepulsion}
\end{align}
where the expectation is over randomness of the matrix $A$. To see that this is enough, we first note 
\begin{align}
\frac{\eta}{2}\mathrm{Im}\tr G_{z}(\eta)&=\sum_{k=1}^{N}\frac{\eta^{2}}{(\xi_{k}^{z})^{2}+\eta^{2}}\geq \frac{2\eta^{2}}{(\xi_{2}^{z})^{2}+\eta^{2}}+\sum_{k=3}^{N}\frac{\eta^{2}}{(\xi_{k}^{z})^{2}+\eta^{2}}.\label{eq:xitraceduality}
\end{align}
Thus, if there is $z\in B$ for which $\xi_{2}^{z}\leq N^{-1-\delta}\ll \eta$, then on this event, we have 
\begin{align*}
\eta\mathrm{Im}\tr G_{z}(\eta)\geq 2\frac{\eta^{2}}{(\xi_{2}^{z})^{2}+\eta^{2}}\geq1.9.
\end{align*}
In particular, since this statement is deterministic in $z$, we have the bounds
\begin{align*}
\mathbb{P}\left(\exists\lambda\in\mathrm{Spec}(A)\cap B_{z_{0}}(RN^{-1/2}): \ \xi_{2}^{\lambda}\leq N^{-1-\delta}\right)&\leq\mathbb{P}\left(\max_{z\in B}\frac{\eta}{2}\mathrm{Im}\tr G_{z}(\eta)\geq1.9 \right)\\
&\leq\E\left[\phi\left(\max_{z\in B}\frac{\eta}{2}\mathrm{Im}\tr G_{z}(\eta)\right)\right],
\end{align*}
at which point we conclude. To prove the remaining estimate \eqref{eq:randomlevelrepulsion}, Let $\mathcal{E}\subseteq B$ be such that the distance between any point in $B$ to $\mathcal{E}$ is at most $N^{1/2-\epsilon}\eta$, and $|\mathcal{E}|=O(N^{2\epsilon+2\delta_{1}})$. Here, $\epsilon>0$ is small. Fix $0<\delta_{2}<\delta_{1}$ such that $2\epsilon+2\delta_{1}-2.1\delta_{2}<0$. We claim
\begin{align*}
\E\left[\phi\left(\max_{z\in B\cap\mathcal{E}}\frac{\eta}{2}\mathrm{Im}\tr G_{z}(\eta)\right)\right]&\leq\sum_{z\in B\cap\mathcal{E}}\mathbb{P}\left(\frac{\eta}{2}\mathrm{Im}\tr G_{z}(\eta)\geq 1.1\right)\\
&\leq \sum_{z\in B\cap\mathcal{E}}\mathbb{P}\left(\xi_{2}^{z}\leq N^{-1-\delta_{2}}\right)\\
&\leq |\mathcal{E}|N^{-2.1\delta_{2}}\lesssim N^{2\epsilon+2\delta_{1}-2.1\delta_{2}}.
\end{align*}
The first bound follows by union bound and definition of $\phi$. To show the second bound, we first use the identity in \eqref{eq:xitraceduality} to get the upper bound
\begin{align*}
\frac{\eta}{2}\mathrm{Im}\tr G_{z}(\eta)&\leq 1+\sum_{k=2}^{N}\frac{\eta^{2}}{(\xi_{k}^{z})^{2}+\eta^{2}}.
\end{align*}
If the LHS of the previous display is $\geq1.1$, then $\xi_{2}^{z}\leq N^{-1-\delta_{2}}$ (for $N$ large enough). Indeed, if not, then the $k=2$ term in the sum is $O(N^{-\kappa})$. The same is true for the first $N^{\tau}$ many terms, where $\tau<\kappa$. Now we use rigidity to handle the remaining terms; see the proof of Lemma \ref{lemma:Vformula}, for example. The last line follows by Proposition \ref{prop:levelrepulsion}. By construction of $\delta_{2}$, we deduce
\begin{align*}
\E\left[\phi\left(\max_{z\in B\cap\mathcal{E}}\frac{\eta}{2}\mathrm{Im}\tr G_{z}(\eta)\right)\right]\leq N^{-c}.
\end{align*}
To extend from $B\cap\mathcal{E}$ to $B$, take any $z_{1}\in B$ and $z_{2}\in\mathcal{E}$. By \eqref{res:nz1} for $T=I$, we have
\begin{align*}
|\eta\tr G_{z_{1}}(\eta)-\eta\tr G_{z_{2}}(\eta)|&=\mathcal{O}(N^{-1/2}\eta^{-1}|z_{1}-z_{2}|).
\end{align*}
Now use $|z_{1}-z_{2}|\leq N^{1/2-\epsilon}\eta$ to conclude. 
\end{proof}
Now, by Lemmas \ref{lemma:Vformula} and \ref{lemma:levelrepulsion}, to conclude the proof of Theorem \ref{theorem:comparisonwitht}, it suffices to show the following.
\begin{prop}\label{prop:Vcomparison}
Retain the setting of Theorem \ref{theorem:comparisonwitht}. Define
\begin{align*}
\mathbf{V}&:=(V(\lambda_{j},F_{12}\otimes T_{j}),V(\lambda_{k},F_{21}\otimes T_{k}))_{j,k}\\
\mathbf{V}(t)&:=(V(\lambda_{j}(t),F_{12}\otimes T_{j}),V(\lambda_{k}(t),F_{21}\otimes T_{k}))_{j,k},
\end{align*}
where $j\in\llbracket 1,m_{R}\rrbracket$ and $k\in\llbracket m_{R}+1,m_{R}+m_{L}\rrbracket$, and where $F_{12},F_{21}$ are the $2\times2$ matrices
\begin{align*}
F_{12}=\begin{pmatrix}0&1\\0&0\end{pmatrix}\quad\mathrm{and}\quad F_{21}=\begin{pmatrix}0&0\\1&0\end{pmatrix}.
\end{align*}
For any test function $F\in C^{\infty}_{c}(\C^{m_{R}+m_{L}}\times \R^{m_{R}+m_{L}})$, we have 
\begin{align*}
\lim_{N\rightarrow\infty}\E F(\mathbf{z},\mathbf{V}) = \lim_{N\rightarrow\infty}\E F(\mathbf{z}(t),\mathbf{V}(t)),
\end{align*}
where the expectation is over all the randomness in $A$ and $\tilde{M}_{t}$, respectively.
\end{prop}
Let us now prove Theorem \ref{theorem:comparisonwitht} assuming that Proposition \ref{prop:Vcomparison} holds.
\begin{proof}[Proof of Theorem \ref{theorem:comparisonwitht}]
By Lemmas \ref{lemma:Vformula} and \ref{lemma:levelrepulsion}, we get $\E F(\mathbf{z},\mathbf{T})=\E F(\mathbf{z},\mathbf{V})+O(\|F\|_{C^{1}}N^{-\kappa})$, where $\kappa>0$. But $\tilde{M}_{t}$ has the same structure as $A$, so $\E F(\mathbf{z}(t),\mathbf{T}(t))=\E F(\mathbf{z}(t),\mathbf{V}(t))+O(\|F\|_{C^{1}}N^{-\kappa})$. Theorem \ref{theorem:comparisonwitht} follows immediately by these two estimates and Proposition \ref{prop:Vcomparison}.
\end{proof}
\subsection{Proof of Proposition \ref{prop:Vcomparison}}
We approximate general $F\in C^{\infty}_{c}(\C^{m_{R}+m_{L}}\times\R^{m_{R}+m_{L}})$ by products of functions $F^{(j)}\in C^{\infty}_{c}(\C\times\R)$ and provide estimates uniform with respect to some norm on $F^{(j)}$ (this can be done by taking cutoffs of the Fourier transform). In particular, we consider 
\begin{align*}
F(\mathbf{z},\mathbf{V})&=\prod_{j=1}^{m_{R}+m_{L}}\frac{1}{N}\sum_{k=1}^{N}f^{(j)}(\lambda_{k}),\\
f^{(j)}(z)&=\begin{cases}NF^{(j)}\left(N^{1/2}(z-z_{j}^{0}),V(z,F_{12}\otimes T_{j})\right)&1\leq j\leq m_{R}\\NF^{(j)}\left(N^{1/2}(z-z_{j}^{0}),V(z,F_{21}\otimes T_{j})\right)&m_{R}+1\leq j\leq m_{R}+m_{L}\end{cases}
\end{align*}
and likewise for $F(\mathbf{z}(t),\mathbf{V}(t))$. Fix $1\leq j\leq m_{R}$; the following analysis holds similarly for $m_{R}+1\leq\ell\leq m_{R}+m_{L}$. For convenience, we define the function $\tilde{f}^{(j)}(z):=F^{(j)}(z,V(z_{j}^{0}+N^{-1/2}z,F_{12}\otimes T_{j}))$, so that $f^{(j)}(z)=N\tilde{f}^{(j)}(N^{1/2}(z-z_{j}^{0}))$. We claim that the quantity
\begin{align*}
\mathcal{E}&:=\frac{1}{N}\sum_{k=1}^{N}f^{(j)}(\lambda_{k})-\frac{1}{\pi}\int_{\mathbb{D}}f^{(j)}(z)\d^{2}z-I_{\delta_{V}}(A,f^{(j)}),
\end{align*}
satisfies the bound $\E|\mathcal{E}|=O(N^{-\kappa})$ for some $\kappa>0$. Above, $\mathbb{D}=\{z\in\C: |z|\leq1\}$ is the unit disc. Moreover, we introduced
\begin{align*}
I_{\delta_{V}}(A,f^{(j)})&:=\frac{1}{2\pi}\int_{\C}\Delta f^{(j)}(z)\int_{N^{-1-10\delta_{V}}}^{N^{-1+10\delta_{V}}}\langle\mathrm{Im}G_{z}(\eta)-\mathrm{Im}m_{z}(i\eta)\rangle\d\eta\d^{2}z,
\end{align*}
and $m_{z}(i\eta)$ is defined via the self-consistent equation
\begin{align*}
-\frac{1}{m_{z}(w)}=w+m_{z}(w)-\frac{|z|^{2}}{w+m_{z}(w)},\quad \mathrm{Im}[m_{z}(w)]\mathrm{Im}[w]>0.
\end{align*}
Let us briefly sketch the proof; it is the same argument as Theorem 2.4 in \cite{CES20}. We start with a priori bounds. Let $\nabla_{1}$ and $\Delta_{1}$ mean gradient and Laplacian, respectively, with respect to the $z$-variable in the first input into $F^{(j)}$, and let $\partial_{2}$ mean derivative with respect to the second input into $\tilde{f}^{(j)}$. Finally, let $\Delta$ be the total Laplacian. We first compute 
\begin{align*}
\Delta \tilde{f}^{(j)}(z,V(z,F_{12}\otimes T_{j}))&=\Delta_{1}\tilde{f}^{(j)}+2N^{-\frac12}\nabla_{1}\tilde{f}^{(j)}\cdot(\partial_{2}\tilde{f}^{(j)}\nabla_{z}V(z,F_{12}\otimes T_{j}))\\
&+N^{-1}(\partial_{2}\tilde{f}^{(j)})^{2}\|\nabla_{z}V(z,F_{12}\otimes T_{j})\|^{2}+N^{-1}\partial_{2}\tilde{f}^{(j)}\Delta V(z,F_{12}\otimes T_{j}).
\end{align*}
Since $F^{(j)}\in C^{\infty}_{c}(\C\times\R)$, its derivatives are $O(1)$ deterministically. Using this and \eqref{dev_V}, we have 
\begin{align*}
|\Delta \tilde{f}^{(j)}(z)|=\mathcal{O}(1+N^{-2}\eta_{V}^{-2}+N^{-4}\eta_{V}^{-4}+N^{-3}\eta_{V}^{-3})=\mathcal{O}(N^{4\delta_{V}}),
\end{align*}
where we recall $\eta_{V}=N^{-1-\delta_{V}}$ with $\delta_{V}>0$ small. On the other hand, since $\delta_{V}>0$ is small, resolvent perturbation also gives the bound $|\Delta \tilde{f}^{(j)}(z)|=O(N^{4})$. This last deterministic bound, for example, shows that $|\mathcal{E}|=O(N^{10})$. Next, by (3.2) in \cite{CES20}, which is a deterministic identity that holds even for our random test function $f^{(j)}$ in place of $g$ therein, we must first prove the following estimates:
\begin{align*}
\int_{\C}\Delta f^{(j)}(z)\log|\det[\mathcal{H}_{z}-iT]|\d^{2}z&=\mathcal{O}(N^{-98}\|\Delta \tilde{f}^{(j)}\|_{L^{1}(\C)})=\mathcal{O}(N^{-94}),\\
\int_{\C}\Delta f^{(j)}(z)\int_{N^{100}}^{+\infty}\left(\mathrm{Im}m_{z}(i\eta)-\frac{1}{\eta+1}\right)\d\eta\d^{2}z&=\mathcal{O}(N^{-98}\|\Delta \tilde{f}^{(j)}\|_{L^{1}(\C)})=\mathcal{O}(N^{-94}),
\end{align*}
The first estimate in each line is by Lemma 3.1 in \cite{CES20}. To prove the second estimate in each line, we use the compact support property of $f^{(j)}$ along with $\|\Delta f^{(j)}\|_{L^{\infty}(\C)}=O(N^{4})$. Next, we define the quantities
\begin{align*}
I_{0}^{\eta_{0}}&:=\frac{1}{2\pi}\int_{\C}\Delta f^{(j)}(z)\int_{0}^{\eta_{0}}\langle \mathrm{Im}G_{z}(\eta)-\mathrm{Im}m_{z}(i\eta)\rangle\d\eta\d^{2}z,\\
I_{\eta_{1}}^{N^{100}}&:=\frac{1}{2\pi}\int_{\C}\Delta f^{(j)}(z)\int_{\eta_{1}}^{N^{100}}\langle \mathrm{Im}G_{z}(\eta)-\mathrm{Im}m_{z}(i\eta)\rangle\d\eta\d^{2}z,
\end{align*}
where $\eta_{0}=N^{-1-10\delta_{V}}$ and $\eta_{1}=N^{-1+10\delta_{V}}$. To conclude $\E|\mathcal{E}|=O(N^{-\kappa})$ for some $\kappa>0$, according to proof of Theorem 2.4 in \cite{CES20}, it suffices to prove 
\begin{align*}
|I_{0}^{\eta_{0}}|+|I_{\eta_{1}}^{N^{100}}|&=\mathcal{O}(\|\Delta \tilde{f}^{(j)}\|_{L^{1}(\C)})=\mathcal{O}(N^{4\delta_{V}}) \quad\mathrm{and}\quad \E|I_{0}^{\eta_{0}}|+\E|I_{\eta_{1}}^{N^{100}}|=O(N^{-\kappa}).
\end{align*}
The first estimate follows from Lemma 3.1 in \cite{CES20}, by $\|\Delta \tilde{f}^{(j)}\|_{L^{1}(\C)}\lesssim\|\Delta\tilde{f}^{(j)}\|_{L^{\infty}(\C)}$ (because $\tilde{f}^{(j))}$ is compactly supported) and by our earlier estimate $\|\Delta\tilde{f}^{(j)}\|_{L^{\infty}(\C)}=\mathcal{O}(N^{4\delta_{V}})$. To prove the second estimate, we write the following with $L>0$ large but fixed:
\begin{align*}
I_{0}^{\eta_{0}}&=\frac{1}{2\pi}\int_{\C}\Delta f^{(j)}(z)\times\frac{1}{N}\sum_{\lambda_{k}<N^{-L}}\log\left(1+\frac{\eta_{0}^{2}}{\lambda_{k}^{2}}\right)\d^{2}z\\
&+\frac{1}{2\pi}\int_{\C}\Delta f^{(j)}(z)\times\frac{1}{N}\sum_{\lambda_{k}\geq N^{-L}}\log\left(1+\frac{\eta_{0}^{2}}{\lambda_{k}^{2}}\right)\d^{2}z\\
&-\frac{1}{2\pi}\int_{\C}\Delta f^{(j)}(z)\times\int_{0}^{\eta_{0}}\mathrm{Im}m_{z}(i\eta)\d^{2}z\\
&=\mathrm{I}+\mathrm{II}+\mathrm{III}.
\end{align*}
Using the deterministic bound $|\Delta f^{(j)}(z)|=O(N^{2}|\Delta\tilde{f}^{(j)}(z)|)=O(N^{6})$, follow (3.5) and use (2.15), all in \cite{CES20}, to show that if $L>0$ is large enough, then $\E|\mathrm{I}|=O(N^{4}N^{-L})=O(N^{-1})$. For $\mathrm{II}$, we use the earlier bound $\|\Delta\tilde{f}^{(j)}\|_{L^{1}(\C)}=\mathcal{O}(N^{4\delta_{V}})$ to get the following for $\epsilon>0$ small and $D>0$ large:
\begin{align*}
\E|\mathrm{II}|&\leq N^{4\delta_{V}+\epsilon}\E\left[\frac{|\mathrm{II}|}{\|\Delta\tilde{f}^{(j)}\|_{L^{1}(\C)}}\right]+O(N^{-D}).
\end{align*}
The calculation (3.7) in \cite{CES20} shows that the expectation on the RHS of the previous estimate is $O(N^{-2\kappa})$ for some $\kappa>0$ fixed depending only on $L>0$. Thus, we can choose $\delta_{V},\epsilon>0$ small enough so that $\E|\mathrm{II}|=O(N^{-\kappa})$. Finally, bounding $\E|\mathrm{III}|=O(N^{-\kappa})$ uses the same argument as for $\mathrm{II}$, except we only need the bound $\mathrm{Im}m_{z}(i\eta)=O(1)$, which can be deduced from (2.12) in \cite{CES20}. We have shown $\E|I_{0}^{\eta_{0}}|=O(N^{-\kappa})$. The proof of $\E|I_{\eta_{1}}^{N^{100}}|=O(N^{-\kappa})$ follows by the same argument. In particular, note the deterministic bound $|I_{\eta_{1}}^{N^{100}}|=O(N^{200})$; this and $|\Delta\tilde{f}^{(j)}(z)|=\mathcal{O}(N^{4\delta_{V}})$ give
\begin{align*}
\E|I_{\eta_{1}}^{N^{100}}|&\leq N^{4\delta_{V}+\epsilon}\E\left[\frac{|I_{\eta_{1}}^{N^{100}}|}{\|\Delta\tilde{f}^{(j)}\|_{L^{1}(\C)}}\right]+O(N^{-D}).
\end{align*}
By the proof of Proposition 3.3 in \cite{CES20}, the expectation on the RHS is $O(N^{-2\kappa})$. So, if we take $\delta,\epsilon>0$ small, we get $\E|I_{\eta_{1}}^{N^{100}}|=O(N^{-\kappa})$. The proof of Theorem 2.4 in \cite{CES20} now gives $\E|\mathcal{E}|=O(N^{-\kappa})$. In particular, the comparison in Proposition \ref{prop:Vcomparison} amounts to comparison of $I_{\delta_{V}}(A,f^{(j)})$ and
\begin{align*}
I_{\delta_{V}}(\tilde{M}_{t},f^{(j)}):=\frac{1}{2\pi}\int_{\C}\Delta f^{(j)}(z)\int_{N^{-1-10\delta_{V}}}^{N^{-1+10\delta_{V}}}\langle\mathrm{Im}G_{z,t}(\eta)-\mathrm{Im}m_{z}(i\eta)\rangle\d\eta\d^{2}z,
\end{align*}
where $G_{z,t}$ is $G_{z}$ but with $\tilde{M}_{t}$ instead of $A$. Precisely, by the proof of Proposition 2.5 in \cite{CES20}, we have the following.
\begin{lemma}\label{lemma:Vcomparisonreduction}
Suppose there exists a constant $c>0$ such that
\begin{align}
\E\left[\prod_{j=1}^{m_{R}+m_{L}}I_{\delta_{V}}(A,f^{(j)})\right]-\E\left[\prod_{j=1}^{m_{R}+m_{L}}I_{\delta_{V}}(\tilde{M}_{t},f^{(j)})\right]=O(N^{-c}).\label{eq:Vcomparisonreduction}
\end{align}
Then Proposition \ref{prop:Vcomparison} follows.
\end{lemma}
\subsection{Proof of \eqref{eq:Vcomparisonreduction}}
We start with the following construction.
\begin{defn}
Take $A$ and $\tilde{M}_{t}$ as in Lemma \ref{lemma:ex35M}. Define the matrices
\begin{align}
\mathcal{H}^{0}:=\tilde{\mathcal{H}}=\begin{pmatrix}0&\tilde{M}_{t}\\\tilde{M}_{t}^{\ast}&0\end{pmatrix},\quad \mathcal{H}^{1}:=\mathcal{H}=\begin{pmatrix}0&A\\A^{\ast}&0\end{pmatrix}.\nonumber
\end{align}
Let $\rho_{ij}^{0}$ and $\rho_{ij}^{1}$ denote laws of $\mathcal{H}^{0}_{ij}$ and $\mathcal{H}^{1}_{ij}$, respectively, and for any $\theta\in[0,1]$, define the interpolated law $\rho_{ij}^{\theta}=(1-\theta)\rho_{ij}^{0}+\theta\rho_{ij}^{1}$. Now, for any $\theta\in[0,1]$, let $\mathcal{H}^{\theta}$ be a matrix of size $2N\times2N$ that satisfies the following properties. First, the triple $(\mathcal{H}^{0},\mathcal{H}^{\theta},\mathcal{H}^{1})$ is jointly independent. Second, the marginal distribution of $\mathcal{H}^{\theta}$ is given by the product measure
\begin{align*}
\prod_{i,j}\rho_{ij}^{\theta}(\d \mathcal{H}_{ij}^{\theta}).
\end{align*}
Next, for any indices $(i,j)$ and any $\lambda\in\C$, define the matrix $\mathcal{H}^{\theta,\lambda}_{(ij)}$ as 
\begin{align*}
(\mathcal{H}^{\theta,\lambda}_{(ij)})_{k\ell}:=\begin{cases}\mathcal{H}_{(ij)}^{\theta}&(i,j)\neq(k,\ell)\\\lambda&(i,j)=(k,\ell)\end{cases}
\end{align*}
Finally, for any $\eta>0$ and $z\in\C$, define the resolvents
\begin{align*}
G_{z}^{\theta}(\eta):=\left[\mathcal{H}^{\theta}-\begin{pmatrix}0&z\\z^{\ast}&0\end{pmatrix}-i\eta\right]^{-1}, \quad G_{z,(ij)}^{\theta,\lambda}(\eta):=\left[\mathcal{H}^{\theta,\lambda}_{(ij)}-\begin{pmatrix}0&z\\z^{\ast}&0\end{pmatrix}-i\eta\right]^{-1}.
\end{align*}
\end{defn}
We note that we do not require $\mathcal{H}^{\theta_{1}}$ to be independent from $\mathcal{H}^{\theta_{2}}$ for $\theta_{1}\neq\theta_{2}$. By calculus, we have the following for any smooth $F:\C^{2N\times2N}\to\C$, provided all expectations exist:
\begin{align}
\frac{\d}{\d\theta}\E F(\mathcal{H}^{\theta})=\sum_{i,j=1}^{2N}\E\left[F\left(\mathcal{H}_{(ij)}^{\theta,\mathcal{H}_{(ij)}^{1}}\right)-F\left(\mathcal{H}_{(ij)}^{\theta,\mathcal{H}_{(ij)}^{0}}\right)\right]. \label{expH}
\end{align}
Now, we combine \eqref{expH} with equations (1.7)-(1.8) in \cite{KY16}. This gives
\begin{align*}
\frac{\d}{\d\theta}\E F(\mathcal{H}^{\theta})&=\sum_{n=1}^{\ell}\sum_{i,j=1}^{N}K_{n,(ij)}^{\theta}\E\left[\left(\frac{\d}{\d\mathcal{H}^{\theta}_{ij}}\right)^{n}F(\mathcal{H}^{\theta})\right]+\mathcal{E}_{\ell},
\end{align*}
where the coefficients are defined by
\begin{align*}
\sum_{n=1}^{\infty}K_{n,(ij)}^{\theta}t^{n}=\frac{\E e^{t\mathcal{H}_{ij}^{1}}-\E e^{t\mathcal{H}_{ij}^{0}}}{\E e^{t\mathcal{H}_{ij}^{\theta}}}.
\end{align*}
Because $A$ and $\tilde{M}_{t}$ match up to three and a half moments, one can easily check that
\begin{align*}
K_{n,(ij)}^{\theta}=\begin{cases}0&n=1,2,3\\O(N^{-2-\delta_{t}})&n=4\\O(N^{-n/2})&5\leq n\leq\ell\end{cases}
\end{align*}
where $\delta_{t}$ is from Lemma \ref{lemma:ex35M}. Next, $\mathcal{E}_{\ell}$ satisfies the usual Taylor series estimate below, where $N^{-\ell/2}$ comes from bounding the $t^{\ell+1}$ coefficient in $\E e^{t\mathcal{H}_{ij}^{\tau}}-\E e^{t\mathcal{H}_{ij}^{0}}$ uniformly in $\tau$ as in the $K_{n,(ij)}^{\theta}$ estimate above:
\begin{align*}
|\mathcal{E}_{\ell}|&=O\left(\sum_{i,j=1}^{N}N^{-\frac{\ell}{2}}\E\left[\sup_{\theta\in[0,1]}\left|\left(\frac{\d}{\d\mathcal{H}^{\theta}_{ij}}\right)^{\ell+1}F(\mathcal{H}^{\theta})\right|\left(1+|\mathcal{H}^{\theta}_{ij}-\mathcal{H}^{0}_{ij}|^{\ell+1}\right)\right]\right)\\
&=O\left(\sum_{i,j=1}^{N}N^{-\frac{\ell}{2}}\left\{\E\left[\sup_{\theta\in[0,1]}\left|\left(\frac{\d}{\d\mathcal{H}^{\theta}_{ij}}\right)^{\ell+1}F(\mathcal{H}^{\theta})\right|^{2}\right]\right\}^{1/2}\right).
\end{align*}
The second bound uses Cauchy-Schwarz and moment bounds on $\mathcal{H}_{ij}^{\theta}$ uniformly in $\theta\in[0,1]$ and $i,j$. So, we have the following.
\begin{lemma}\label{lemma:comparisonlemma}
Suppose $F:\C^{2N\times2N}\to\C$ is smooth and satisfies the following for any $\ell>0$ fixed:
\begin{align}
\sup_{i,j=1,\ldots,N}\sup_{4\leq n\leq\ell}\E\left[\left(\frac{\d}{\d\mathcal{H}^{\theta}_{ij}}\right)^{n}F(\mathcal{H}^{\theta})\right]&=O(N^{\frac{\delta_{t}}{3}}N^{C\delta_{V}}),\label{eq:comparisonlemma1}\\
\sup_{i,j=1,\ldots,N}\left\{\E\left[\sup_{\theta\in[0,1]}\left|\left(\frac{\d}{\d\mathcal{H}^{\theta}_{ij}}\right)^{\ell+1}F(\mathcal{H}^{\theta})\right|^{2}\right]\right\}^{1/2}&=O(N^{C\delta_{V}}).\label{eq:comparisonlemma2}
\end{align}
Then for $\delta_{V}>0$ small, we have $\E F(\mathcal{H}^{1})=\E F(\mathcal{H}^{0})+O(N^{-\frac{\delta_{t}}{2}})$.
\end{lemma}
Now, define the function
\begin{align*}
I_{\delta_{V}}(\theta,f^{(j)})&:=\frac{1}{2\pi}\int_{\C}\Delta f^{(j)}(z)\int_{N^{-1-10\delta_{V}}}^{N^{-1+10\delta_{V}}}\langle\mathrm{Im}G_{z}^{\theta}(\eta)-\mathrm{Im}m_{z}(i\eta)\rangle\d\eta\d^{2}z,
\end{align*}
where we recall $f^{(j)}(z)=N\tilde{f}^{(j)}(N^{1/2}(z-z_{j}^{0}))$ with $\tilde{f}^{(j)}\in C^{\infty}_{c}(\C)$. Let $\mathcal{I}:=[N^{-1-10\delta_{V}},N^{-1+10\delta_{V}}]$ and $m:=m_{R}+m_{L}$ for convenience. We compute
\begin{align*}
&\E\left[\prod_{j=1}^{m}I_{\delta_{V}}(\theta,f^{(j)})\right]=\int_{\C^{m}}\int_{\mathcal{I}^{m}}\E\left[\prod_{j=1}^{m}\Delta f^{(j)}(z_{j})\langle\mathrm{Im}G_{z_{j}}^{\theta}(\eta_{j})-\mathrm{Im}m_{z_{j}}(i\eta_{j})\rangle\right]\prod_{j=1}^{m}\d\eta_{j}\d^{2}z_{j}\\
&=\int_{\C^{m}}N^{m}\int_{\mathcal{I}^{m}}\E\left[\prod_{j=1}^{m}\Delta \tilde{f}^{(j)}(w_{j})\langle\mathrm{Im}G_{z_{j}^{0}+N^{-1/2}w_{j}}^{\theta}(\eta_{j})-\mathrm{Im}m_{z_{j}^{0}+N^{-1/2}w_{j}}(i\eta_{j})\rangle\right]\prod_{j=1}^{m}\d\eta_{j}\d^{2}w_{j},
\end{align*}
where the second line holds by change-of-variables $w_{j}=N^{1/2}(z-z_{j}^{0})$ and the observation $\Delta f^{(j)}(z)=N\Delta\tilde{f}^{(j)}(N^{1/2}(z-z_{j}^{0}))$. Thus, by Lemmas \ref{lemma:Vcomparisonreduction} and \ref{lemma:comparisonlemma}, we are interested in the function
\begin{align*}
F(\mathcal{H}^{\theta}):=\prod_{j=1}^{m}\Delta \tilde{f}^{(j)}(w_{j})\langle\mathrm{Im}G_{z_{j}^{0}+N^{-1/2}w_{j}}^{\theta}(\eta_{j})-\mathrm{Im}m_{z_{j}^{0}+N^{-1/2}w_{j}}(i\eta_{j})\rangle.
\end{align*}
In particular, by Lemmas \ref{lemma:Vcomparisonreduction} and \ref{lemma:comparisonlemma}, to complete the proof of Theorem \ref{theorem:main}, it suffices to prove \eqref{eq:comparisonlemma1}-\eqref{eq:comparisonlemma2} for this choice of $F$. Indeed, this would give 
\begin{align*}
&\E\left[\prod_{j=1}^{m_{R}+m_{L}}I_{\delta_{V}}(A,f^{(j)})\right]-\E\left[\prod_{j=1}^{m_{R}+m_{L}}I_{\delta_{V}}(\tilde{M}_{t},f^{(j)})\right]\\
&=\E\left[\prod_{j=1}^{m}I_{\delta_{V}}(1,f^{(j)})\right]-\E\left[\prod_{j=1}^{m}I_{\delta_{V}}(0,f^{(j)})\right]\\
&=\int_{\C^{m}}N^{m}\int_{\mathcal{I}^{m}}\left[\E F(\mathcal{H}^{1})-\E F(\mathcal{H}^{0})\right]\prod_{j=1}^{m}\d\eta_{j}\d^{2}w_{j}\\
&=O\left(N^{-\frac{\delta_{t}}{2}}\int_{\mathbb{K}^{m}}N^{m}\int_{\mathcal{I}^{m}}\prod_{j=1}^{m}\d\eta_{j}\d^{2}w_{j}\right)=O(N^{-\frac{\delta_{t}}{2}}N^{C_{m}\delta_{V}}).
\end{align*}
The last line follows by Lemma \ref{lemma:comparisonlemma} and the fact that $\tilde{f}^{(j)}$ is compactly supported in $\C$ (here, $\mathbb{K}\subseteq\C$ is a compact subset). Then, we use the fact that $|\mathcal{I}|\leq 2N^{-1+10\delta_{V}}$. Now, by the Leibniz rule, to show \eqref{eq:comparisonlemma1}-\eqref{eq:comparisonlemma2} with the above choice of $F$, it suffices to show that for any $n\geq0$ and locally uniformly in $w_{j}\in\C$ and $\eta_{j}\in\mathcal{I}$, we have 
\begin{align}
\left(\frac{\d}{\d\mathcal{H}^{\theta}_{ij}}\right)^{n}\mathcal{A}_{j}(\mathcal{H}^{\theta},w_{j},\eta_{j})=\mathcal{O}\left(N^{C_{n}\delta_{V}}\right),\label{eq:comparisonderivatives}
\end{align}
where $\mathcal{A}_{j}(\mathcal{H}^{\theta},w_{j},\eta_{j})$ either has the form $g(V(z_{j}^{0}+N^{-1/2}w_{j},T))$ with $T$ finite-rank and Hermitian and $g\in C^{\infty}_{c}(\C)$ or the form $\langle\mathrm{Im}G_{z_{j}^{0}+N^{-1/2}w_{j}}^{\theta}(\eta_{j})-\mathrm{Im}m_{z_{j}^{0}+N^{-1/2}w_{j}}(i\eta_{j})\rangle$. (Such $\mathcal{A}_{j}(\mathcal{H}^{\theta},w_{j},\eta_{j})$ have $m$-th order derivatives that are deterministically $O(N^{C_{m}})$ if $\eta_{j}\geq N^{-1-10\delta_{V}}$ and $\delta_{V}>0$ is small; this can be checked by elementary resolvent perturbation identities. Thus, an $\mathcal{O}$-estimate is enough. We also note that straightforward regularity bounds in $\theta$ of $\mathcal{A}_{j}(\mathcal{H}^{\theta},w_{j},\eta_{j})$ let us use a net argument to extend \eqref{eq:comparisonderivatives} to the same estimate but with a supremum over $\theta\in[0,1]$, upon possibly changing the constant $C_{n}>0$.) 

We are left to prove \eqref{eq:comparisonderivatives}. First, the assumption $g\in C^{\infty}_{c}(\C)$ and the local law (see Theorem 2.6 in \cite{cipolloni2023optimal}) imply \eqref{eq:comparisonderivatives} for $n=0$, so we focus on $n\geq1$. By definition of $V(z,T)$, it suffices to show that locally uniformly in $z\in\C$ and uniformly in $\eta\in[N^{-1-10\delta_{V}},N^{-1+10\delta_{V}}]$, for any finite-rank, Hermitian $T$, we have
\begin{align}
\left(\frac{\d}{\d\mathcal{H}^{\theta}_{ij}}\right)^{n}\tr TG^{\theta}_{z}(\eta)&=\mathcal{O}(N^{C_{n}\delta_{V}})\label{eq:trtgderivative}\\
\left(\frac{\d}{\d\mathcal{H}^{\theta}_{ij}}\right)^{n}\tr \mathrm{Im}G^{\theta}_{z}(\eta)&=\mathcal{O}(N^{1+C_{n}\delta_{V}})\label{eq:trgderivative}.
\end{align}
For this, we record the following consequence of resolvent perturbation identities:
\begin{align*}
\left(\frac{\d}{\d\mathcal{H}^{\theta}_{ij}}\right)^{n}G_{z}^{\theta}(\eta)&=C_{n}G_{z}^{\theta}(\eta)\left[\Delta_{ij}G_{z}^{\theta}(\eta)\right]^{n},\quad (\Delta_{ij})_{k\ell}=\delta_{ik}\delta_{\ell j}+\delta_{i\ell}\delta_{jk}.
\end{align*}
Using this identity and a spectral decomposition $T=\sum_{a=1}^{\ell}\tau_{a}\mathbf{w}_{a}\mathbf{w}_{a}^{\ast}$ we have 
\begin{align*}
\left(\frac{\d}{\d\mathcal{H}^{\theta}_{ij}}\right)^{n}\tr TG_{z}^{\eta}(\eta)&=C_{n}\tr TG_{z}^{\theta}(\eta)[\Delta_{ij}G_{z}^{\theta}(\eta)]^{n}G_{z}^{\theta}(\eta)\\
&=\sum_{a=1}^{\ell}C_{n}\tau_{a}\mathbf{w}_{a}^{\ast}G_{z}^{\theta}(\eta)\mathbf{e}_{i}\mathbf{e}_{j}^{\ast}G_{z}^{\theta}(\eta)\mathbf{e}_{i}\ldots\mathbf{e}_{j}^{\ast}G_{z}^{\theta}(\eta)\mathbf{w}_{a},
\end{align*}
where the product has at most $n+2$ many factors of the form $\mathbf{x}^{\ast}G_{z}^{\theta}(\eta)\mathbf{y}$ with $\mathbf{x},\mathbf{y}$ deterministic unit vectors. By rigidity and delocalization, i.e. Proposition \ref{prop:rigiditydeloc}, for eigenvalues $\xi_{k}$ and eigenvectors $\mathbf{u}_{k}$ of $G_{z}^{\theta}(\eta)$, we have the following for some $c>0$ independent of $N$:
\begin{align*}
\mathbf{x}^{\ast}G_{z}^{\theta}(\eta)\mathbf{y}&=\sum_{k}\frac{\mathbf{x}^{\ast}\mathbf{u}_{k}\mathbf{u}_{k}^{\ast}\mathbf{y}}{\xi_{k}-i\eta}=\mathcal{O}\left(\frac{1}{N}\sum_{k}\frac{1}{\xi_{k}-i\eta}\right)+\sum_{k:|\xi_{k}|\geq c}\frac{\mathbf{x}^{\ast}\mathbf{u}_{k}\mathbf{u}_{k}^{\ast}\mathbf{y}}{\xi_{k}-i\eta}\\
&=\mathcal{O}(N^{-1}\eta^{-1})+O(1)=\mathcal{O}(N^{10\delta_{V}}).
\end{align*}
The last term in the first line has the form $\mathbf{x}^{\ast}Y\mathbf{y}$, where $\|Y\|_{\mathrm{op}}=O(1)$, hence it is $O(1)$.

Combining the previous two displays gives \eqref{eq:trtgderivative}. To prove \eqref{eq:trgderivative}, we write $\mathrm{Im}G_{z}^{\theta}(\eta)=G_{z}^{\theta}(\eta)-G_{z}^{\theta}(\eta)^{\ast}$. In particular, it suffices to prove \eqref{eq:trgderivative} but with $\mathrm{Im}G_{z}^{\theta}(\eta)$ replaced by $G_{z}^{\theta}(\eta)$ and $G_{z}^{\theta}(\eta)^{\ast}$. We prove \eqref{eq:trgderivative} with $\mathrm{Im}G_{z}^{\theta}(\eta)$ replaced by $G_{z}^{\theta}(\eta)$; for $G_{z}^{\theta}(\eta)^{\ast}$, the same argument works. Note that 
\begin{align*}
\left(\frac{\d}{\d\mathcal{H}^{\theta}_{ij}}\right)^{n}\tr G^{\theta}_{z}(\eta)&=\sum_{k=1}^{2N}\left(\frac{\d}{\d\mathcal{H}^{\theta}_{ij}}\right)^{n}\tr \mathbf{e}_{k}\mathbf{e}_{k}^{\ast}G^{\theta}_{z}(\eta).
\end{align*}
The estimate \eqref{eq:trtgderivative} with $T=\mathbf{e}_{k}\mathbf{e}_{k}^{\ast}$ implies that the RHS of the previous display is $\mathcal{O}(N^{1+C_{n}\delta_{V}})$. This proves \eqref{eq:trgderivative}, so as mentioned right before \eqref{eq:trtgderivative}, this shows \eqref{eq:comparisonderivatives} and thus completes the proof of Theorem \ref{theorem:main}. \qed
\subsection{Proof of Proposition \ref{prop:levelrepulsion}}
This argument has three steps. First, we reduce to proving the estimate for $\xi_{2}^{z}(t)$, the second-smallest singular value of the matching matrix $\tilde{M}_{t}$ from Lemma \ref{lemma:ex35M}. Then, we further reduce to proving the estimate for the second-smallest singular value of a Ginibre matrix. Finally, we cite known estimates for the Ginibre case. As argued in the proof of Lemma \ref{lemma:levelrepulsion}, we have 
\begin{align*}
\mathbb{P}(\xi_{2}^{z}\leq N^{-1-\delta})\lesssim\E\left[\phi\left(\frac{\eta}{2}\mathrm{Im}\tr G_{z}(\eta)\right)\right].
\end{align*}
Now, define $G_{t,z}(\eta):=[\mathcal{H}_{t,z}-i\eta]^{-1}$, where $\mathcal{H}_{t,z}$ is defined by
\begin{align*}
\mathcal{H}_{t,z}&:=\begin{pmatrix}0&\tilde{M}_{t}-z\\(\tilde{M}_{t}-z)^{\ast}&0\end{pmatrix}.
\end{align*}
Since $\phi:\R\to\R$ is smooth and compactly supported, we can use \eqref{eq:trgderivative} and $\eta\leq N^{-1}$ and Lemma \ref{lemma:comparisonlemma} to deduce 
\begin{align*}
\E\left[\phi\left(\frac{\eta}{2}\mathrm{Im}\tr G_{z}(\eta)\right)\right]-\E\left[\phi\left(\frac{\eta}{2}\mathrm{Im}\tr G_{t,z}(\eta)\right)\right]=O(N^{-\delta_{t}/2}).
\end{align*}
(The input \eqref{eq:comparisonlemma2} to use Lemma \ref{lemma:comparisonlemma} can be shown, again, by \eqref{eq:trgderivative} and a standard net argument.) As argued in the proof of Lemma \ref{lemma:levelrepulsion}, we have 
\begin{align*}
\E\left[\phi\left(\frac{\eta}{2}\mathrm{Im}\tr G_{t,z}(\eta)\right)\right]\lesssim\mathbb{P}(\xi_{2}^{z}(t)\leq N^{-1-\delta_{2}}),
\end{align*}
where $\xi_{2}^{z}(t)$ is the second-smallest (positive) singular value of $\tilde{M}_{t}$, and $\delta_{2}$ is any fixed number strictly smaller than $\delta_{1}$. Summarizing so far, we have proven
\begin{align*}
\mathbb{P}(\xi_{2}^{z}\leq N^{-1-\delta})\lesssim\mathbb{P}(\xi_{2}^{z}(t)\leq N^{-1-\delta_{2}})+N^{-\delta_{t}/2}.
\end{align*}
We now use Theorem 3.2 in \cite{CL19}. This gives the following. There exists a matrix $W(t)$ whose entries are i.i.d. standard complex Gaussians mutiplied by $(1+t)^{1/2}$ such that if $\{\mu_{k}(t)\}_{k}$ are the singular values of $W(t)$ with $\mu_{k}(t)\leq\mu_{k+1}(t)$ for $k\geq1$ and $\mu_{-k}(t)=\mu_{k}(t)$, then 
\begin{align*}
\mathbb{P}\left(|\xi_{2}^{z}(t)-\mu_{2}^{z}(t)|\leq N^{-1-\kappa}\right)\lesssim N^{-D}
\end{align*}
for any fixed $D>0$, where $\kappa>0$ is fixed. (This requires that $\langle \mathrm{Im}G_{t,z}(\eta)\rangle=\mathcal{O}(1)$ for all $\eta\geq N^{-1+\epsilon}$ for some $\epsilon>0$ fixed. Such an estimate follows by the local law in (3.6) in \cite{cipolloni2023mesoscopic} combined with (3.7) in \cite{cipolloni2023mesoscopic}.) Since $\kappa>0$ is fixed, for small enough $\delta$, we deduce
\begin{align*}
\mathbb{P}(\xi_{2}^{z}\leq N^{-1-\delta})\lesssim\mathbb{P}(\mu_{2}(t)\leq N^{-1-\delta_{2}})+N^{-\delta_{t}/2}=\mathbb{P}((1+t)^{1/2}\mu_{2}(0)\leq N^{-1-\delta_{2}})+N^{-\delta_{t}/2}.
\end{align*}
We now use Theorem 2.10 in \cite{EJ23}; this shows that the first term on the far RHS of the previous display is $O(N^{\epsilon}N^{-8\delta_{2}})$ with $\epsilon>0$ fixed but otherwise arbitrary. If we choose $\epsilon>0$ small enough and $\delta_{2}$ sufficiently close to $\delta_{1}$, then the far RHS of the previous display becomes $O(N^{-2.1\delta})$ for $\delta>0$ small enough. This completes the proof.
%
%
%
\section{Green's function estimates}\label{sec:green-function}
The purpose of this section is to record various auxiliary bounds for Green's functions. These bounds are mostly standard. First, recall the Hermitization $\mathcal{H}_{z}$ of $A$ and Green's function:
\begin{align*}
\mathcal{H}_{z}:=\begin{pmatrix}0&A-z\\(A-z)^{\ast}&0\end{pmatrix}\quad\mathrm{and}\quad G_{z}(\eta):=[\mathcal{H}_{z}-i\eta]^{-1}
\end{align*}
for $\eta>0$. Green's function can be written as
\[
G_z(\eta) = \begin{pmatrix} 
i\eta\tilde{H}_z(\eta) & \tilde{H}_z(\eta)(A-z)\\
(A^\ast-\bar{z})\tilde{H}_z(\eta) & i\eta H_z(\eta)
\end{pmatrix},
\]
where
\begin{align*}
H_z(\eta) &= \left[(A-z)^\ast(A-z)+\eta^2\right]^{-1},\\
\tilde{H}_z(\eta) &= \left[(A-z)(A-z)^\ast+\eta^2\right]^{-1}.
\end{align*}
We assume $A$ has entries that satisfy $\E A_{ij}=0$ and $\E A_{ij}^{2}=N^{-1}$ and $\E|A_{ij}|^{p}\leq C_{p}N^{-p/2}$ for all $p\geq2$. As in Definition \ref{defn:Vfunction}, we let $\xi_{k}^{z}$ be eigenvalues of the Hermitization $\mathcal{H}_{z}$, and we let $\mathbf{u}_{k}^{z}$ be the corresponding eigenvalues. We adopt the convention $\xi_{k}^{z}\geq0$ for all $k\in\llbracket1,N\rrbracket$, and $\xi_{k}^{z}=-\xi_{-k}^{z}$ for any $k\in\llbracket-N,-1\rrbracket$. We also impose $\xi_{k}^{z}\leq\xi_{\ell}^{z}$ for $k\leq\ell$.

We now introduce a deterministic local law approximation for $G_{z}$. Define 
\begin{align*}
M_{z}(w):=\begin{pmatrix}m_{z}(w)&-zu_{z}(w)\\-\overline{z}u_{z}(w)&m_{z}(w)\end{pmatrix},\quad u_{z}(w):=\frac{m_{z}(w)}{w+m_{z}(w)},
\end{align*}
where $m_{z}(w)$ satisfies the self-consistent equation
\begin{align}\label{eq:mz-equation}
-\frac{1}{m_{z}(w)}=w+m_{z}(w)-\frac{|z|^{2}}{w+m_{z}(w)},\quad \mathrm{Im}[m_{z}(w)]\mathrm{Im}[w]>0.
\end{align}
%
\begin{prop}\label{prop:locallaw}
Fix any deterministic matrix $Y\in M_{2N}(\C)$ with $\|Y\|_{\mathrm{op}}=O(1)$ and any deterministic unit vectors $\mathbf{x},\mathbf{y}\in\C^{2N}$. Fix any $\tau>0$. Fix any $\eta\geq N^{-1+\epsilon}$ for any $\epsilon>0$, and fix $|z|\leq1-\tau$. We have 
\begin{align*}
\langle Y(G_{z}(\eta)-M_{z}(i\eta))\rangle&=\mathcal{O}\left(N^{-1}\eta^{-1}\right)\\
\mathbf{x}^{\ast}(G_{z}(\eta)-M_{z}(i\eta))\mathbf{y}&=\mathcal{O}\left(N^{-1/2}\eta^{-1/2}\right).
\end{align*}
In particular, we have \eqref{assn:3.1}-\eqref{assn:3.2}.
\end{prop}
\begin{proof}
The first two bounds are (3.6) in \cite{cipolloni2023mesoscopic}. We now show \eqref{assn:3.1}; the proof of \eqref{assn:3.2} is similar. We use the first two bounds with $Y=\begin{pmatrix}0&0\\0&I_{N}\end{pmatrix}$ and $\mathbf{x}^{\ast}=\begin{pmatrix}0&\mathbf{w}_{1}^{\ast}\end{pmatrix}$ and $\mathbf{y}^{\ast}=\begin{pmatrix}0&\mathbf{w}_{2}^{\ast}\end{pmatrix}$, so
\begin{align*}
i\eta\mathbf{w}_{1}^{\ast}H_{z}(\eta)\mathbf{w}_{2}-m_{z}(i\eta)\mathbf{w}_{1}^{\ast}\mathbf{w}_{2}&=\mathcal{O}(N^{-1/2}\eta^{-1/2})\\
\langle i\eta H_{z}(\eta)\rangle-m_{z}(i\eta)&=\mathcal{O}(N^{-1}\eta^{-1}).
\end{align*}
This gives \eqref{assn:3.1}.
\end{proof}
Define the following limiting empirical spectral density of $\mathcal{H}_{z}$ for $E\in\R$:
\begin{align*}
\rho_{z}(E):=\lim_{\eta\to0^{+}}\frac{1}{\pi}|\mathrm{Im}m_{z}(E+i\eta)|.
\end{align*}
With this, we define two objects. The first is the quantile $\gamma_{j,z}$ for any $|j|\leq N$ via the equation
\begin{align*}
\frac{j+N}{2N}=\int_{-\infty}^{\gamma_{j,z}}\rho_{z}(E)\d E.
\end{align*}
Note that $\gamma_{j,z}=-\gamma_{-j,z}$. The second is the $\kappa$-bulk $\mathbb{B}_{\kappa,z}:=\{E\in\R:\rho_{z}(E)\geq \kappa^{1/3}\}$, where $\kappa>0$. These notations are taken from Section 2 of \cite{cipolloni2023optimal}. By Example 2.5 in \cite{cipolloni2023optimal}, if $|z|\leq1-\tau$ with $\tau>0$ fixed, then we can take $\kappa>0$ small enough depending only on $\tau$ such that $\mathbb{B}_{\kappa,z}$ contains an interval around $0$.
\begin{prop}[Rigidity and delocalization, Corollaries 2.8 and 2.9 in \cite{AEK18}]\label{prop:rigiditydeloc}
Suppose $|z|\leq1-\tau$ for some $\tau>0$. Fix any $\kappa>0$ small enough, and suppose $\gamma_{k,z}\in\mathbb{B}_{\kappa,z}$. We have $|\xi_{k}^{z}-\gamma_{k,z}|=\mathcal{O}(N^{-1})$. Moreover, we have $\|\mathbf{u}_{k}^{z}\|_{\infty}=\mathcal{O}(N^{-1/2})$.
\end{prop}
Next, we note that
\begin{align}
\mathrm{Im}[M_{z}(E)]&=\mathrm{Im}[m_{z}(E)]-\frac{E+o(E)}{\mathrm{Im}[m_{z}(0)]}\begin{pmatrix}0&-z\\-\overline{z}&0\end{pmatrix},\quad\mathrm{Im}[m_{z}(0)]=\sqrt{1-|z|^{2}}.\label{eq:Mzbound}
\end{align}
%

\begin{prop}[Eigenstate thermalization hypothesis, Theorem 2.7 in \cite{cipolloni2023optimal}]\label{prop:eth}
Fix any deterministic matrix $Y\in M_{2N}(\C)$ such that $\|Y\|_{\mathrm{op}}=O(1)$. We have 
\begin{align*}
\max_{|i|,|j|\leq N}\left|\mathbf{u}_{i}^{z,\ast}Y\mathbf{u}_{j}^{z}-\delta_{ij}\frac{\langle\mathrm{Im}M_{z}(\gamma_{j})Y\rangle}{\langle\mathrm{Im}M_{z}(\gamma_{j})\rangle}-\delta_{-ij}\frac{\langle\mathrm{Im}M_{z}(\gamma_{j})(E_{11}-E_{22})Y\rangle}{\langle\mathrm{Im}M_{z}(\gamma_{j})\rangle}\right|=\mathcal{O}(N^{-\frac12}),
\end{align*}
where $\mathcal{O}$ is with respect to randomness of $A$, while $E_{11},E_{22}$ are the $2N\times2N$ block matrices
\begin{align*}
E_{11}:=\begin{pmatrix}I_{N}&0\\0&0\end{pmatrix}\quad\mathrm{and}\quad E_{22}:=\begin{pmatrix}0&0\\0&I_{N}\end{pmatrix}.
\end{align*}
Thus, by \eqref{eq:Mzbound}, if $\langle YE_{11}\rangle,\langle YE_{22}\rangle=0$, then
\begin{align*}
\mathbf{u}_{i}^{z,\ast}Y\mathbf{u}_{j}^{z}=\delta_{|i|=|j|}O(N^{-1}i)+\mathcal{O}(N^{-1/2}).
\end{align*}
\end{prop}

\begin{prop}
Suppose $|z_1|, |z_2|\le 1-\tau$ for some $\tau>0$ and $\eta_1,\eta_2\in[N^{-\frac13+\epsilon_0}, 10]$. Define $\eta^\ast=\max\{\eta_1,\eta_2\}$. Then there exists a constant $c>0$ such that for any $D>0$ we have
\begin{align}
\pp\left(\ntr{H_{z_1}(\eta_1)H_{z_2}(\eta_2)} \le c\left(\eta^\ast\right)^{-2}\right) \le N^{-D}\label{eq:2G-LL-1}\\
\pp\left(\ntr{\tilde{H}_{z_1}(\eta_1)\tilde{H}_{z_2}(\eta_2)} \le c\left(\eta^\ast\right)^{-2}\right) \le N^{-D}\label{eq:2G-LL-2}\\
\pp\left(\ntr{H_{z_1}(\eta_1)\tilde{H}_{z_2}(\eta_2)} \le \frac{c}{\eta^\ast\left(\eta^\ast+|z_1-z_2|^2\right)}\right)\le N^{-D}\label{eq:2G-LL-3}
\end{align}
for large enough $N$ uniformly in $z_1, z_2, \eta_1, \eta_2$.
\end{prop}

\begin{proof}
Define $\eta_\ast = \min\{\eta_1,\eta_2\}$. To show \eqref{eq:2G-LL-1}, we note that
\[
\ntr{H_{z_1}(\eta_1)H_{z_2}(\eta_2)} = -\frac{1}{\eta_1\eta_2}\ntr{G_{z_1}(\eta_1)E_{22}G_{z_2}(\eta_2) E_{22}}.
\]
Now we use Theorem 3.5 of \cite{cipolloni2021fluctuation} to approximate this trace by a deterministic quantity
\[
\ntr{H_{z_1}(\eta_1)H_{z_2}(\eta_2)} = -\frac{1}{\eta_1\eta_2} \ntr{\mathcal{B}^{-1}[M_{z_1}(\eta_1)E_{22}M_{z_2}(\eta_2)] E_{22}} + \oh\left(\frac{N^{-3\epsilon_0}}{\eta_1\eta_2}\right).
\]
Here $\mathcal{B}: M_{2N}(\C)\rightarrow M_{2N}(\C)$ is a linear operator given by
\[
\mathcal{B}(\cdot) = Id - M_{z_1}(\eta_1)\mathcal{S}(\cdot)M_{z_2}(\eta_2).
\]
By a straightforward computation we see that
\[
-\ntr{\mathcal{B}^{-1}[M_{z_1}(\eta_1)E_{22}M_{z_2}(\eta_2)] E_{22}} = \frac{|m_{z_1}(\eta_1)m_{z_2}(\eta_2)|}{|1-z_1\bar{z_2}u_{z_1}(\eta_1)u_{z_2}(\eta_2)|^2 - m_{z_1}(\eta_1)^2 m_{z_2}(\eta_2)^2} \ge c
\]
for some constant $c$. Here we used that $|m_{z_j}(\eta_j)|$ is bounded below by a constant; for $\eta\geq c_{0}$ for any fixed $c_{0}>0$, this follows by continuity of solutions to \eqref{eq:mz-equation}, and for $\eta\leq c_{0}$, this follows by Lemma 3.2 in \cite{cipolloni2023mesoscopic} if $c_{0}>0$ is small enough. We also used that the denominator is bounded above by a constant, which can be readily checked using properties of $m_{z_{j}}(\eta_{j})$, and positive since we know that $\ntr{H_{z_1}(\eta_1)H_{z_2}(\eta_2)}>0$. This concludes the proof of \eqref{eq:2G-LL-1}. The bound \eqref{eq:2G-LL-2} is proved the same way.

Now we prove \eqref{eq:2G-LL-3}. If $|z_1-z_2|<N^{-\frac{\epsilon_0}{2}}$, we use Theorem 3.3 of \cite{cipolloni2023mesoscopic} and if $|z_1-z_2|\ge N^{-\frac{\epsilon_0}{2}}$, we use Theorem 3.5 of \cite{cipolloni2021fluctuation} to get
\begin{align*}
\ntr{H_{z_1}(\eta_1)\tilde{H}_{z_2}(\eta_2)} &= -\frac{1}{\eta_1\eta_2}\ntr{G_{z_1}(\eta_1)E_{12}G_{z_2}(\eta_2) E_{21}}\\
&= -\frac{1}{\eta_1\eta_2} \ntr{\mathcal{B}^{-1}[M_{z_1}(\eta_1)E_{12}M_{z_2}(\eta_2)] E_{21}} + \oh\left(\frac{N^{\epsilon_0}}{N\eta_{\ast}^{2}}\frac{1}{\eta_1\eta_2}\right).
\end{align*}
(The quantity $\hat{\beta}_{\ast}$ in Theorem 3.5 of \cite{cipolloni2021fluctuation} is bounded below uniformly by Lemma 6.1 in \cite{cipolloni2021fluctuation}.) Since $\eta_{\ast}\geq N^{-1/3+\epsilon_{0}}$, the error term $\mathcal{O}(N^{\epsilon_0}N^{-1}\eta_{\ast}^{-4})$ is much smaller than $\eta_{\ast}^{-1}$. Thus, we can drop the $\oh$-term in the second line. Again, we can directly compute
\begin{align}
&-\ntr{\mathcal{B}^{-1}[M_{z_1}(\eta_1)E_{12}M_{z_2}(\eta_2)] E_{21}}\nonumber \\
&= \frac{|m_{z_1}(\eta_1)m_{z_2}(\eta_2)|\left(1 - (|z_1|^2 u_{z_1}(\eta_1)^2 - m_{z_1}(\eta_1)^2)(|z_2|^2 u_{z_2}(\eta_2)^2 - m_{z_2}(\eta_2)^2)\right)}{|1-z_1\bar{z_2}u_{z_1}(\eta_1)u_{z_2}(\eta_2)|^2 - m_{z_1}(\eta_1)^2 m_{z_2}(\eta_2)^2}.\label{eq:HHt-LL-main-term}
\end{align}
Now from the self-consistent equation \eqref{eq:mz-equation} for $m_z(\eta)$, we see that
\[
0\le |z|^2 u_z(\eta)^2 - m_z(\eta)^2 \le 1 - \eta |m_z(\eta)| \le 1-c\eta.
\]
Thus
\[
-\ntr{\mathcal{B}^{-1}[M_{z_1}(\eta_1)E_{12}M_{z_2}(\eta_2)] E_{21}} \ge c(\eta_1 + \eta_2)
\]
for some constant $c$. The denominator of \eqref{eq:HHt-LL-main-term} is the determinant of $\mathcal{B}$. From the calculations in Appendix B of \cite{cipolloni2023mesoscopic} we see that
\[
\det \mathcal{B} \lesssim \eta_1+\eta_2 + |z_1-z_2|^2.
\]
This concludes the proof of \eqref{eq:2G-LL-3}.
\end{proof}

\bibliographystyle{abbrv}
\bibliography{leftrightevector_4}

\begin{thebibliography}{10}

\bibitem{aggarwal2021eigenvector}
A.~Aggarwal, P.~Lopatto, and J.~Marcinek.
\newblock Eigenvector statistics of {L}{\'e}vy matrices.
\newblock {\em The Annals of Probability}, 549(4):1778--1846, 2021.

\bibitem{AEK18}
O.~H. Ajanki, L.~Erd{\H{o}}s, and T.~Kr\"{u}ger.
\newblock Stability of the matrix {D}yson equation and random matrices with
  correlations.
\newblock {\em Probability Theory and Related Fields}, 173:293--373, 2019.

\bibitem{alt2018local}
J.~Alt, L.~Erd{\H{o}}s, and T.~Kr{\"u}ger.
\newblock Local inhomogeneous circular law.
\newblock {\em The Annals of Applied Probability}, 28(1):148--203, 2018.

\bibitem{alt2021spectral}
J.~Alt, L.~Erd{\H{o}}s, and T.~Kr{\"u}ger.
\newblock Spectral radius of random matrices with independent entries.
\newblock {\em Probability and Mathematical Physics}, 2(2):221--280, 2021.

\bibitem{benigni2023fluctuations}
L.~Benigni, N.~Chen, P.~Lopatto, and X.~Xie.
\newblock Fluctuations in quantum unique ergodicity at the spectral edge.
\newblock {\em arXiv preprint arXiv:2303.11142}, 2023.

\bibitem{benigni2022fluctuations}
L.~Benigni and P.~Lopatto.
\newblock Fluctuations in local quantum unique ergodicity for generalized
  wigner matrices.
\newblock {\em Communications in Mathematical Physics}, 391(2):401--454, 2022.

\bibitem{BD20}
P.~Bourgade and G.~Dubach.
\newblock The distribution of overlaps between eigenvectors of {G}inibre
  matrices.
\newblock {\em Probability Theory and Related Fields}, 177:397--464, 2020.

\bibitem{bourgade2017sparse}
P.~Bourgade, J.~Huang, and H.-T. Yau.
\newblock Eigenvector statistics of sparse random matrices.
\newblock {\em Electronic Journal of Probability}, 22(64):1--38, 2017.

\bibitem{bourgade2017eigenvector}
P.~Bourgade and H.-T. Yau.
\newblock The eigenvector moment flow and local quantum unique ergodicity.
\newblock {\em Communications in Mathematical Physics}, 350:231--278, 2017.

\bibitem{bourgade2018random}
P.~Bourgade, H.-T. Yau, and J.~Yin.
\newblock Random band matrices in the delocalized phase, {I}: {Q}uantum unique
  ergodicity and universality.
\newblock {\em Communications on Pure and Applied Mathematics: A Journal Issued
  by the Courant Institute of Mathematical Sciences}, 13(7):1526--1596, 2020.

\bibitem{CL19}
Z.~Che and P.~Lopatto.
\newblock Universality of the least singular value for sparse random matrices.
\newblock {\em Electronic Journal of Probability}, 24:1--53, 2019.

\bibitem{cipolloni2023optimal}
G.~Cipolloni, L.~Erd{\H{o}}s, J.~Henheik, and D.~Schr{\"o}der.
\newblock Optimal lower bound on eigenvector overlaps for non-{H}ermitian
  random matrices.
\newblock {\em arXiv preprint arXiv:2301.03549}, 2023.

\bibitem{CES20}
G.~Cipolloni, L.~Erd{\H{o}}s, and D.~Schr{\"o}der.
\newblock Towards the bulk universality of non-hermitian random matrices.
\newblock {\em arXiv preprint arXiv:1909.06350}, 2020.

\bibitem{CES2020}
G.~Cipolloni, L.~Erd{\H{o}}s, and D.~Schr{\"o}der.
\newblock Eigenstate thermalization hypothesis for {W}igner matrices.
\newblock {\em Communications in Mathematical Physics}, 388:1005--1048, 2021.

\bibitem{cipolloni2021fluctuation}
G.~Cipolloni, L.~Erd{\H{o}}s, and D.~Schr{\"o}der.
\newblock Fluctuation around the circular law for random matrices with real
  entries.
\newblock {\em Electronic Journal of Probability}, 26(24):1--61, 2021.

\bibitem{cipolloni2022normal}
G.~Cipolloni, L.~Erd{\H{o}}s, and D.~Schr{\"o}der.
\newblock Normal fluctuation in quantum ergodicity for {W}igner matrices.
\newblock {\em The Annals of Probability}, 50(3):984--1012, 2022.

\bibitem{cipolloni2023mesoscopic}
G.~Cipolloni, L.~Erd{\H{o}}s, and D.~Schr{\"o}der.
\newblock Mesoscopic central limit theorem for non-{H}ermitian random matrices.
\newblock {\em Probability Theory and Related Fields}, pages 1--52, 2023.

\bibitem{dubova2024bulk}
S.~Dubova and K.~Yang.
\newblock Bulk universality for complex eigenvalues of real non-symmetric
  random matrices with iid entries.
\newblock {\em arXiv preprint arXiv:2402.10197}, 2024.

\bibitem{EJ23}
L.~Erd{\H{o}}s and H.~C. Ji.
\newblock Wegner estimate and upper bound on the eigenvalue condition number of
  non-{H}ermitian random matrices.
\newblock {\em arXiv preprint arXiv:2301.04981}, 2023.

\bibitem{erdHos2010bulk}
L.~Erd{\H{o}}s, S.~P{\'e}ch{\'e}, J.~A. Ram{\'\i}rez, B.~Schlein, and H.-T.
  Yau.
\newblock Bulk universality for {W}igner matrices.
\newblock {\em Communications on Pure and Applied Mathematics: A Journal Issued
  by the Courant Institute of Mathematical Sciences}, 63(7):895--925, 2010.

\bibitem{EYY12}
L.~Erd{\H{o}}s, H.-T. Yau, and J.~Yin.
\newblock Rigidity of eigenvalues of generalized {W}igner matrices,.
\newblock {\em Advances in Mathematics}, 229(3):15--81, 1435--1515.

\bibitem{EYY11}
L.~Erd{\H{o}}s, H.-T. Yau, and J.~Yin.
\newblock Universality for generalized {W}igner matrices with {B}ernoulli
  distribution.
\newblock {\em Journal of Combinatorics}, 2(1):15--81, 2011.

\bibitem{girko1984}
V.~L. Girko.
\newblock The circular law.
\newblock {\em Teoriya Veroyatnostei i ee Primeneniya}, 29(4):669--679, 1984.

\bibitem{grela2018full}
J.~Grela and P.~Warcho{\l}.
\newblock Full {D}ysonian dynamics of the complex {G}inibre ensemble.
\newblock {\em Journal of Physics A: Mathematical and Theoretical},
  51(42):425203, 2018.

\bibitem{KY13}
A.~Knowles and J.~Yin.
\newblock Eigenvector distribution of {W}igner matrices.
\newblock {\em Probability Theory and Related Fields}, 155:543--582, 2013.

\bibitem{KY16}
A.~Knowles and J.~Yin.
\newblock Anisotropic local laws for random matrices.
\newblock {\em Probability Theory and Related Fields}, 169:257--352, 2017.

\bibitem{luh2020eigenvector}
K.~Luh and S.~O'Rourke.
\newblock Eigenvector delocalization for non-{H}ermitian random matrices and
  applications.
\newblock {\em Random Structures \& Algorithms}, 57(1):169--210, 2020.

\bibitem{lytova2020delocalization}
A.~Lytova and K.~Tikhomirov.
\newblock On delocalization of eigenvectors of random non-{H}ermitian matrices.
\newblock {\em Probability Theory and Related Fields}, 177(1):465--524, 2020.

\bibitem{MO}
A.~Maltsev and M.~Osman.
\newblock Bulk universality for complex non-{H}ermitian matrices with
  independent and identically distributed entries.
\newblock {\em arXiv preprint arXiv:2310.11429}, 2023.

\bibitem{marcinek2022high}
J.~Marcinek and H.-T. Yau.
\newblock High dimensional normality of noisy eigenvectors.
\newblock {\em Communications in Mathematical Physics}, 395(3):1007--1096,
  2022.

\bibitem{M2}
M.~Osman.
\newblock Bulk universality for real matrices with independent and identically
  distributed entries.
\newblock {\em arXiv preprint arXiv:2402.04071}, 2024.

\bibitem{rudelson2015delocalization}
M.~Rudelson and R.~Vershynin.
\newblock Delocalization of eigenvectors of random matrices with independent
  entries.
\newblock {\em Duke Mathematical Journal}, 164(13):2507--2538, 2015.

\bibitem{rudelson2016no}
M.~Rudelson and R.~Vershynin.
\newblock No-gaps delocalization for general random matrices.
\newblock {\em Geometric and Functional Analysis}, 26(6):1716--1776, 2016.

\bibitem{RS94}
Z.~Rudnick and P.~Sarnak.
\newblock The behaviour of eigenstates of arithmetic hyperbolic manifolds.
\newblock {\em Communications in Mathematical Physics}, 161(1):195--213, 1994.

\end{thebibliography}
\end{document}